\documentclass[a4paper, 10pt]{amsart}

\pdfoutput=1

\usepackage{amsmath,amssymb,amsthm,url}
\usepackage[utf8]{inputenc}
\usepackage[T1]{fontenc}
\usepackage{libertine}
\usepackage[libertine,cmintegrals,cmbraces]{newtxmath}
\usepackage[shortlabels]{enumitem}
\usepackage{mathtools}
\usepackage{xcolor}
\usepackage{tikz}  
\usepackage{tikz-cd}
\usepackage{nicefrac}
\usepackage{dsfont}
\usepackage{todonotes}
\usepackage{a4wide}
\usepackage{xcolor}
\usepackage{bbold}
\usepackage{spectralsequences}
\usepackage{quiver}
\usepackage{booktabs}
\usepackage{adjustbox}

\usepackage[
  style=alphabetic,
  backend=biber,
  sorting=nyt,
  isbn=false,
  maxbibnames=99
]{biblatex}

\AtEveryBibitem{\clearlist{language}}
\renewbibmacro{in:}{}

\AtBeginDocument{%
   \def\MR#1{}
} 
\usepackage{euscript}
\usepackage[all]{xy}

\usepackage{thmtools}
\usepackage{hyperref}
  
\hypersetup{%
  bookmarksnumbered=true,
  colorlinks=true,%
  linkcolor=blue,%
  citecolor=blue,%
  filecolor=blue,%
  menucolor=blue,%
  urlcolor=blue,%
  pdfnewwindow=true,%
  pdfstartview=FitBH}

\usepackage[capitalise]{cleveref}

\definecolor{indiguy}{RGB}{195,179,227}

\usepackage{xcolor}
\definecolor{seagreen}{RGB}{46,139,87}
\definecolor{maroon}{RGB}{128,0,0}
\definecolor{darkviolet}{RGB}{148,0,211}
\definecolor{twelve}{RGB}{100,100,170}
\definecolor{thirteen}{RGB}{100,150,50}
\definecolor{fourteen}{RGB}{200,0,0}
\definecolor{fifteen}{RGB}{0,200,0}
\definecolor{sixteen}{RGB}{0,0,200}
\definecolor{seventeen}{RGB}{200,0,200}
\definecolor{eighteen}{RGB}{0,200,200}

\newcommand{\bb}[1]{\mathbb{#1}}

\newcommand{\es}[1]{\EuScript{#1}}
\renewcommand{\sf}[1]{\mathsf{#1}}

\DeclareMathOperator{\coker}{\mathrm{coker}}

\DeclareMathOperator{\fibre}{\mathrm{fib}}

\DeclareMathOperator{\cofibre}{\mathrm{cofib}}

\newcommand{\fib}{\mathsf{fib}}

\newcommand{\s}{{\sf{Sp}}}
\DeclareMathOperator{\T}{\es{S}_\ast}

\newcommand{\vect}[1]{\mathsf{Vect}_{#1}}

\newcommand{\Aut}{\mathrm{Aut}}
\newcommand{\Fun}{\sf{Fun}}

\DeclareMathOperator{\Map}{\mathsf{Map}}

\DeclareMathOperator{\id}{\mathrm{Id}}

\DeclareMathOperator{\res}{\mathsf{res}}

\newcommand{\R}{\mbf{R}}

\newcommand{\del}{\partial}
\newcommand{\BU}{\mathsf{BU}}
\newcommand{\BO}{\mathsf{BO}}

\newcommand{\BAut}{\mathsf{BAut}}
\newcommand{\Bu}{\mathsf{Bu}}
\newcommand{\Bo}{\mathsf{Bo}}
\newcommand{\Ba}{\mathsf{Ba}}
\newcommand{\Bg}{\mathsf{Bg}}
\newcommand{\C}{\mathbb{C}}
\renewcommand{\R}{\mathbb{R}}

\renewcommand{\coker}{\mathsf{coker}}
\DeclareRobustCommand\longtwoheadrightarrow
     {\relbar\joinrel\twoheadrightarrow}

  \newcommand{\adjunction}[4]{
\xymatrix{
#1:#2 \ar@<.5ex>[r] &
\ar@<.5ex>[l] #3:#4
}}

\newtheorem{thm}{Theorem}[section]
\newtheorem{prop}[thm]{Proposition}
\newtheorem{lem}[thm]{Lemma}
\newtheorem{cor}[thm]{Corollary}

\newtheorem*{thm*}{Theorem}

\newtheorem{xxthm}{Theorem}

\theoremstyle{definition}
\newtheorem{definition}[thm]{Definition}

\newtheorem{rem}[thm]{Remark}

\begin{document}

\title{Realification of stably trivial vector bundles}
\author{Guy Boyde}
\address[Boyde]{Vrije Universiteit Amsterdam}
\email{g.boyde@vu.nl}
\author{Niall Taggart}
\address[Taggart]{Queen's University Belfast}
\email{n.taggart@qub.ac.uk}

\begin{abstract}
    The set of stably trivial complex vector bundles over complex projective spaces and spheres has a natural group structure when the corank is small enough. With respect to this group structure, the operations of taking the underlying real vector bundle (\emph{realification}) and of adding a trivial line bundle (\emph{stabilisation}), are group homomorphisms. Building on Hu's recent enumerations of stably trivial complex bundles, we compute these homomorphisms in a range by using Weiss calculus to translate the problem to stable homotopy theory.
\end{abstract}

\maketitle

\setcounter{tocdepth}{1}
{\hypersetup{linkcolor=black} \tableofcontents}

\section{Introduction}
Let $\vect{\bb{k}, d}(X)$ denote the set of isomorphism classes of topological $\bb{k}$-vector bundles of rank $d$ over a space $X$, where $\bb{k}$ denotes either the real or complex numbers. Whitney sum with a trivial line bundle gives a \emph{stabilisation} map:
\[
\vect{\bb{k}, d}(X) \xrightarrow{(-)\oplus~\underline{\bb{k}}} \vect{\bb{k}, d+1}(X)
\]
which is a bijection for $d$ large enough relative to the dimension on $X$, see~\cref{table:ranges} for precise conditions. In this \emph{stable range}, the set of vector bundles may be identified with topological $K$-theory:
\[
\vect{\bb{C}, d}(X) = \widetilde{KU}^0(X), \quad \text{and} \quad \vect{\bb{R}, d}(X) = \widetilde{KO}^0(X).
\]
In particular, to each vector bundle we can associate a $K$-theory class: its \emph{stabilisation}.
We say that a bundle is \emph{stably trivial} if its stabilisation is zero. When $\bb{k}=\bb{C}$ and $X = \bb{CP}^\ell$ is complex projective space, then stable triviality is equivalent to all Chern classes vanishing (see \cite[Theorem 2.1]{Hu}), but this is not true for all spaces $X$.

We write $\vect{\bb{k}, d}^0(X)$ for the set of \emph{stably trivial} rank $d$ $\bb{k}$-vector bundles over $X$. Said more geometrically:
\[
\vect{\bb{k}, d}^0(X) = \{ E \to X \ | \ E \oplus \underline{\bb{k}}^s \cong \underline{\bb{k}}^{d+s} \textrm{ for some } s \gg 0\}.
\]
In general $\vect{\bb{k}, d}^0(X)$ is just a set, however, it happens that in a so-called \emph{metastable} dimensional range, $\vect{\bb{k}, d}^0(X)$ has a natural abelian group structure, see \cref{thm: stably trivial stunted projective} and \cref{table:ranges} for precise conditions. This group structure comes from (stable) homotopy theory and does not seem to be geometrically meaningful in general, but for example if $X$ is a sphere then it coincides with the addition given by clutching. In particular, $\vect{\bb{k}, d}^0(X)$ is always non-empty, since it contains a zero element: the trivial bundle $\underline{\bb{k}}^d$.

In \cite{Hu}, Hu considers complex bundles over complex projective space $\bb{CP}^\ell$. The stable range is $d \geq \ell$, so the first two dimensions in which $\vect{\bb{k}, d}^0(\bb{CP}^\ell)$ can be non-trivial are $d=\ell -1$ and $d = \ell -2$. For the former, Hu~\cite[Theorem 1.1]{Hu} proves that for $\ell \geq 3$, the group $\vect{\bb{C}, \ell - 1}^0(\bb{CP}^\ell)$ is cyclic of order 2 if $\ell$ is even and trivial if $\ell$ is odd, while for the latter, Hu~\cite[Theorem 1.2]{Hu} proves that for $\ell \geq 4$, the group $\vect{\bb{C}, \ell - 2}^0(\bb{CP}^\ell)$  is cyclic of order $\psi(\ell)$ depending on the value of $\ell$ mod 24 as follows:
\[
\begin{array}{|c|*{12}{c}|}
\hline
\ell \bmod 24 
& 0 & 1 & 2 & 3 & 4 & 5 & 6 & 7 & 8 & 9 & 10 & 11 \\
\hline
\psi(\ell) 
& 1 & 1 & 12 & 2 & 1 & 3 & 2 & 4 & 3 & 1 & 4 & 6 \\
\hline
\hline
\ell \bmod 24 
& 12 & 13 & 14 & 15 & 16 & 17 & 18 & 19 & 20 & 21 & 22 & 23 \\
\hline
\psi(\ell) 
& 1 & 1 & 6 & 4 & 1 & 3 & 4 & 2 & 3 & 1 & 2 & 12 \\
\hline
\end{array}
\]

Noting that $\vect{\bb{k}, d}^0(S^{2 \ell})$ acts on $\vect{\bb{k}, d}^0(\mathbb{CP}^\ell)$, Hu also calculates this action on the above groups, as well as showing that pullback along the inclusion $\bb{CP}^{\ell - 1} \to \bb{CP}^\ell$ induces the zero homomorphism \cite[Propositions 1.3 and 1.4]{Hu}. Opie \cite{Opie} has since generalised the case $d = \ell - 1$ to bundles with any prescribed Chern classes, and this program has been pursued much further by Chatham--Hu--Opie \cite{ChathamHuOpie}.

In this paper we take an orthogonal approach. A complex vector bundle $E$ has an underlying \emph{realification} $rE$: the real vector bundle obtained by forgetting the complex structure on $E$. Our main result describes the effect of realification and stabilisation on vector bundles over complex projective space in small corank, i.e., just outside the stable range. We state it after localisation at 2: for the reader unfamiliar with localisation of abelian groups, localisation at 2 is intuitively given by taking the subgroup comprised of the free and 2-power torsion summands, i.e.~by `discarding odd-primary information'.

\begin{xxthm} \label{thm: CPell diagram}
After localisation at 2, the action of realification and stabilisation on stably trivial vector bundles on complex projective space $\bb{CP}^\ell$ for $\ell \geq 5$ are as shown in~\cref{fig:CPell diagram} for small corank.
\begin{figure}[ht]
    \begin{center}
\begin{tikzpicture}[commutative diagrams/every diagram]
\matrix[matrix of math nodes, name=m, ampersand replacement=\&,
        row sep=1.5em, column sep=0.5cm,
        commutative diagrams/every cell]{
\&[-0.3cm]
\scalebox{0.75}{$\begin{cases}\bb{Z}/4 & \ell\equiv 2,7\ (8)\\ \bb{Z}/2 & \ell\equiv 3,6\ (8)\\ 0 & \text{otherwise}\end{cases}$}
\& \&
\scalebox{0.75}{$\begin{cases}\bb{Z}/2 & \ell\text{ odd}\\ 0 & \ell\text{ even}\end{cases}$}
\& \&
0 \\
\& \vect{\bb{C},\,\ell-2}^0 \& \& \vect{\bb{C},\,\ell-1}^0 \& \& \vect{\bb{C},\,\ell}^0 \\
\& \& \& \& \& \\[1.8em]
\& \vect{\bb{R},\,2\ell-4}^0 \&[0.8cm] \vect{\bb{R},\,2\ell-3}^0 \&[0.7cm] \vect{\bb{R},\,2\ell-2}^0 \& \vect{\bb{R},\,2\ell-1}^0 \& \vect{\bb{R},\,2\ell}^0 \\
\& \scalebox{0.75}{$\begin{cases}\bb{Z} & \ell\equiv 0,1\ (4)\\ \bb{Z} \oplus\bb{Z}/2 & \ell\equiv 2,3\ (4)\end{cases}$}
\& \scalebox{0.75}{$\begin{cases}0 & \ell\equiv 0\ (4)\\ \bb{Z}/2 & \ell\equiv 1,2\ (4)\\ (\bb{Z}/2)^2 & \ell\equiv 3\ (4)\end{cases}$}
\& \scalebox{0.75}{$\begin{cases}\bb{Z} \oplus(\bb{Z}/2)^2 & \ell\text{ odd}\\ \bb{Z} & \ell\text{ even}\end{cases}$}
\& \scalebox{0.75}{$\begin{cases}\bb{Z}/2 & \ell\text{ odd}\\ 0 & \ell\text{ even}\end{cases}$}
\& \bb{Z} \\};
\path[commutative diagrams/.cd, every arrow, every label]
(m-2-2) edge node {$0$} (m-2-4)
(m-2-4) edge node {$0$} (m-2-6)
(m-4-2) edge node {\scalebox{0.8}{$\begin{cases}\begin{psmallmatrix}1\end{psmallmatrix} & \ell\equiv 1\ (4)\\ \begin{psmallmatrix}0&1\end{psmallmatrix} & \ell\equiv 2\ (4)\\ \begin{psmallmatrix}0&1\\1&0\end{psmallmatrix} & \ell\equiv 3\ (8) \\ \begin{psmallmatrix}0&1\\0&0\end{psmallmatrix} & \ell\equiv 7\ (8)\end{cases}$}} (m-4-3)
(m-4-3) edge node {\scalebox{0.8}{$\begin{cases}\begin{psmallmatrix}0\\0\\1\end{psmallmatrix} & \ell\equiv 1\ (4)\\ \begin{psmallmatrix}0&0\\0&0\\0&1\end{psmallmatrix} & \ell\equiv 3\ (4)\end{cases}$}} (m-4-4)
(m-4-4) edge node {\scalebox{0.8}{$\begin{psmallmatrix}0&1\end{psmallmatrix}\quad \ell\text{ odd}$}} (m-4-5)
(m-4-5) edge node {$0$} (m-4-6)
(m-2-2) edge node[swap] {\scalebox{0.8}{$\begin{psmallmatrix}0\\1\end{psmallmatrix}\quad \ell\equiv 2,3\ (4)$}} (m-4-2)
(m-2-4) edge[commutative diagrams/hook] node {\scalebox{0.8}{$\begin{psmallmatrix}0\\1\\0\end{psmallmatrix}\quad \ell\text{ odd}$}} (m-4-4)
(m-2-6) edge node {$0$} (m-4-6)
(m-1-2) edge[commutative diagrams/equal] (m-2-2)
(m-1-4) edge[commutative diagrams/equal] (m-2-4)
(m-1-6) edge[commutative diagrams/equal] (m-2-6)
(m-4-2) edge[commutative diagrams/equal] (m-5-2)
(m-4-3) edge[commutative diagrams/equal] (m-5-3)
(m-4-4) edge[commutative diagrams/equal] (m-5-4)
(m-4-5) edge[commutative diagrams/equal] (m-5-5)
(m-4-6) edge[commutative diagrams/equal] (m-5-6);
\end{tikzpicture}
\end{center}
    \caption{The action of realification and stabilisation on stably trivial bundles over complex projective space for small corank after localisation at 2. For compactness we write $\vect{\bb{k}, d}^0 : =\vect{\bb{k}, d}^0(\mathbb{CP}^\ell)_{(2)}$, and use the convention that any map that is not written is zero.}
    \label{fig:CPell diagram}
\end{figure}
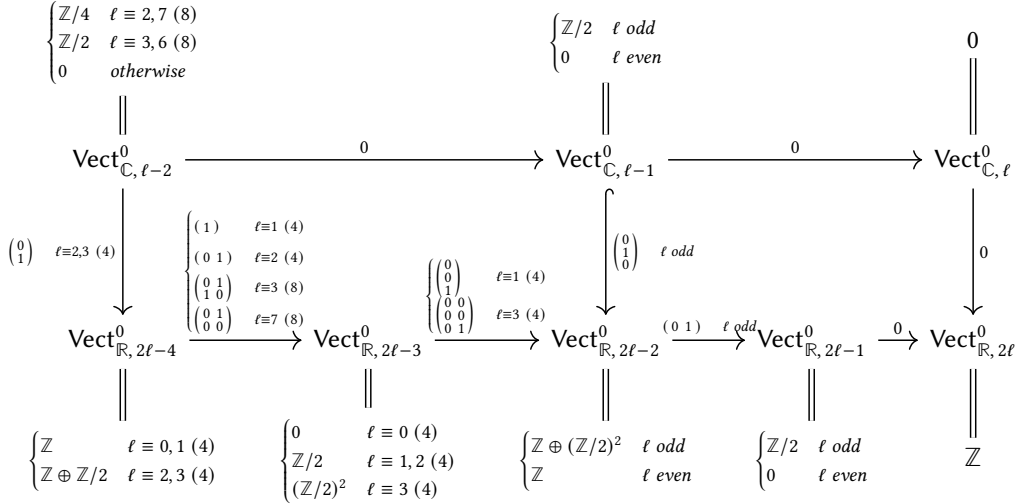
\end{xxthm}

Some words on interpreting the figure are in order: the top row consists of the groups computed by Hu (as \cite[Theorems 1.1 and 1.2]{Hu}), and the bottom row consists of the corresponding groups of real bundles. Maps are written as matrices, with stabilisation horizontally and realification vertically. Groups `to the right' of the figure are in the stable range, and hence automatically zero.

The main reason for stating the result after localisation at 2 is for simplicity and compactness (e.g.~in light of Hu's theorem, the group $\vect{\bb{C},  \ell - 2}^0$ is 24-periodic integrally, and only 8-periodic 2-locally). The odd-primary information is much sparser, and the determined reader can use the odd-primary results of \cref{section: odd primes} to recover integral statements.

This result has a number of down-to-earth consequences about vector bundles. An example of the sort of information which can be deduced from \cref{fig:CPell diagram} is as follows.
\begin{itemize}
    \item If $\ell \geq 5$ and $\ell \equiv 2,3 \pmod{4}$ then there is a unique non-trivial real rank $2 \ell - 4$ vector bundle over $\bb{CP}^\ell$ which admits a stably trivial complex structure. This real vector bundle is represented by the non-trivial element of the image of left-hand realification map. The horizontal maps tell us that $E \oplus \underline{\bb{R}}^2$ is trivial, but $E \oplus \underline{\bb{R}}$ is not. Note that the stable range begins in rank $2 \ell + 1$, so a priori we only knew that $E \oplus \underline{\bb{R}}^5$ is trivial. 
    \item In the remaining cases, i.e., if $\ell \geq 5$ and $\ell \equiv 0, 1 \pmod{4}$, then no such real bundle admits a stably trivial complex structure since the left-hand realification map is zero in these cases. This is not enough to conclude that no such real bundle admits a stably \emph{non}-trivial complex structure, however, because the realification map
    \[\widetilde{KU}^0(\bb{CP}^\ell) \xrightarrow{\ r \ }\widetilde{KO}^0(\bb{CP}^\ell)\] is not injective, see \cite[Theorem 3.9]{Sanderson}.
\end{itemize}

In order to prove \cref{thm: CPell diagram}, we (essentially) first do the corresponding calculations for bundles over spheres. These computations depend on the parity of the dimension and are as follows.

\begin{xxthm} \label{thm:sphere diagrams}
After localisation at 2, the action of realification and stabilisation on stably trivial vector bundles on spheres of dimension at least 16 are as shown in~\cref{fig:even sphere diagram} ($S^{2n}$, $n \geq 8$) and~\cref{fig:odd sphere diagram} ($S^{2n+1}$, $n \geq 8$), respectively. 
\begin{figure}[ht]
     \begin{center}
\begin{tikzpicture}[commutative diagrams/every diagram]
\matrix[matrix of math nodes, name=m, ampersand replacement=\&,
        row sep=1.5em, column sep=0.6cm,
        commutative diagrams/every cell]{
\&[-0.3cm]
\scalebox{0.75}{$\begin{cases}
  0        & n\equiv 1,5\ (8)\\
  \bb{Z}/2 & n\equiv 0,3,4\ (8)\\
  \bb{Z}/4 & n\equiv 6,7\ (8)\\
  \bb{Z}/8 & n\equiv 2\ (8)
\end{cases}$}
\& \&
\scalebox{0.75}{$\begin{cases}\bb{Z}/2 & n\text{ odd}\\ 0 & n\text{ even}\end{cases}$}
\& \&
0 \\
\& \vect{\bb{C},\,n-2}^0 \& \& \vect{\bb{C},\,n-1}^0 \& \& \vect{\bb{C},\,n}^0 \\
\& \& \& \& \& \\[1.8em]
\& \vect{\bb{R},\,2n-4}^0 \&[0.8cm] \vect{\bb{R},\,2n-3}^0 \&[1.4cm] \vect{\bb{R},\,2n-2}^0 \&[1.2cm] \vect{\bb{R},\,2n-1}^0 \& \vect{\bb{R},\,2n}^0 \\
\& \scalebox{0.75}{$\begin{cases}\bb{Z}/2 & n\not\equiv 1\ (4)\\ 0 & n\equiv 1\ (4)\end{cases}$}
\& \scalebox{0.75}{$\begin{cases}(\bb{Z}/2)^2 & n\equiv 3\ (4)\\ \bb{Z}/2 & \text{otherwise}\end{cases}$}
\& \scalebox{0.75}{$\begin{cases}(\bb{Z}/2)^2 & n\text{ odd}\\ 0 & n\text{ even}\end{cases}$}
\& \scalebox{0.75}{$\begin{cases}\bb{Z}/2 & n\text{ odd}\\ 0 & n\text{ even}\end{cases}$}
\& \bb{Z} \\};
\path[commutative diagrams/.cd, every arrow, every label]
(m-2-2) edge node {$0$} (m-2-4)
(m-2-4) edge node {$0$} (m-2-6)
(m-4-2) edge node {\scalebox{0.8}{$\begin{cases}
  \begin{psmallmatrix}1\\0\end{psmallmatrix} & n\equiv 3\ (4)\\
  \begin{psmallmatrix}1\end{psmallmatrix} & n\text{ even}
\end{cases}$}} (m-4-3)
(m-4-3) edge node {\scalebox{0.8}{$\begin{cases}
  \begin{psmallmatrix}0&0\\0&1\end{psmallmatrix} & n\equiv 3\ (4)\\
  \begin{psmallmatrix}0\\1\end{psmallmatrix} & n\equiv 1\ (4)
\end{cases}$}} (m-4-4)
(m-4-4) edge node {\scalebox{0.8}{$\substack{\begin{psmallmatrix}1&0\end{psmallmatrix}\\[2pt]n\text{ odd}}$}} (m-4-5)
(m-4-5) edge node {$0$} (m-4-6)
(m-2-2) edge[commutative diagrams/two heads] (m-4-2)
(m-2-4) edge[commutative diagrams/hook] node {\scalebox{0.8}{$\begin{psmallmatrix}1\\0\end{psmallmatrix}\quad n\text{ odd}$}} (m-4-4)
(m-2-6) edge node {$0$} (m-4-6)
(m-1-2) edge[commutative diagrams/equal] (m-2-2)
(m-1-4) edge[commutative diagrams/equal] (m-2-4)
(m-1-6) edge[commutative diagrams/equal] (m-2-6)
(m-4-2) edge[commutative diagrams/equal] (m-5-2)
(m-4-3) edge[commutative diagrams/equal] (m-5-3)
(m-4-4) edge[commutative diagrams/equal] (m-5-4)
(m-4-5) edge[commutative diagrams/equal] (m-5-5)
(m-4-6) edge[commutative diagrams/equal] (m-5-6);
\end{tikzpicture}
\end{center}
    \caption{The action of realification and stabilisation on stably trivial bundles over even spheres $S^{2n}$ for $n \geq 8$ after localisation at 2. For compactness we write $\vect{\bb{k}, d}^0 : =\vect{\bb{k}, d}^0(S^{2n})_{(2)}$, and use the convention that any map that is not written is zero.}
    \label{fig:even sphere diagram}
\end{figure}
\begin{figure}[ht]
    \begin{center}
\begin{tikzpicture}[commutative diagrams/every diagram]
\matrix[matrix of math nodes, name=m, ampersand replacement=\&,
        row sep=1.5em, column sep=0.6cm,
        commutative diagrams/every cell]{
\&[-0.3cm]
\scalebox{0.75}{$\begin{cases}\bb{Z}/n!\oplus\bb{Z}/2 & n\text{ odd}\\ \bb{Z}/\tfrac{1}{2}n! & n\text{ even}\end{cases}$}
\& \&
\scalebox{0.75}{$\bb{Z}/n!$}
\& \&
\scalebox{0.75}{$0$} \\
\& \vect{\bb{C},\,n-1}^0 \& \& \vect{\bb{C},\,n}^0 \& \& \vect{\bb{C},\,n+1}^0 \\
\& \& \& \& \& \\[1.8em]
\& \vect{\bb{R},\,2n-2}^0 \&[1.7cm] \vect{\bb{R},\,2n-1}^0 \&[1.3cm] \vect{\bb{R},\,2n}^0 \&[1.3cm] \vect{\bb{R},\,2n+1}^0 \&[0.8cm] \vect{\bb{R},\,2n+2}^0 \\[2.2em]
\& \scalebox{0.75}{$\begin{cases}(\bb{Z}/8)^2 & n\equiv 1\ (4)\\ \bb{Z}/4 & n\equiv 0,2\ (4)\\ \bb{Z}/16\oplus\bb{Z}/4 & n\equiv 3\ (4)\end{cases}$}
\& \scalebox{0.75}{$\begin{cases}\bb{Z}/8 & n\text{ odd}\\ \bb{Z}/2 & n\text{ even}\end{cases}$}
\& \scalebox{0.75}{$\begin{cases}\bb{Z}/4 & n\text{ odd}\\ (\bb{Z}/2)^2 & n\text{ even}\end{cases}$}
\& \scalebox{0.75}{$\bb{Z}/2$}
\& \scalebox{0.75}{$0$} \\};
\path[commutative diagrams/.cd, every arrow, every label]
(m-2-2) edge node {\scalebox{0.8}{$\begin{cases}\begin{psmallmatrix}1&0\end{psmallmatrix} & n\text{ odd}\\ (2) & n\text{ even}\end{cases}$}} (m-2-4)
(m-2-4) edge node {$0$} (m-2-6)
(m-4-2) edge node[swap, yshift=-2.5mm] {\scalebox{0.8}{$\begin{cases}\begin{psmallmatrix}1&0\end{psmallmatrix} & n\equiv 1\ (4)\\ \begin{psmallmatrix}1&2\end{psmallmatrix} & n\equiv 3\ (4)\end{cases}$}} (m-4-3)
(m-4-3) edge node[swap, yshift=-2.5mm] {\scalebox{0.8}{$\begin{cases}(1) & n\text{ odd}\\ \begin{psmallmatrix}0\\1\end{psmallmatrix} & n\text{ even}\end{cases}$}} (m-4-4)
(m-4-4) edge node[swap, yshift=-2.5mm] {\scalebox{0.8}{$\begin{cases}(1) & n\text{ odd}\\ \begin{psmallmatrix}1&0\end{psmallmatrix} & n\text{ even}\end{cases}$}} (m-4-5)
(m-4-5) edge node[swap] {$0$} (m-4-6)
(m-2-2) edge
  node[swap] {\scalebox{0.8}{$\begin{cases}\begin{psmallmatrix}1&4\\0&0\end{psmallmatrix} & n\equiv 1\ (4)\\ \begin{psmallmatrix}1&0\\0&2\end{psmallmatrix} & n\equiv 3\ (4)\\ (1) & n\equiv 0\ (4)\\ (2) & n\equiv 6\ (8)\end{cases}$}} (m-4-2)
(m-2-4) edge node {\scalebox{0.8}{$\begin{cases}(1) & n\text{ odd}\\ \begin{psmallmatrix}1\\0\end{psmallmatrix} & n\text{ even}\end{cases}$}} (m-4-4)
(m-2-6) edge node {$0$} (m-4-6)
(m-1-2) edge[commutative diagrams/equal] (m-2-2)
(m-1-4) edge[commutative diagrams/equal] (m-2-4)
(m-1-6) edge[commutative diagrams/equal] (m-2-6)
(m-4-2) edge[commutative diagrams/equal] (m-5-2)
(m-4-3) edge[commutative diagrams/equal] (m-5-3)
(m-4-4) edge[commutative diagrams/equal] (m-5-4)
(m-4-5) edge[commutative diagrams/equal] (m-5-5)
(m-4-6) edge[commutative diagrams/equal] (m-5-6);
\end{tikzpicture}
\end{center}
    \caption{The action of realification and stabilisation on stably trivial bundles over odd spheres $S^{2n+1}$ for $n \geq 8$ after localisation at 2. For compactness we write $\vect{\bb{k}, d}^0 : =\vect{\bb{k}, d}^0(S^{2n+1})_{(2)}$, and use the convention that any map that is not written is zero.}
    \label{fig:odd sphere diagram}
\end{figure}
\end{xxthm}

As with above, we may make several deductions about vector bundles over spheres from \cref{fig:even sphere diagram,fig:odd sphere diagram}. Note first that (since we know for $m \geq 2d$ that $\pi_m(U(d))$ is torsion, and $\pi_m(U)$ is torsion-free) all complex bundles over spheres in our dimensional range are stably trivial. In particular, if a bundle admits a complex structure then that complex bundle is stably trivial. One can then see in \cref{fig:odd sphere diagram} that if $n$ is congruent to 3 modulo 4, and large enough, then $S^{2n+1}$ has 64 stably trivial real bundles of corank $3$, 32 of which admit complex structures. Adding copies of $\underline{\bb{R}}$ cuts this group successively down to a $\bb{Z}/8$, then a $\bb{Z}/4$, then a $\bb{Z}/2$. If additionally $n$ is congruent to 3 mod 8 then Adams \cite{AdamsVF} tells us that the tangent bundle $TS^{2n+1}$ admits 7, but not 8, linearly independent vector fields. In particular, $TS^{2n+1}$ must be the generator of this last $\bb{Z}/2$, and there are (a different) 32 possibilities for the bundle which remains after removing three of these linearly independent directions.

\subsection*{Methods}
This paper has two main parts. The first (\cref{part: calculus}) is concerned with producing representing objects for stably trivial vector bundles and phrasing realification in terms of these representing objects. In the complex case this was done by Hu~\cite{Hu}, and the two main ideas are as follows.

\subsubsection*{Idea 1: translating the problem to homotopy theory.} 
The data of an isomorphism class of a complex vector bundle on a space $X$ is equivalent to an element of the set $[X,\BU(d)]$ of homotopy classes of maps from $X$ to $\BU(d)$. Now consider the fibre sequence 
\[
U/U(d) \longrightarrow \BU(d) \longrightarrow \BU
\]
obtained by delooping the fibre sequence associated to the homogeneous space $U/U(d)$. Under the above identification, the the induced map $[X,\BU(d)] \to [X,\BU]$ corresponds to stabilisation and hence $[X,U/U(d)]$ corresponds to the set of rank $d$ complex bundles $E$ \emph{equipped with a stable trivialisation}. Forgetting the choice of trivialisation gives a surjection $[X, U/U(d)] \twoheadrightarrow \vect{\bb{C}, d}^0(X)$ with kernel the image of the connecting map
\[
[X,U] \xrightarrow{(\delta_{\bb{C}})_*} [X,U/U(d)].
\]
Analogous reasoning applies in the real case, and this translates the problem of calculating $\vect{\bb{k}, d}^0(X)$ to (unstable) homotopy theory.

\subsubsection*{Idea 2: making the homotopy theory problem computable.} The infinite-dimensional Stiefel manifolds $O/O(d)$ and $U/U(d)$ have rather complicated homotopy theory, so we attempt to approximate $O/O(d)$ by a simpler space. Thinking of the projective space $\bb{RP}^{d-1}$ as a subspace of $O(d)$ - by sending a line $L$ to the automorphism flipping $L$ and fixing the orthogonal complement - gives a map
\[
\bb{RP}^\infty_d :=  \bb{RP}^\infty/\bb{RP}^{d-1} \longrightarrow O/O(d)
\]
which is known to be a good approximation. For example, Whitehead \cite{WhiteheadSphere} showed that $\bb{RP}^{\infty}_d$ is the $(2d-1)$-skeleton of $O/O(d)$ (c.f.~\cref{table:ranges}), and Miller \cite{Miller} showed that this inclusion gives the first summand of a stable splitting of $O/O(d)$ (proving also an analogous splitting of $U/U(d)$).

\emph{Weiss calculus} \cite{Weiss} gives an alternative way to approximate these infinite-dimensional Stiefel manifolds. Taking the real case as an example, the space $\BO(d)$ may be seen as the value at $\R^d$ of the functor $\Bo: \vect{\bb{R}} \to \T$ sending an inner product space $V$ to the classifying space $\Bo(V) = \BAut(V)$ of its automorphism group. In this framework, the map $\BO(d) \to \BO$ is (the value at $\R^d$) of the $0$-th polynomial approximation of $\Bo$, and the functor 
\[
O/\sf{o} : \vect{\bb{R}} \longrightarrow \T, \quad V \longmapsto O/O(V)
\]
may be through of as throwing away the $0$-polynomial part of $\Bo$. The linear approximation
\[
O/\sf{o} \longrightarrow P_1(O/\sf{o})
\]
of $O/\sf{o}$ is a good approximation of $O/O(d)$ and it follows from routine computations of Weiss that this linear approximation map is given by
\[
O/O(d) \longrightarrow \Omega^\infty\Sigma^\infty \bb{RP}^\infty_d.
\]
The direction of the approximation has changed, but the spirit remains the same, and it is almost certain that these approximations are related, a hint of which appears in~\cite[Example 10.2]{Weiss} and the role of the Miller splitting in Arone's~\cite{Arone} computation of the calculus of $\Bo$. Roughly speaking, the linear approximation supersedes the classical one, where crucially it comes for free as the infinite loop space of a spectrum with a simple cell structure. 
There is an analogous story in the complex case, which allows for the following description of stably trivial vector bundles and the behaviour of realification (see \cref{thm: realification vbs calc}) and stabilisation (see~\cref{thm: truncation is stabilisation}) in terms of stable homotopy theory.

\begin{xxthm}\label{main theorem: calc}
Let $X$ be a finite cell complex of dimension at least $8=2\cdot4$. In the metastable range
\begin{enumerate}
    \item stabilisation $\underline{\bb{k}} \oplus (-):{\vect{\bb{k},d}^0(X)} \to {\vect{\bb{k},d+1}^0(X)}$ is controlled by a commutative diagram
\[\begin{tikzcd}[ampersand replacement=\&,cramped]
	{[\Sigma X, \BAut] } \& {[\Sigma^\infty X,\Sigma^{\infty + \dim_\R(\bb{k})-1}\bb{kP}^\infty_{d}]} \& {\vect{\bb{k},d}^0(X)} \& 0 \\
	{[\Sigma X, \BAut] } \& {[\Sigma^\infty X,\Sigma^{\infty + \dim_\R(\bb{k})-1}\bb{kP}^\infty_{d+1}]} \& {\vect{\bb{k},d+1}^0(X)} \& 0
	\arrow["{\delta_\bb{k}}", from=1-1, to=1-2]
	\arrow[equals, from=1-1, to=2-1]
	\arrow[from=1-2, to=1-3]
	\arrow["{s_\ast}", from=1-2, to=2-2]
	\arrow["{\underline{\bb{k}} \oplus (-)}", from=1-3, to=2-3]
	\arrow["{\delta_\bb{k}}"', from=2-1, to=2-2]
	\arrow[from=2-2, to=2-3]
    \arrow[from=1-3, to=1-4]
    \arrow[from=2-3, to=2-4]
\end{tikzcd}\]
with exact rows, in which the maps $\delta_\bb{k}$ arise as connecting maps from the long exact sequence associated to the first layer of the corresponding Weiss tower, the middle vertical map is given by post-composition with the truncation map $t: \bb{kP}^\infty_{d} \to \bb{kP}^\infty_{d+1}$ which on $\bb{F}_2$-cohomology
    \[t^*:
    H^\ast(\bb{kP}^\infty_{d+1}; \bb{F}_2) \hookrightarrow H^\ast(\bb{kP}^\infty_{d}; \bb{F}_2)
    \]
    is an injection.
    \item realification $r: {\vect{\C,d}^0(X)} \to {\vect{\R,2d}^0(X)}$ is controlled by a commutative diagram
\[\begin{tikzcd}[ampersand replacement=\&,cramped]
	{[\Sigma X, \BU]} \& {[\Sigma^\infty X, \Sigma^{\infty +1}\bb{CP}^\infty_d]} \& {\vect{\C,d}^0(X)} \& 0\\
	{[\Sigma X, \BO]} \& {[\Sigma^\infty X, \Sigma^{\infty}\bb{RP}^\infty_{2d}]} \& {\vect{\R,2d}^0(X)} \& 0
	\arrow["{\delta_\C}", from=1-1, to=1-2]
	\arrow["{r_\ast}"', from=1-1, to=2-1]
	\arrow[from=1-2, to=1-3]
	\arrow["{r_\ast}"', from=1-2, to=2-2]
	\arrow["r"', from=1-3, to=2-3]
	\arrow["{\delta_\R}"', from=2-1, to=2-2]
	\arrow[from=2-2, to=2-3]
    \arrow[from=2-3, to=2-4]
    \arrow[from=1-3, to=1-4]
\end{tikzcd}\]
with exact rows, in which the middle vertical map is induced by post-composition with a specific map $r: \Sigma^{\infty +1}\bb{CP}^\infty_d \to \Sigma^{\infty}\bb{RP}^\infty_{2d}$ which on $\bb{F}_2$-cohomology
\[
r^* : H^\ast(\bb{RP}^\infty_{2d}; \bb{F}_2) \longtwoheadrightarrow H^\ast(\Sigma\bb{CP}^\infty_{d}; \bb{F}_2)
\]
is a surjection, and the leftmost vertical map is the usual $K$-theory realification (see e.g.~\cite{AdamsVF}).
\end{enumerate}
\end{xxthm}

The proof of \cref{main theorem: calc} involves a rather intricate analysis of the Weiss towers of $\Bo$ and $\Bu$ together with how these towers interact under realification. In particular, we have to improve on the current literature in several ways. Firstly, it is not enough to know that these Weiss towers converge, but critical to know the speed and radius of convergence so that our metastable range is as sharp as possible. Secondly, our methods differ from~\cite{Hu} in that we treat both variants of Weiss calculus simultaneously. Using the sharp bounds on convergence, the statement for stabilisation follows from the spherical fibrations for the unitary and orthogonal groups.

Realification is more subtle since $\Bo$ and $\Bu$ live in different flavours of Weiss calculus. Realification on the level of vector spaces lands in the even-dimensional vector spaces, thus we may phrase realification as a natural transformation $\Bu \to r^\ast\Bo$ internal to unitary Weiss calculus, where $r: \vect{\bb{C}} \to \vect{\bb{R}}$ is the realification functor on the level of vector spaces. Analysis of the interactions of realification and Weiss towers, extending that of the second author~\cite{TaggartOCUC} provide all but the cohomological surjection in the second part of \cref{main theorem: calc}. For the cohomological surjection, a slightly more careful argument using the spherical fibrations is then required.

In \cref{part: calculus} we also include a number of \emph{digressions} which are of independent interest and exhibit the utility of Weiss calculus in solving related geometric problems. Our first digression (\cref{digression: spherical}) is a discussion of stably trivial spherical bundles and the Weiss calculus incarnation of the $J$-homomorphism, which the general machinery developed in \cref{part: calculus} allows us to do without having to compute the derivatives of the functor $\Bg$ which sends a real vector space $V$ to the classifying space of the topological monoid of homotopy automorphisms of the unit sphere in $V$. Our second digression (\cref{digression: honat}) is a calculation of the space of natural transformations between the linear approximations of $\Bu$ and $r^\ast\Bu$. We show by way of the Segal Conjecture that realification is not a unique such transformation. The third digression (\cref{dig: fib of real}) is a computation of the derivatives of the functor $\sf{o/u}: \C^d \mapsto = O(2d)/U(d)$ using the computations of Arone~\cite{Arone} for the functors $\Bo$ and $\Bu$ (and our sharp convergence bounds). This result can be used to study almost complex structures \cite{Harris,Massey}. Enumerations of almost complex structures when $X$ is a homotopy complex projective space have been done in small dimensions~\cite{Mills}. We are aware of work-in-progress of Frankland and Hu on extending these enumerations to higher dimensions. Our methods also apply to complexification, but for largely trivial reasons do not say anything interesting here. We include a summary for completeness, as our final digression (\cref{digression: complexification}).

In the second part of the paper (\cref{part:computations}), we use \cref{main theorem: calc} to make calculations, proving \cref{thm:sphere diagrams} (where $X = S^m$) and then \cref{thm: CPell diagram} (where $X = \bb{CP}^\ell$). The sets of bundles over a sphere equipped with a stable trivialisation may be identified
\begin{align*}
[\Sigma^\infty S^m, \Sigma^{\infty +1}\bb{CP}^\infty_d] &\cong \pi_m^s(\Sigma\bb{CP}^\infty_d) \\
[\Sigma^\infty S^m, \Sigma^{\infty}\bb{RP}^\infty_d] &\cong \pi_m^s(\bb{RP}^\infty_d)
\end{align*}
with the stable homotopy groups of stunted projective spaces. To compute these, we use the Adams spectral sequence: the Steenrod algebra structure on the cohomology of stunted projective spaces are well understood so Bruner's Ext software \cite{Bruner} gives us the relevant $E_2$-pages, and we make ad-hoc arguments to determine the possible differentials. In the complex case, our methods differ slightly from those of Hu~\cite{Hu} in that we are more dogmatic about using only stable information, in particular avoiding the use of unstable homotopy groups of unitary groups as input. 

Combining these calculations with knowledge of the effect of realification and truncation on cohomology (\cref{main theorem: calc}) allows us to deduce the effect of those maps on these stable homotopy groups, or equivalently on vector bundles equipped with a stable trivialisation. The upshot is an analogue of \cref{thm:sphere diagrams} for bundles with stable trivialisation, which is \cref{thm: sphere diagrams 2}. In \cref{section: B' implies B} we compute the connecting maps coming from $K$-theory in \cref{main theorem: calc} to deduce \cref{thm:sphere diagrams}. Here the calculation of the complex groups was not done in~\cite{Hu}, because these groups are not needed in the later calculations for complex projective spaces. This calculation can be deduced from what we have already done with the help of some classical ingredients, and it answers our main question for spheres, so we do it.

Given a space $X$ and a spectrum $E$, the Atiyah--Hirzebruch spectral sequence allows one to compute $[\Sigma^\infty X,E]$ from the stable homotopy groups of $E$, with differentials controlled by the cell structure of $X$. Taking $E$ to be the suspension spectrum of the stunted projective spaces in \cref{main theorem: calc} we can theoretically use our calculations for spheres to compute bundles equipped with a stable trivialisation over an arbitrary base. Complex projective spaces are especially viable for this due to their sparse, well-understood cell structure which has been extensively studied by Mosher~\cite{Mosher}. We make these calculations in \cref{section: B' implies A}, tracing the effect of the maps through the relevant spectral sequences.

\subsection*{Assumed knowledge} \cref{part: calculus} and \cref{part:computations} use quite different methods, and can be read more or less independently. \cref{part: calculus} consists of theoretical results about the Weiss Calculus (particularly of $\Bo$ and $\Bu$) which could not be found in the literature. We assume some familiarity with the general ideas of functor calculus but introduce the key aspects as and when they are needed. \cref{part:computations} consists of calculations in stable homotopy theory, using the Adams and Atiyah--Hirzebruch spectral sequence, and so we assume only a basic familiarity with stable homotopy theory.

\subsection*{Conventions}
\begin{itemize}
    \item In~\cref{part: calculus} we distinguish between stable and unstable information through the standard notation. 
    \item In~\cref{part:computations} we abuse notation and will write $X$ for $\Sigma^\infty X$ and $\pi_\ast(X)$ for $\pi_\ast(\Sigma^\infty X) = \pi_\ast^s(X)$.
    \item For us, a map is \emph{$k$-connected} if its homotopy fibre is $(k-1)$-connected.
    \item We write $[X,Y]=\pi_0(\Map_*(X,Y))$ for based homotopy classes of maps.
    \item It will at times be convenient to shorten $a \equiv b \pmod{m}$ to $a \equiv b (m)$.
    \item We will denote by $\bb{k}$ either of the fields $\R$ or $\C$ and by $V$ a finite-dimensional $\bb{k}$-vector space with an inner product. It should be clear from context whether we are consider $V$ to be real or complex.
\end{itemize}

\subsection*{Acknowledgements}
This paper benefited from the generosity of many people. GB would especially like to thank Christian Carrick for helping him to get started with stable homotopy theory and Bruner's Ext, and Thomas Rot for providing a really useful geometric point of view on lots of the phenomena encountered along the way. Lauran Toussaint kindly read a draft of the introduction, caught a couple of misprints, and suggested many tweaks which improved the clarity. The authors are very grateful to Robert Bruner for making his software \cite{Bruner} available, and especially for including the utilities \verb|makeP| and \verb|makeCP|, which for this project especially represented a near-infinite quality of life improvement. This work also greatly benefited from conversations with Greg Arone, Miguel Barata, Gil Cavalcanti, Emanuele Dotto, Gijs Heuts, John Jones, Marco Nervo, Sven van Nigtevecht, Álvaro del Pino Gómez, Oscar Randal-Williams, and Baylee Schutte. 

Both authors were supported by European Research Council (ERC) through the grant “Chromatic homotopy theory of spaces”, grant number 950048, and thank the Isaac Newton Institute, Cambridge, for support and hospitality during the programme Equivariant homotopy theory in context, where work on this paper was undertaken (EPSRC grant EP/Z000580/1). GB was supported by the NWO (Dutch Research Council), grant number VI.Veni.242.230. NT was supported by the NWO under Vidi grant number VI.Vidi.203.004 and by the EPSRC under grant number EP/Z534705/1.

\part{The Weiss calculus of vector bundles}\label{part: calculus}
In this part we use Weiss calculus to provide homotopical models for stabilisation and realification of stably trivial vector bundles. Metastably, we get stable representing objects and describe these operations on the level of stable representing objects. This part contains a number of `digressions' which do not impact the proofs of our main theorems but are closely related and are likely to be of independent interest.

\section{Stabilisation of vector bundles and Weiss calculus}\label{section: general calc}

There is a well-known isomorphism
\[
\vect{\R,d}(X)\cong [X, \BO(d)]
\]
between the set of rank $d$ real bundles over $X$ (up to isomorphism) and the set of homotopy classes of maps from $X$ to the classifying space $\BO(d)$ of the orthogonal group $O(d)$. Under this isomorphism, the map induced on classifying spaces by the standard inclusion $O(d) \hookrightarrow O(d+1)$ represents adding a trivial line bundle, i.e., stabilisation.

The space $\BO(d)$ may be viewed as the value at $\R^d$ of the continuous functor
\[
\Bo: \vect{\R} \longrightarrow \T, \ V \longmapsto \BO(V) = \BAut(V),
\]
from the category $\vect{\R}$ of finite-dimensional real inner products spaces to the category $\T$ of pointed spaces and the stabilisation map is the image of the standard inclusion $\bb{R}^d \hookrightarrow \bb{R}^{d+1}$ under the functor $\Bo$. Functors of this form are called \emph{orthogonal} and are the input to (orthogonal) Weiss calculus~\cite{Weiss}.

To an orthogonal functor $F: \vect{\bb{R}} \to \T$, Weiss calculus associates a tower 
\[\begin{tikzcd}
	&& \vdots \\
	&& {P_2F(V)} & {D_2F(V)} \\
	&& {P_1F(V)} & {D_1F(V)} \\
	{F(V)} && {P_0F(V)} & {F(\R^\infty)}
	\arrow[from=1-3, to=2-3]
	\arrow[from=2-3, to=3-3]
	\arrow[from=2-4, to=2-3]
	\arrow[from=3-3, to=4-3]
	\arrow[from=3-4, to=3-3]
	\arrow[curve={height=-12pt}, from=4-1, to=1-3]
	\arrow[curve={height=-12pt}, from=4-1, to=2-3]
	\arrow[curve={height=-12pt}, from=4-1, to=3-3]
	\arrow[from=4-1, to=4-3]
	\arrow[equals, from=4-3, to=4-4]
\end{tikzcd}\]
of orthogonal functors, in which the $n$-th \emph{layer} $D_nF$ is the fibre of the map $P_nF \to P_{n-1}F$ from the $n$-th \emph{stage} $P_{n}F$ to the $(n-1)$-st stage $P_{n-1}F$. Weiss showed that the layers of this tower are infinite loop spaces \cite[Theorem 7.3]{Weiss}, describing them as
\[
D_nF(V) \simeq \Omega^\infty \bb{D}_nF(V) ,
\]
for a functor $\bb{D}_nF: \vect{\bb{k}} \to \s$ which is given (up to equivalence) by 
\begin{equation}\label{eq: class. of homog}
 \bb{D}_nF(V) \simeq (\partial_nF \otimes S^{n \cdot V})_{hO(n)}   ,
\end{equation}
where $\partial_nF \in \s^{\BO(n)}$ is a (Borel) $O(n)$-spectrum, $S^{n \cdot V}$ is the one-point compactification of $n \cdot V = \R^n \otimes V$, and $O(n)$ acts on $\partial_nF \otimes S^{n \cdot V}$ diagonally. The Borel $O(n)$-spectrum $\partial_n F$ is called the \emph{$n$-th derivative} of $F$.

There is an analogous theory of \emph{unitary} Weiss calculus in which the role of the real line $\bb{R}$ is replaced by the complex plane $\bb{C}$ as considered by the second author~\cite{TaggartUC}. We will wish to treat both calculi in a uniform matter and so will refer to functors 
\[
F: \vect{\bb{k}} \longrightarrow \T,
\]
from the category $\vect{\bb{k}}$ of finite-dimensional inner product $\bb{k}$-vector spaces and linear isometries (for $\bb{k} = \bb{R}$ or $\bb{C}$) to the category of pointed spaces as \emph{Weiss functors}. For the same reason, we write
\[
\Ba : \vect{\bb{k}} \longrightarrow \T, V \longrightarrow \BAut(V),
\]
for the Weiss functor sending a vector space $V$ to the classifying space of the automorphisms of $V$, i.e.,
\[
\Ba = \begin{cases}
 \Bo &\text{if $\bb{k}=\bb{R}$} \\   
  \Bu &\text{if $\bb{k}=\bb{C}$.} \\  
\end{cases}
\]
We summarise the relevant information about the Weiss calculus of $\Ba$ in the following lemma. The description of the first derivative and first layer already appeared in~\cite{Weiss} while convergence was first stated in~\cite{Arone} but first proved in~\cite{BarnesEldred}. Much more is known about this functor, e.g., Arone~\cite{Arone} has computed the derivatives through an analogy with the Goodwillie tower of the identity. The main thread of our paper does not require the full computation of the derivatives but they will play a role in the digression of~\cref{dig: fib of real}. We reprove convergence since explicit connectivity estimates are vital to our work, but the literature is rather muddy on these.

\begin{lem}\label{lem: tower of BAut}
The Weiss tower of $\Ba$ is given by
\[\begin{tikzcd}
	&& \vdots \\
	&& {P_2\Ba(V)} & {D_2\Ba(V)} \\
	&& {P_1\Ba(V)} & {\Omega^\infty \Sigma^{\infty+\dim_\bb{R}(\bb{k})-1}\bb{kP}^\infty_{\dim_\bb{k}(V)}} \\
	{\Ba(V)} && \BAut
	\arrow[from=1-3, to=2-3]
	\arrow[from=2-3, to=3-3]
	\arrow[from=2-4, to=2-3]
	\arrow[from=3-3, to=4-3]
	\arrow[from=3-4, to=3-3]
	\arrow[curve={height=-12pt}, from=4-1, to=1-3]
	\arrow[curve={height=-12pt}, from=4-1, to=2-3]
	\arrow[curve={height=-12pt}, from=4-1, to=3-3]
	\arrow[from=4-1, to=4-3]
\end{tikzcd}\]
in low degrees, with the connectivity of the layers growing linearly in $n$, i.e.,
\[
\sf{Conn}(D_n\Ba(V)) = n(\dim_\bb{R}(V) + \dim_{\R}(\bb{k}) -2) = 
\begin{cases}
    n(\dim_\R(V)-1) & \bb{k} = \R \\
    2n\dim_\C(V) & \bb{k} = \C,
\end{cases}
\]
whenever $\dim_\bb{R}(V) \geq 2$. Also, the first derivative of $\Ba$ is the Borel $\Aut(1)$-spectrum $\partial_1\Ba \simeq \Sigma^{\dim_\bb{R}(\bb{k})-1} \bb{S}$ with trivial $\Aut(1)$-action.
\end{lem}

In the proof we will need the following lemma, which is immediate from the definitions see e.g., \cite[Theorem 6.3]{Weiss} or \cite[Definition 3.4]{TaggartUC} (caution: these references write $T_n$ rather than $P_n$).

\begin{lem}\label{lem: shift commutes with tower}
For $F: \vect{\bb{k}} \to \T$, there are natural equivalences 
\begin{align*}
    P_n(F(V \oplus -)) &\simeq P_nF(V \oplus -) \\
    D_n(F(V \oplus -)) &\simeq D_nF(V\oplus -) 
\end{align*}
In particular, the functor $V \oplus (-): \vect{\bb{k}} \to \vect{\bb{k}}$ commutes with Weiss towers. \qed
\end{lem}

\begin{proof}[Proof of~\cref{lem: tower of BAut}]
The computation of the first derivative is a routine: for $\bb{k} = \bb{R}$ it follows from~\cite[Example 2.7]{Weiss} and for $\bb{k}=\bb{C}$ from \cite[Example 10.2]{Weiss} but see also \cite[Example 4.6]{TaggartUC} for more detail. The calculation of the first layer is almost immediate from the classification \eqref{eq: class. of homog} and \cite[Proposition 4.3]{AtiyahThomComplexes}. It may be extracted from \cite[Example 10.1]{Weiss} for $\bb{k}=\R$ and \cite[Example 10.2]{Weiss} for $\bb{k}=\C$ after observing that the functors
\begin{align*}
O/\sf{o} : \vect{\bb{R}} &\longrightarrow \T, \quad V \longmapsto O/O(V) \\
U/\sf{u} : \vect{\bb{C}} &\longrightarrow \T, \quad V \longmapsto U/U(V)
\end{align*}
are reduced versions of the functors $\Bo$ and $\Bu$ so have the same homogeneous layers.

For the connectivity of the layers it suffices to show that  when $\dim_\R(V) \geq 3 - \dim_\R(\bb{k})$, the approximation map $\Ba(V) \to P_n\Ba(V)$ is $((n+1)(\dim_\bb{R}(V) + \dim_{\R}(\bb{k}) -2)+1)$-connected, because this implies that the map $P_n\Ba(V) \longrightarrow P_{n-1}\Ba(V)$ is $(n(\dim_\bb{R}(V) + \dim_{\R}(\bb{k}) -2)+1)$-connected, and hence that
\[
\sf{Conn}(D_n\Ba(V)) = n(\dim_\bb{R}(V) + \dim_{\R}(\bb{k}) -2). 
\]

There is a natural fibre sequence 
\[
\Sigma^{\dim_{\R}(\bb{k}) -1}S^{V} \longrightarrow \Ba(V) \longrightarrow \Ba(V \oplus \bb{k}),
\]
so by~\cite[Theorem 9.14]{TaggartUC} (or by an analogous direct argument using the Five Lemma and the fact that $P_n$ preserves fibre sequences of functors) it further suffices to obtain the same connectivity estimate for the map $\Sigma^{\dim_\R(\bb{k})-1}S^{V} \to P_n(\Sigma^{\dim_\R(\bb{k})-1}S^{(-)})(V)$.

For this, we take a detour through Goodwillie calculus. The identity functor $\id: \T \to \T$ is $1$-analytic \cite[Example 4.3]{GoodCalcIII} in the sense of \cite[Definition 4.2]{GoodCalcIII}, i.e., the identify functor converges in Goodwillie calculus on simply connected spaces. In particular, this means that the map
\[
X=\id(X) \longrightarrow P_n\id(X)
\]
is $((n+1)\sf{Conn}(X)+1)$-connected for all pointed spaces $X$ which have connectivity $\sf{Conn}(X) \geq 1$\footnote{Note that~\cite[Definition 1.2]{GoodCalcIII} is the relative version of this statement: Goodwillie works in the slice category $\es{S}_{/Y}$ for some space $Y$, and says that a natural transformation $F \to G$ satisfies $O_n(c,\kappa)$ is for all $k \geq \kappa$ and for all $X \to Y$ $k$-connected, the induced map $F(X) \to G(X)$ is $((n+1)k -c)$-connected. In our example, $Y = \ast$, so the connectivity of the map $X \to \ast$ is $\sf{Conn}(X)+1$, and the condition $k \geq \kappa$ translates to $\sf{Conn}(X) +1 \geq \rho +1$. The literature is, on occasion, \emph{lax} about this $+1$.}. Now consider the one-point compactification functor $S^{(-)} : \vect{\bb{k}} \to \T$. By \cite[Proposition 3.4]{BarnesEldred}, the natural map $P_n(S^{(-)}) \to P_n(\id) \circ S^{(-)}$ is an equivalence, so, setting $X=S^V$ above, the universal map $S^{V} \longrightarrow P_n(S^{(-)})(V)$ is $((n+1)(\dim_\R(V)-1) +1)$-connected whenever $\dim_\bb{R}(V) = \sf{Conn}(S^{V}) + 1 \geq 2$. Specialising to $\bb{k}=\R$ resolves the real case.

In the case that $\bb{k}=\C$, the fibre $\Sigma^{\dim_{\R}(\bb{k}) -1}S^{(-)} = \Sigma S^{(-)}$ is the composite
\[
\vect{\C} \xrightarrow{\ r\ } \vect{\R} \xrightarrow{S^{(-) \oplus \bb{R}}} \T.
\]
The second functor $S^{(-) \oplus \bb{R}}$ has $n$-th stage given by $P_n(S^{(-) \oplus \bb{R}})(V) \simeq P_n(S^{(-)})(V \oplus \bb{R})$ by \cref{lem: shift commutes with tower}, so by our earlier analysis of the functor $S^{(-)}$, the (orthogonal calculus) approximation map
\[
S^{V \oplus \bb{R}} \longrightarrow P_n(S^{(-)\oplus \bb{R}})(V)
\]
is $((n+1)\dim_\bb{R}(V)+1)$-connected whenever $\dim_\R(V) \geq 1$. By \cite[Proposition 5.4]{TaggartOCUC} realification commutes with polynomial approximation in the sense that $P_n(r^\ast F) \simeq r^\ast P_n F$ for a reduced weakly polynomial orthogonal functor $F$, of which $S^{(-)}$ is an example, so it follows immediately that the $n$-th (unitary) polynomial approximation to the whole composite: 
\[
\Sigma S^{(-)} \longrightarrow P_n(\Sigma S^{(-)})
\]
is $((n+1)\dim_\bb{R}(V)+1)$-connected at $V \in \vect{\bb{C}}$ whenever $\dim_\R(V) \geq 2$.
\end{proof}

For a Weiss functor $F$ and a pointed space $X$, there is an induced Weiss functor
\[
F^X : \vect{\bb{k}} \longrightarrow \T, \ V \longmapsto F^X(V) = \Map_\ast(X, F(V)).
\]
The functor $\Ba^X = \Map_\ast(X, \Ba)$ governs the behaviour of vector bundles over $X$: evaluating at $V$ and taking connected components yields the set of rank $d$ bundles over $X$ up to isomorphism. Under suitable hypothesis on $X$, the the first layer of $\Ba^X$ governs stably trivial bundles over $X$.

\begin{lem}\label{lem: tower of BAutX}
Let $X$ be a finite cell complex. The Weiss tower of $\Ba^X$ is given by
\[\begin{tikzcd}[ampersand replacement=\&,cramped]
	\&\& \vdots \& \\
	\&\& {\Map_\ast(X, P_2\Ba(V))} \& {\Map_\ast(X, D_2\Ba(V))} \\
	\&\& {\Map_\ast(X, P_1\Ba(V))} \& {\Map_\ast(X, D_1\Ba(V))} \\
	{\Ba^X(V)} \&\& {\Map_\ast(X, \BAut)}
	\arrow[from=1-3, to=2-3]
	\arrow[from=2-3, to=3-3]
	\arrow[from=2-4, to=2-3]
	\arrow[from=3-3, to=4-3]
	\arrow[from=3-4, to=3-3]
	\arrow[curve={height=-12pt}, from=4-1, to=1-3]
	\arrow[curve={height=-12pt}, from=4-1, to=2-3]
	\arrow[curve={height=-12pt}, from=4-1, to=3-3]
	\arrow[from=4-1, to=4-3]
\end{tikzcd}\]
with the connectivity of the layers growing linearly in $n$, i.e.,
\[
\sf{Conn}(D_n\Ba^X(V)) = n(\dim_\bb{R}(V) + \dim_{\R}(\bb{k}) -2) - \dim(X) = 
\begin{cases}
    n(\dim_\R(V)-1) -\dim(X) & \bb{k} = \R \\
    2n\dim_\C(V) - \dim(X)& \bb{k} = \C
\end{cases}
\]
whenever $\dim_\bb{R}(V) \geq 2$.
\end{lem}
\begin{proof}
The functor 
\[
\Map_\ast(X, -) : \T \longrightarrow \T,
\]
commutes with limits. Since $X$ it a finite cell complex it also commutes with filtered colimits. It follows that for any Weiss functor $F$, we have
\[
P_n (F^X)(V) \simeq \Map_\ast(X, P_nF(V)),
\]
since $P_nF$ may be constructed as a filtered colimit of compact limits, see e.g., \cite[Theorem 6.3]{Weiss} or \cite[Definition 3.4]{TaggartUC}. The description of the layers follows readily since $\Map_\ast(X, -)$ preserves fibre sequences. Taking $F = \Ba$ yields the description of the Weiss tower for $\Ba^X$. The connectivity estimate follows from the connectivity estimates of~\cref{lem: tower of BAut}.
\end{proof}

The stabilisation map
\[
\Ba(\bb{k}^d) \longrightarrow \Ba(\bb{k}^\infty) = \BAut,
\]
is $((d+1)\dim_\bb{R}(\bb{k})-1)$-connected, so if $\dim(X)<(d+1)\dim_\bb{R}(\bb{k})-1$, then rank $d$ vector bundles over $X$ are stable, i.e., the induced map
\[
[X, \BAut(d)] \longrightarrow [X, \BAut],
\]
is a bijection, and bundles in this \emph{stable range} are detected in (real or complex) topological $K$-theory. We will examine bundles that live just outside of this stable range. We call such bundles \emph{metastable}.

\begin{definition}
A $\bb{k}$-vector bundle of rank $d$ over a finite cell complex $X$ is  said to be \emph{in the metastable range} if 
\[
\frac{\dim(X)+2(2-\dim_\R(\bb{k}))}{2\dim_\bb{R}(\bb{k})} \leq d \leq \frac{\dim(X)-(\dim_\bb{R}(\bb{k}) -1)}{\dim_\bb{R}(\bb{k})}
\]
\end{definition}

The lower bound comes from the connectivity of the layers of the Weiss tower for $\Ba^X$ as given in \cref{lem: tower of BAutX}. This will become clear in the proof of~\cref{thm: stably trivial stunted projective}. In particular, we have the ranges shown in \cref{table:ranges}. With $\bb{k}=\bb{C}$ and $X = \bb{CP}^\ell$ (so $\dim(X)=2 \ell$), this recovers the range given in~\cite[Theorem 2.1]{Hu}.
\begin{table}[ht]
    \setlength{\tabcolsep}{6pt} 
    \renewcommand{\arraystretch}{2.2} 
     \centering
    \begin{tabular}{|c|c|c|}
    \hline
         $\bb{k}$ & $\bb{R}$ & $\bb{C}$  \\
         \hline 
        Stable & $ d\geq \dim(X)+1 $ & $\displaystyle d \geq \frac{\dim(X)}{2}$ \\
        \hline
            Metastable & $\displaystyle \frac{\dim(X)}{2} +1 \leq d < {\dim(X)}+1$ &  $\displaystyle \frac{\dim(X)}{4} \leq d < \frac{\dim(X)}{2}$ \\
        \hline
    \end{tabular}
    
    \caption{(Meta)Stable ranges for rank $d$ vector bundles over a finite cell complex $X$.}
    \label{table:ranges}
\end{table}

We now make precise that within the metastable range, Weiss calculus (namely, the first layer of $\Ba^X$) determines stably trivial bundles.
\begin{thm}\label{thm: stably trivial stunted projective}
Let $X$ be a finite cell complex of dimension at least $2\dim_\R(\bb{k})$. In the metastable range, there is an exact sequence
\[
 [\Sigma X, \BAut] \xrightarrow{ \delta_\bb{k}} [\Sigma^\infty X, \Sigma^\infty  \Sigma^{\dim_\bb{R}(\bb{k})-1}\bb{kP}^\infty_d] \to \vect{\bb{k}, d}^0(X) \to 0,
\]
expressing the set of stably trivial rank $d$ $\bb{k}$-vector bundles over $X$ as a quotient of the group of bundles equipped with a stable trivialisation. The map $\delta_{\bb{k}}$ is the connecting map in the defining fibre sequence for the first layer of $\Ba^X$. 
\end{thm}
\begin{proof}In terms of classifying spaces, to say that a bundle is stably trivial is to say that it is in the kernel
$$\vect{\bb{k}, d}^0(X) \cong \ker\left([X,\Ba(\bb{k}^d)] \longrightarrow [X,\BAut]\right)$$
and by \cref{lem: tower of BAutX}, this map is in fact the $0$-th Weiss approximation to the functor $\Ba^X$.

Consider the long exact sequence
\[
\cdots \longrightarrow [\Sigma X, \BAut] \xrightarrow{\delta_\bb{k}} [X, D_1\Ba(\bb{k}^d)] \longrightarrow [X, P_1\Ba(\bb{k}^d)] \longrightarrow [X, \BAut]
\]
associated to the fibration
\[
\Map_\ast(X,D_1\Ba(\bb{k}^d)) \longrightarrow \Map_\ast(X,P_1\Ba(\bb{k}^d)) \longrightarrow \Map_\ast(X,\BAut)
\]
defining the first layer of the Weiss tower of $\Ba^X$, see~\cref{lem: tower of BAutX}. By exactness, the kernel of 
\[
[X, P_1\Ba(\bb{k}^d)] \longrightarrow [X, \BAut]
\]
is the cokernel of $\delta_\bb{k}$. It therefore suffices to show that the map 
\[
\Map_\ast(X, \Ba(\bb{k}^d)) \longrightarrow \Map_\ast(X, P_1\Ba(\bb{k}^d)),
\]
induces an isomorphism on $\pi_0$, or equivalently, that the layers $D_n\Ba(\bb{k}^d)$ are connected for $n \geq 2$.

Since $X$ had dimension at least $2\dim_\R(\bb{k})$ and we are in the metastable range, we have that 
\[
2 \leq \frac{\dim(X)+2(2-\dim_\R(\bb{k}))}{2} \leq d \dim_\R(\bb{k}) = \dim_\R(\bb{k}^d),
\]
so by~\cref{lem: tower of BAutX} and our assumption that we are in the metastable range, for $n \geq 2$ we have
\begin{align*}
\sf{Conn}(D_n\Ba^X(\bb{k}^d))
&= n(\dim_\bb{R}(\bb{k}^d) + \dim_{\R}(\bb{k}) -2) - \dim(X) \\
&\geq 2(d\dim_\bb{R}(\bb{k}) + \dim_{\R}(\bb{k}) -2) - \dim(X) \geq 0. \qedhere
\end{align*}
\end{proof}

In~\cref{section:realification} we will see that realification of a (stably trivial) complex bundle may be described in calculus terms. We observe here that the same is true for stabilisation by a trivial line bundle.

\begin{thm}\label{thm: truncation is stabilisation}
Let $X$ be a finite cell complex of dimension at least $2\dim_\R(\bb{k})$. 
\begin{enumerate}
    \item In the metastable range, there is a commutative diagram with exact rows
\[\begin{tikzcd}[ampersand replacement=\&,cramped]
	{[\Sigma X, \BAut] } \& {[\Sigma^\infty X,\Sigma^{\infty + \dim_\R(\bb{k})-1}\bb{kP}^\infty_{d}]} \& {\vect{\bb{k},d}^0(X)} \& 0 \\
	{[\Sigma X, \BAut] } \& {[\Sigma^\infty X,\Sigma^{\infty + \dim_\R(\bb{k})-1}\bb{kP}^\infty_{d+1}]} \& {\vect{\bb{k},d+1}^0(X)} \& 0 
	\arrow["{\delta_\bb{k}}", from=1-1, to=1-2]
	\arrow[equals, from=1-1, to=2-1]
	\arrow[from=1-2, to=1-3]
	\arrow["{s_\ast}", from=1-2, to=2-2]
	\arrow["{\underline{\bb{k}} \oplus (-)}", from=1-3, to=2-3]
	\arrow["{\delta_\bb{k}}"', from=2-1, to=2-2]
	\arrow[from=2-2, to=2-3]
    \arrow[from=2-3, to=2-4]
    \arrow[from=1-3, to=1-4]
\end{tikzcd}\]
    where the right vertical map is given by Whitney sum with a trivial bundle and the middle vertical map is post-composition with the truncation map $s: \bb{kP}^\infty_{d} \to \bb{kP}^\infty_{d+1}$.
    \item The map $H^\ast(\Sigma^{\infty+\dim_\R(\bb{k})-1}\bb{kP}^\infty_{d+1}; \bb{F}_2) \to H^\ast(\Sigma^{\infty+\dim_\R(\bb{k})-1}\bb{kP}^\infty_{d}; \bb{F}_2)$ is a degree-wise injection.
\end{enumerate}
\end{thm}

\begin{proof} 
The rows are exact by \cref{thm: stably trivial stunted projective}. Adding a trivial line bundle is represented by the map $\Ba(\bb{k}^d) \to \Ba(\bb{k}^{d+1})$ induced by the standard inclusion $\bb{k}^d \subset \bb{k}^{d+1}$, and so we need only describe the effect of calculus on this map.  

The functor $D_{[1,\infty]}\Ba(-)= \fib(\Ba(-) \to \BAut)$ is the reduced approximation of $\Ba$, i.e., the derivatives (and hence the layers) of $D_{[1,\infty]}\Ba(-)$ agree with those of $\Ba$, but the $0$-th stage of the Weiss tower of $D_{[1,\infty]}\Ba(-)$ is trivial. As such, it suffices to show that the first derivative of this reduced functor is given by the desired truncation map. 

The functor
\[
(D_{[1,\infty]}\Ba(-))^{(1)}  = \fib(D_{[1,\infty]}\Ba(-) \longrightarrow D_{[1,\infty]}\Ba(-)((-)\oplus \bb{k}))
\]
may be, as in the proof of \cref{lem: tower of BAut}, naturally identified with $\Sigma^{\dim_{\bb{R}}(\bb{k})-1}S^{(-)}$. The first layer of the Weiss tower of $\Sigma^{\dim_{\bb{R}}(\bb{k})-1}S^{(-)}$ is given by $D_1\Sigma^{\dim_{\bb{R}}(\bb{k})-1}S^{(-)} \simeq \Omega^{\infty} \Sigma^{\infty + \dim_{\bb{R}}(\bb{k})-1}S^{(-)}$. One can see this by combining the proof of~\cref{lem: tower of BAut} with \cite[Theorem 3.5]{BarnesEldred} and the observation that $D_1(\id)(X) \simeq \Omega^\infty\Sigma^\infty(X)$ on any simply connected space $X$, see e.g., \cite[Proposition 2.1]{Johnson}.

Since $D_{[1,\infty]}\Ba(-)$ is reduced, there is an equivalence $P_1(D_{[1,\infty]}\Ba(-)) \simeq D_1 (D_{[1,\infty]}\Ba(-))$ and on account of the layers of the Weiss tower preserving fibre sequences we obtain a commutative diagram
\[\begin{tikzcd}[ampersand replacement=\&,cramped]
	{\Sigma^{\dim_{\bb{R}}(\bb{k})-1}S^{(-)}} \& {D_{[1,\infty]}\Ba(-)} \& {D_{[1,\infty]}\Ba(-)((-)\oplus \bb{k})} \\
	{\Omega^{\infty} \Sigma^{\infty + \dim_{\bb{R}}(\bb{k})-1}S^{(-)}} \& {D_1(D_{[1,\infty]}\Ba(-))} \& {D_1(D_{[1,\infty]}\Ba(-))((-) \oplus \bb{k})}
	\arrow[from=1-1, to=1-2]
	\arrow[from=1-1, to=2-1]
	\arrow[from=1-2, to=1-3]
	\arrow[from=1-2, to=2-2]
	\arrow[from=1-3, to=2-3]
	\arrow[from=2-1, to=2-2]
	\arrow[from=2-2, to=2-3]
\end{tikzcd}\]
with exact rows, and vertical maps linear approximations. After evaluation at $\bb{k}^d$, we may identify the domain and codomain of the bottom right horizontal map. It is $\Omega^\infty$ of \emph{some} map of spectra
$$
 \Sigma^{\infty + \dim_\R(\bb{k})-1}\bb{kP}^\infty_{d} \simeq \bb{D}_1(D_{[1,\infty]}\Ba(-))(\bb{k}^d)  \to
\bb{D}_1(D_{[1,\infty]}\Ba(-))(\bb{k}^{d+1}) \simeq  \Sigma^{\infty + \dim_\R(\bb{k})-1}\bb{kP}^\infty_{d+1}.
$$
The commutative diagram above identifies the fibre of this latter map as a shift of the sphere spectrum, so the map must indeed be truncation. This establishes (1) and (2) now follows immediately.
\end{proof}

\subsection{Digression: stably trivial spherical fibrations}\label{digression: spherical}
Denote by $\mathsf{G}(d)$ the topological monoid of homotopy automorphisms of the sphere $S^{d-1}$. The classifying space $\sf{BG}(d)$ classifies $(d-1)$-spherical fibrations, i.e., fibrations with fibre homotopy equivalent to $S^{d-1}$, see e.g.,~\cite{Stasheff}. The space $\sf{BG}(d)$ may be viewed as the value at $\R^d$ of an orthogonal functor
\[
\Bg : \vect{\bb{k}} \longrightarrow \T, \ V \longmapsto \sf{BG}(V),
\]
for instance by identifying $S^{\dim(V)-1}$ with the unit sphere $S(V)$ in $V$, or by identifying $\sf{G}(V)$ (up to homotopy) with the space of proper homotopy equivalences of $V$, see e.g.,~\cite[\S9.1.2]{KrannichRandal-Williams}.

\begin{prop}\label{prop: tower of BG}
The Weiss tower of $\Bg$ is given by
\[\begin{tikzcd}
	&& \vdots \\
	&& {P_2\Bg(V)} & {D_2\Bg(V)} \\
	&& {P_1\Bg(V)} & {\Omega^\infty \Sigma^\infty\bb{RP}^\infty_{\dim_\R(V)}} \\
	{\Bg(V)} && {\sf{BG}}
	\arrow[from=1-3, to=2-3]
	\arrow[from=2-3, to=3-3]
	\arrow[from=2-4, to=2-3]
	\arrow[from=3-3, to=4-3]
	\arrow[from=3-4, to=3-3]
	\arrow[curve={height=-12pt}, from=4-1, to=1-3]
	\arrow[curve={height=-12pt}, from=4-1, to=2-3]
	\arrow[curve={height=-12pt}, from=4-1, to=3-3]
	\arrow[from=4-1, to=4-3]
\end{tikzcd}\]
in low degrees, with the connectivity of the layers growing linearly in $n$, i.e.,
\[
\sf{Conn}(D_n\Bg(V)) = n(\dim_\R(V)-2)
\]
whenever $\dim_\R(V) \geq 3$. In particular, the first derivative of $\Bg$ is the Borel $O(1)$-spectrum $\partial_1\Bg \simeq \bb{S}$ with trivial $O(1)$-action.
\end{prop}
\begin{proof}
There is a natural transformation of orthogonal functors
\[
\Bo \longrightarrow \Bg
\]
induced by the inclusion $O(d) \subset \sf{G}(d)$. This map induces a map of first derivative Borel $O(1)$-spectra
\[
\partial_1(\Bo) \longrightarrow \partial_1(\Bg)
\]
which is an equivalence by~\cite[Proposition 3.5]{ReisWeiss}, hence our identification of the first derivative of $\Bg$ as the sphere spectrum with trivial $O(1)$-action. It follows that the natural transformation on first layers
\[
D_1\Bo \longrightarrow D_1\Bg
\]
is an equivalence thus providing the identification of the first layer. 

The connectivity estimate for the layers is slightly more involved that in the case of $\Ba$. There is a fibre sequence 
\[
\Omega^{V}S^{V} \longrightarrow \sf{G}(V) \longrightarrow S(V)
\]
where $S(V)$ denotes the unit sphere in $V$. Arguing analogously to the case of $\Ba$ in~\cref{lem: tower of BAut}, it follows from this fibre sequence that the map 
\[
\sf{G}(V) \longrightarrow P_n\sf{G}(V)
\]
is $((n+1)(\dim_\R(V)-2)+1)$-connected whenever $\dim_\R(V) \geq 3$. The equivalence $\sf{G}(V) \simeq \Omega\Bg(V)$, yields the same estimate for the map $\Bg(V) \to P_n\Bg(V)$, so
\[
\sf{Conn}(D_n\Bg(V)) = n(\dim_\R(V)-2),
\]
whenever $\dim_\R(V) \geq 3$.
\end{proof}

It follows from \cref{prop: tower of BG} that for any finite complex $X$, there is an equivalence $D_1\Bo^X \simeq D_1\Bg^X$. Since the map
\[
\Bg(\R^d) \longrightarrow \sf{BG}\simeq \sf{BGL}_1(\bb{S})
\]
is $d$-connected~\cite[Lemma 6.30]{LuckMacko}, this gives the following result classifying stably trivial spherical fibrations and their relationship to real vector bundles.

\begin{prop} Let $X$ be a finite cell complex of dimension at least $2$. For $\frac{\dim(X)+4}{2} 
    \leq d < \dim(X)+1$:
\begin{enumerate}
    \item There is an exact sequence
    \[
    [\Sigma X, \sf{BG}] \xrightarrow{\delta_\sf{G}} [\Sigma^\infty X, \Sigma^\infty \bb{RP}^\infty_d]  \longrightarrow \sf{Sph}_{d-1}^0(X) \longrightarrow 0
    \]
    expressing the set $\sf{Sph}_{d-1}^0(X)$ of stably trivial $(d-1)$-spherical fibrations over $X$ as a quotient of the group of spherical fibrations equipped with a stable trivialisation.
    \item In particular, taking sphere bundles is a surjection in this range:
    \[
    \vect{\R, d}^0(X) \twoheadrightarrow  \sf{Sph}_{d-1}^0(X)
    \]
    i.e.~every stably trivial spherical bundle comes from a stably trivial real bundle.
\end{enumerate}
\end{prop}
\begin{proof}
Part (1) follows analogously to~\cref{thm: stably trivial stunted projective}. Part (2) is the combination of (1), the equivalence 
\[
D_1\Bo^X \xrightarrow{ \ \simeq \ } D_1\Bg^X
\]
obtained from~\cref{prop: tower of BG}, and the classification of~\cref{thm: stably trivial stunted projective}.
\end{proof}

We are sure this result must be well-known to the experts, although we were unable to locate a reference. Using this result, some information about stably trivial spherical bundles may be deduced from our results.

\section{Realification of vector bundles and Weiss calculus}\label{section:realification}

Every complex vector space has an underlying real vector space obtained by forgetting the complex multiplication. This process of \emph{realification} assembles into a functor $r : \mathsf{Vect}_\bb{C} \to \mathsf{Vect}_\bb{R}$, from the category of finite-dimensional complex inner product spaces and linear isometries to the category of finite-dimensional real inner product spaces and linear isometries. By precomposition, there is an induced functor
\[
r^\ast : \Fun(\mathsf{Vect}_\bb{R}, \T) \longrightarrow \Fun(\mathsf{Vect}_\bb{C}, \T),
\]
on the level of functor categories. The second author~\cite{TaggartOCUC} proved that this realification functor provides a synergy between the orthogonal and unitary Weiss calculi. In this section we extend on the aspects of~\cite{TaggartOCUC} that we will need to relate the unitary calculus of $\Bu$ to the orthogonal calculus of $\Bo$. 

For this comparison to work correctly we need to assume that the functors involved are \emph{weakly polynomial} in the sense of~\cite[Definition 9.11]{TaggartUC}. This condition essentially amounts to having a convergent Weiss tower, and the reader should feel free to think of weakly polynomial functors as the Weiss calculus version of `analytic' functors from Goodwillie calculus. 

\begin{thm}\label{thm: realification of towers}
Let $F: \vect{\bb{R}} \to \T$ be an orthogonal functor. If $F$ is weakly polynomial, then there are natural equivalences
\begin{align*}
    P_n(r^\ast F) &\simeq r^\ast(P_n F) \\
    D_n(r^\ast F) &\simeq r^\ast(D_n F)
\end{align*}
of unitary functors. In words, the unitary tower of $r^\ast F$ is the realification of the orthogonal tower of $F$.
\end{thm}
\begin{proof}
    The special case where $F$ is reduced (i.e., when $F(\R^\infty) \simeq *$) was proven as~\cite[Theorem 5.5]{TaggartOCUC}.

    To extend to non-reduced functors, consider the fibre sequence 
    \[
    D_{[1,\infty]}F \longrightarrow F \longrightarrow P_0F,
    \]
    which defines $D_{[1,\infty]}F $ (the \emph{reduced part} of $F$). We have that $P_0( D_{[1,\infty]}F) \simeq *$, and the derivatives of $D_{[1,\infty]}F$ agree with those of $F$.  There is a diagram of fibre sequences 
\[\begin{tikzcd}
	{r^\ast P_n( D_{[1,\infty]}F)} & {r^\ast P_nF} & {r^\ast P_nP_0F} \\
	{P_n(r^\ast P_n( D_{[1,\infty]}F))} & {P_n(r^\ast P_nF)} & {P_n(r^\ast P_nP_0F)}
	\arrow[from=1-1, to=1-2]
	\arrow["{\eta_{r^\ast P_n( D_{[1,\infty]}F)}}"', from=1-1, to=2-1]
	\arrow[from=1-2, to=1-3]
	\arrow["{\eta_{r^\ast P_nF}}"', from=1-2, to=2-2]
	\arrow["{\eta_{r^\ast P_nP_0F}}"', from=1-3, to=2-3]
	\arrow[from=2-1, to=2-2]
	\arrow[from=2-2, to=2-3]
\end{tikzcd}\]
in which the top row is obtained from the fibre sequence above by applying the $n$-th polynomial approximation and precomposing with the realification functor, and the vertical maps are the universal maps in the \emph{unitary} Weiss tower. The left vertical map is an equivalence by~\cite[Theorem 5.5]{TaggartOCUC} so it suffices to show that the right vertical map is an equivalence. To see this, observe that there are equivalences
\[
r^\ast P_nP_0F(V) \simeq r^\ast P_0F(V) = P_0F(rV) \simeq F(\R^\infty)
\]
since a $0$-polynomial functor is $n$-polynomial for all $n \geq 0$. It follows that the functor $r^\ast P_nP_0F$ is $0$-polynomial (i.e., homotopically constant) and hence $n$-polynomial for all $n \geq 0$. In particular, the right vertical map is an equivalence, as required.
\end{proof}

The weakly polynomial assumption of~\cref{thm: realification of towers} is essentially a convergence criterion: weakly polynomial functors are those for which the maps $F(V) \to P_nF(V)$ are connected in a range linearly dependent on $n$ and $\dim_\R(V)$ (up to some finite error). Functors with this property have convergent Weiss tower, see e.g.,~\cite[Lemma 9.10]{TaggartUC}. The functor $\Ba$ is weakly polynomial by~\cref{lem: tower of BAut}.

We use these comparisons to provide a relationship between real and complex vector bundles. Realification of vector bundles is represented by the subgroup inclusion $U(d) \subseteq  O(2d)$. This map induces a natural transformation
\[
\Bu \longrightarrow r^\ast\Bo,
\]
of unitary functors, and hence a natural transformation between the unitary Weiss towers. By~\cref{thm: realification of towers}, we get a commutative diagram, which summarises the situation on $n$-th stages:
\[\begin{tikzcd}
	{P_n\Bu(V)} && {P_n\Bo(rV)} \\
	& {P_n(r^\ast\Bo)(V).}
	\arrow[from=1-1, to=1-3]
	\arrow[from=1-1, to=2-2]
	\arrow["\simeq", from=2-2, to=1-3]
\end{tikzcd}\]
This natural transformation between unitary Weiss towers provides a calculus level interpretation of the realification map at the level of stably trivial bundles. We introduce the notation $r: \Sigma^{\infty+1}\bb{CP}^\infty_d \to \Sigma^\infty\bb{RP}^\infty_{2d}$ for the \emph{realification map} on first layers $D_1\Bu(\C^d) \to D_1\Bo(\R^{2d})$.

\begin{thm}\label{thm: realification vbs calc}
Let $X$ be a finite cell complex of dimension at least $4 = 2\dim_{\bb{R}}(\C)$.
\begin{enumerate}
    \item In the metastable range, there is a commutative diagram with exact rows
\[\begin{tikzcd}[ampersand replacement=\&,cramped]
	{[\Sigma X, \BU]} \& {[\Sigma^\infty X, \Sigma^{\infty +1}\bb{CP}^\infty_d]} \& {\vect{\C,d}^0(X)} \& 0\\
	{[\Sigma X, \BO]} \& {[\Sigma^\infty X, \Sigma^{\infty}\bb{RP}^\infty_{2d}]} \& {\vect{\R,2d}^0(X)} \& 0
	\arrow["{\delta_\C}", from=1-1, to=1-2]
	\arrow["{r_\ast}"', from=1-1, to=2-1]
	\arrow[from=1-2, to=1-3]
	\arrow["{r_\ast}"', from=1-2, to=2-2]
	\arrow["r"', from=1-3, to=2-3]
	\arrow["{\delta_\R}"', from=2-1, to=2-2]
	\arrow[from=2-2, to=2-3]
    \arrow[from=2-3, to=2-4]
    \arrow[from=1-3, to=1-4]
\end{tikzcd}\]
    in which the right vertical map is realification of vector bundles, the middle vertical map is by post-composition with the above map $r: \Sigma^{\infty+1}\bb{CP}^\infty_d \to \Sigma^\infty\bb{RP}^\infty_{2d}$, and the left vertical map is induced by the inclusion $U \subseteq O$.
    \item The map $r^\ast: H^\ast(\Sigma^\infty\bb{RP}^\infty_{2d}; \bb{F}_2) \to H^\ast(\Sigma^{\infty+1}\bb{CP}^\infty_d; \bb{F}_2)$ is a degree-wise surjection. Furthermore, the restriction to bottom cells
    \[
    \Sigma^{\infty+1}\bb{CP}^{d}_d \xrightarrow{\ r \ } \Sigma^\infty\bb{RP}^{2d+1}_{2d} \xrightarrow{\ t \ } \Sigma^\infty\bb{RP}^{2d+1}_{2d+1}
    \]
    of the composite of realification and truncation is an integral homotopy equivalence.
\end{enumerate}
\end{thm}

\begin{proof}
For (1), the diagram exists by our above definition of the middle vertical map $r$. The rows are as usual exact by \cref{thm: stably trivial stunted projective}. That the diagram commutes follows by the definition of our map $r$.

For (2) observe that the realification and truncation maps fit into a diagram
\[\begin{tikzcd}
	{\Sigma S^{(-)}} & {U(\C \oplus -)/U(-)} & \Bu & {\Bu(\C\oplus -)} \\
	& {O(\R^2 \oplus r(-))/O(r(-))} & {r^\ast \Bo} & {(r^\ast \Bo)(\C \oplus -)} \\
	{S^{\R \oplus r(-)}} & {O(\R^2 \oplus r(-))/O(\R \oplus r(-))} & {r^\ast (\Bo(\R \oplus -))} & {r^\ast (\Bo(\R^2 \oplus -))}
	\arrow["\cong"', from=1-1, to=1-2]
	\arrow["\simeq"', from=1-1, to=3-1]
	\arrow[from=1-2, to=1-3]
	\arrow["r"', from=1-2, to=2-2]
	\arrow[from=1-3, to=1-4]
	\arrow["r"', from=1-3, to=2-3]
	\arrow["r"', from=1-4, to=2-4]
	\arrow[from=2-2, to=2-3]
	\arrow[from=2-2, to=3-2]
	\arrow[from=2-3, to=2-4]
	\arrow[from=2-3, to=3-3]
	\arrow["\cong", from=2-4, to=3-4]
	\arrow["\cong"', from=3-1, to=3-2]
	\arrow[from=3-2, to=3-3]
	\arrow[from=3-3, to=3-4]
\end{tikzcd}\]
in which the horizontal directions are fibre sequences. The vertical composite $\Sigma S^{(-)} \longrightarrow S^{\bb{R} \oplus r(-)}$ is induced by realification and is clearly seen to be an equivalence of unitary functors. Since $\bb{D}_1$ preserves fibre sequences, passage to first homogeneous layer produces an analogous diagram of fibre sequences,
\[\begin{tikzcd}
	{\mathbb{D}_1\Sigma S^{(-)}} & {\mathbb{D}_1U(\C \oplus -)/U(-)} & {\mathbb{D}_1\Bu} & {\mathbb{D}_1\Bu(\C\oplus -)} \\
	& {\mathbb{D}_1O(\R^2 \oplus r(-))/O(r(-))} & {\mathbb{D}_1r^\ast \Bo} & {(r^\ast \Bo)(\C \oplus -)} \\
	{\mathbb{D}_1S^{\R \oplus r(-)}} & {\mathbb{D}_1O(\R^2 \oplus r(-))/O(\R \oplus r(-))} & {\mathbb{D}_1r^\ast (\Bo(\R \oplus -))} & {\mathbb{D}_1r^\ast (\Bo(\R^2 \oplus -))}
	\arrow["\cong"', from=1-1, to=1-2]
	\arrow["\simeq"', from=1-1, to=3-1]
	\arrow[from=1-2, to=1-3]
	\arrow["r"', from=1-2, to=2-2]
	\arrow[from=1-3, to=1-4]
	\arrow["r"', from=1-3, to=2-3]
	\arrow["r"', from=1-4, to=2-4]
	\arrow[from=2-2, to=2-3]
	\arrow[from=2-2, to=3-2]
	\arrow[from=2-3, to=2-4]
	\arrow[from=2-3, to=3-3]
	\arrow["\cong", from=2-4, to=3-4]
	\arrow["\cong"', from=3-1, to=3-2]
	\arrow[from=3-2, to=3-3]
	\arrow[from=3-3, to=3-4]
\end{tikzcd}\]
where $\mathbb{D}_1\Sigma S^{(-)} \to \mathbb{D}_1 S^{\bb{R} \oplus r(-)}$ is an equivalence. By~\cref{lem: tower of BAut}, when evaluated at $\C^d$, the right-hand part of this diagram becomes
\[\begin{tikzcd}[ampersand replacement=\&,cramped]
	{\Sigma^{\infty+1}\bb{CP}^d_d} \& {\Sigma^{\infty+1}\bb{CP}^\infty_d} \& {\Sigma^{\infty+1}\bb{CP}^\infty_{d+1}} \\
	{\Sigma^{\infty}\bb{RP}^{2d+1}_{2d}} \& {\Sigma^\infty\bb{RP}^\infty_{2d}} \& {\Sigma^\infty\bb{RP}^\infty_{2d+2}} \\
	{\Sigma^{\infty}\bb{RP}^{2d+1}_{2d+1}} \& {\Sigma^\infty\bb{RP}^\infty_{2d+1}} \& {\Sigma^\infty\bb{RP}^\infty_{2d+2}}
	\arrow[from=1-1, to=1-2]
	\arrow["r"', from=1-1, to=2-1]
	\arrow[from=1-2, to=1-3]
	\arrow["r"', from=1-2, to=2-2]
	\arrow["r"', from=1-3, to=2-3]
	\arrow[from=2-1, to=2-2]
	\arrow[from=2-1, to=3-1]
	\arrow[from=2-2, to=2-3]
	\arrow[from=2-2, to=3-2]
	\arrow["\cong", from=2-3, to=3-3]
	\arrow[from=3-1, to=3-2]
	\arrow[from=3-2, to=3-3]
\end{tikzcd}\]
where we may identify the terms of the leftmost column as stunted projective spaces by~\cref{thm: truncation is stabilisation} and in which the left-most vertical composite is an equivalence. It follows by induction on $n \geq d$ that the restriction of our map to any cell,
$$\Sigma^{\infty+1}\bb{CP}^{n}_n \longrightarrow \Sigma^{\infty}\bb{RP}^{2n+1}_{2n} \longrightarrow \Sigma^{\infty}\bb{RP}^{2n+1}_{2n+1}$$
is an integral homotopy equivalence, so in particular the first part of this composite is a mod-2 cohomology surjection. Knowing the mod-2-cohomology of the spaces $\bb{CP}^\infty_d$ and $\bb{RP}^\infty_{2d}$ (in particular, that all boundary maps in their cellular chain complexes are zero) now implies that the original map
\[
r: \Sigma^{\infty+1}\bb{CP}^\infty_d \longrightarrow \Sigma^\infty\bb{RP}^\infty_{2d}
\]
must also be a mod-2 cohomology surjection.
\end{proof}

\subsection{Digression: Natural transformations}\label{digression: honat}

As we have just discussed, realification induces a natural transformation $D_1\Bu \to D_1(r^\ast\Bo)$.
It turns out that the group of (homotopy classes of) such natural transformations is surprisingly rich. We digress here to give a description of this group via a detour into Borel equivariant stable homotopy theory and the Segal conjecture. The result is the following, where we write $\bb{Z}_2$ for the 2-adic integers. 

\begin{prop}
There is an isomorphism
\[
[D_1\Bu, D_1(r^\ast\Bo)] \cong \bb{Z} \oplus \bb{Z}_2
\]
between homotopy classes of natural transformations from $D_1\Bu$ to $D_1\Bo(r^\ast)$ and the group $\bb{Z} \oplus \bb{Z}_2$.
\end{prop}
\begin{proof}
The classification of homogeneous functors~\cite[Theorem 7.3]{Weiss} implies that there is an equivalence between the category of homogeneous functors of degree $n$ and the category of Borel $\Aut(n)$-spectra. Fully faithfulness of this equivalence implies that there is an isomorphism
\[
[D_1\Bu, r^\ast D_1\Bo] \cong [\del_1\Bu, \del_1(r^\ast \Bo)]_{\BU(1)} \cong [\del_1 \Bu, \sf{Ind}_{O(1)}^{U(1)}\del_1\Bo]_{\BU(1)},
\]
where $[-,-]_{\BU(1)}$ denotes homotopy classes of maps of Borel $U(1)$-spectra and where the last isomorphism follows from~\cite[Corollary 5.6]{TaggartOCUC}, where the reduced assumption may be removed from \emph{loc. cit.} analogously to~\cref{thm: realification of towers}.

By~\cref{lem: tower of BAut}, we may identify 
\[
 [\del_1 \Bu, \sf{Ind}_{O(1)}^{U(1)}\del_1\Bo]_{\BU(1)} \cong [\bb{S}^1, \sf{Ind}_{O(1)}^{U(1)}\bb{S}]_{\BU(1)}
\]
in which $\bb{S}^1$ has trivial $U(1)$-action and $\bb{S}$ has trivial $O(1)$-action. We will now use the Wirthm\"uller isomorphism for Borel equivariant spectra to identify these homotopy classes.

Consider the composition
\[\begin{tikzcd}
	{BO(1)} & {BU(1)} & \ast
	\arrow["i"', from=1-1, to=1-2]
	\arrow["q", curve={height=-24pt}, from=1-1, to=1-3]
	\arrow["p"', from=1-2, to=1-3]
\end{tikzcd}\]
where $i$ is the canonical map induced by the inclusion $O(1) \subset U(1)$. These map induce adjoint triples
\[\begin{tikzcd}
	{\s^{\BO(1)}} & {\s^{\BU(1)}} & \s
	\arrow["{i_!}", shift left=3, from=1-1, to=1-2]
	\arrow["{i_\ast}"', shift right=3, from=1-1, to=1-2]
	\arrow["{i^\ast}"{description}, from=1-2, to=1-1]
	\arrow["{p_!}", shift left=3, from=1-2, to=1-3]
	\arrow["{p^\ast}"{description}, from=1-3, to=1-2]
	\arrow["{p_\ast}", shift left=3, tail reversed, no head, from=1-3, to=1-2]
\end{tikzcd}\]
on the level of categories of Borel spectra (a nice survey of these ideas is given in~\cite[\S2.5]{NaefStoll} but see also~\cite{Cnossen}). Viewing $\bb{S}$ and $\bb{S}^1$ as non-equivariant spectra, then the homotopy class of maps in question has a `shrieks and stars' interpretation as
\[
[\bb{S}^1, \sf{Ind}_{O(1)}^{U(1)} \bb{S}]_{\BU(1)} = [p^\ast(\bb{S}^1), i_!q^\ast(\bb{S})]_{\BU(1)} \cong [\bb{S}^1, p_\ast i_!q^\ast(\bb{S})]
\]
since the respective actions on $\bb{S}^1$ and $\bb{S}$ are trivial, and $q^\ast$ is left adjoint to $q_\ast$. 

The map $i: \BO(1) \to \BU(1)$ has fibre a compact manifold $U(1)/O(1)$, hence by a fiberwise form of Poincaré duality (or a ``Borel'' form of the Wirthm\"uller isomorphism), there is a natural equivalence
\[
i_!(D_i \wedge -) \simeq i_\ast
\]
as functors of the form $\s^{\BO(1)} \to \s^{\BU(1)}$, where $D_i$ is the dualizing spectrum of $i$. Said differently, these categories of Borel spectra form a Wirthm\"uller context, see e.g.,~\cite[Lemma 2.5.4]{NaefStoll}. In this particular case, the dualizing spectrum is given by $D_i \simeq S^{-L}$, where $L = T_{O(1)/O(1)} U(1)/O(1)$ is the tangent representation of the identity coset of $U(1)/O(1)$ with the canonical $O(1)$-action. The representation $L$ may be computed by examining the difference between the Lie algebras for $U(1)$ and $O(1)$. Computing this representation, see e.g.,~\cite[Remark 3.2.26]{SchwedeGlobal}, we see that $L=\R$ is the trivial representation of $O(1)$ and hence $D_i = S^{-1}$. It follows that
\[
p_\ast i_! q^\ast (\bb{S}) \simeq p_\ast i_\ast(D_i^\vee \wedge (q^\ast(\bb{S}))) \simeq q_\ast (\bb{S}^1) \simeq (\bb{S}^1)^{hC_2}.
\]
Putting this all together we see that
\[
[\bb{S}^1, \sf{Ind}_{O(1)}^{U(1)}\bb{S}]_{\BU(1)} \cong [\bb{S}^1, \Map_\ast((BC_2)_+, \bb{S}^1)]  \cong [\bb{S}^1, (\bb{S}^1)^{hC_2}] \cong [(BC_2)_+, \bb{S}] \cong \bb{Z} \oplus \bb{Z}_2
\]
where $\bb{Z}_2$ denotes the $2$-adic integers. The last isomorphism is the Segal conjecture~\cite{Lin} for $C_2$: the group $[(BC_2)_+, \bb{S}]$ is the stable cohomotopy of $BC_2$ which is identified with the completion $\bb{Z} \oplus \bb{Z}_2$ of the Burnside ring $A(C_2)$ at its augmentation ideal.
\end{proof}

\subsection{Digression: The fibre of realification}\label{dig: fib of real}
The space $O(2d)/U(d)$ of orthogonal complex structures on $\bb{R}^{2d}$ is the value at $\C^d$ of a unitary functor 
\[
\sf{o/u}: \vect{\bb{C}} \to \T, \quad \bb{C}^d \longmapsto O(\bb{R}^{2d})/U(\bb{C}^d).
\]
Much like how the functors $\Bo$ and $\Bu$ have been used to classify vector bundles, the functor $\sf{o/u}$ may be used to classify almost complex structures. We will not pursue this idea here. From our point of view, the utility of this functor comes from (its defining) fibre sequence
\[
\sf{o/u} \longrightarrow \Bu \longrightarrow r^\ast\Bo,
\]
through which $\sf{o/u}$ measures the homotopical difference between complex bundles and their realification. We give a brief account of the unitary Weiss tower of $\sf{o/u}$ since it may be of independent interest. 

Let $\es{L}_n^\bb{k}$ denote the topological poset of proper direct sum decompositions of $\bb{k}^n$. The objects are unordered collections of mutually orthogonal subspaces of $\bb{k}^n$ whose direct sum is $\bb{k}^n$. The space of objects may be topologised as a subspace
\[
\sf{Ob}(\es{L}_n^\bb{k}) \subseteq \coprod_{m >1}\left[ \left( \coprod_{k \geq 1} \sf{Gr}_k(\bb{k}^n)\right)^m/\Sigma_m \right].
\]
Given two direct sum decompositions $\Lambda_1$ and $\Lambda_2$, there exists a unique morphism $\Lambda_1 \to \Lambda_2$ if every component of $\Lambda_1$ is a subspace of a component of $\Lambda_2$. Denote by $L_n^\bb{k} \coloneq |\mathcal{L}_n^\bb{k}|^\diamond$ the unreduced suspension of the classifying space of $\es{L}_n^\bb{k}$. These spaces were used by Arone~\cite{Arone} to give a complete description of the derivatives and layers of $\Bo$ and $\Bu$. 

The $\Aut(n)$-action on $\bb{k}^n$ induces an $\Aut(n)$-action on $L_n^\bb{k}$. There is an  $O(n)$-equivariant map $L_n^\bb{R} \to L_n^\bb{C}$ induced by complexification of vector spaces, where $L_n^\C$ is given an $O(n)$-action by restricting the $U(n)$-action. This $O(n)$-equivariant map is adjoint to a $U(n)$-equivariant map 
\begin{equation}\label{eq:poset map}
\sf{Ind}_{O(n)}^{U(n)} L_n^\R \longrightarrow L_n^\C.    
\end{equation}
We will denote the homotopy cofibre of this $U(n)$-equivariant map by $L_n^{\C/\R}$. In this digression we prove the following.

\begin{prop}\label{prop: tower of o/u}
The Weiss tower of $\sf{o/u}$ is given by
\[\begin{tikzcd}
	&& \vdots \\
	&& {P_2(\sf{o/u})(V)} & {\Omega^\infty \Map_\ast( L_2^{\C/\R}, S^{\bb{C}^2 \otimes V} \wedge \bb{S}^{\sf{Ad}_{U(2)}})_{hU(2)}} \\
	&& {P_1(\sf{o/u})(V)} & {\Omega^\infty \Map_\ast(U(1)/O(1)^\diamond, \Sigma\bb{S}^{V})_{hU(1)}} \\
	{\sf{o}/\sf{u}(V)} && {O/U.}
	\arrow[from=1-3, to=2-3]
	\arrow[from=2-3, to=3-3]
	\arrow[from=2-4, to=2-3]
	\arrow[from=3-3, to=4-3]
	\arrow[from=3-4, to=3-3]
	\arrow[curve={height=-12pt}, from=4-1, to=1-3]
	\arrow[curve={height=-12pt}, from=4-1, to=2-3]
	\arrow[curve={height=-12pt}, from=4-1, to=3-3]
	\arrow[from=4-1, to=4-3]
\end{tikzcd}\]
with $n$-th derivative Borel $U(n)$-spectra given by
\[
\partial_n(\sf{o/u}) \simeq \Map_\ast( L_n^{\bb{C}/\bb{R}}, \bb{S}^{\sf{Ad}_{U(n)}}).
\]
In particular, $\partial_1(\sf{o/u}) \simeq \Map_\ast(U(1)/O(1)^\diamond, \bb{S}^1)$, where $(-)^\diamond: \es{S} \to \T$ denotes unreduced suspension. The Weiss tower of $\sf{o/u}$ is convergent, with 
\[
\sf{Conn}(D_n(\sf{o/u})(V)) = n(\dim_\R(V)-1)-1
\]
whenever $\dim_\R(V) \geq 2$.
\end{prop}

We now need a lemma from (genuine) equivariant stable homotopy theory based on incomplete universes, which we could not find in the literature. Let $G$ be a compact Lie group, with $\es{U}_G$ a $G$-universe: a countably infinite-dimensional $G$-representation. The $\infty$-category $\s^G_{\es{U}_G}$ of $\es{U}_G$-spectra is the initial stable $\infty$-category in which the representation sphere $S^V$ is invertible for all irreducible sub-representations $V$ of $\es{U}_G$, i.e., for any $F: \es{S}_\ast^G \to \mathcal{D}$ such that $F(S^V)$ is invertible for every $V \subset \es{U}_G$ , there exists (see e.g.,~\cite[Corollary C.7]{GepnerMeier}) an essentially unique lift
\[\begin{tikzcd}
	{\es{S}_\ast^G} & {\es{D}} \\
	{\s^G_{\es{U}_G}}
	\arrow["F", from=1-1, to=1-2]
	\arrow["{\Sigma^\infty_{G}}"', from=1-1, to=2-1]
	\arrow["{\exists!\bar{F}}"', dashed, from=2-1, to=1-2]
\end{tikzcd}\]
For a concrete definition of the $\infty$-category $\s^G_{\es{U}_G}$, see e.g.,~\cite[Definition C.1]{GepnerMeier}. The lemma we will need to describe the derivatives of $\sf{o/u}$ is the following, which informally says that the stabilisation functor commutes with restriction to subgroups, and is immediate from the universal property of $G$-spectra. 

\begin{lem}\label{lem: restriction of universes}
Let $G$ be a compact Lie group with $G$-universe $\es{U}_G$ and let $H$ be a subgroup of $G$ with $H$-universe $\es{U}_H$. If the restriction $\res_H^G \es{U}_G$ of $\es{U}_G$ to a $H$-universe embeds $H$-equivariantly in $\es{U}_H$, then there exists an essentially unique filler
\pushQED{\qed}
\[\begin{tikzcd}
	{\es{S}_\ast^{G}} & {\es{S}_\ast^{H}} \\
	{\s^{G}_{\es{U}_G}} & {\s^{H}_{\es{U}_H}.}
	\arrow["{\res_{H}^{G}}", from=1-1, to=1-2]
	\arrow["{\Sigma^\infty_{G}}"', from=1-1, to=2-1]
	\arrow["{\Sigma^\infty_H}", from=1-2, to=2-2]
	\arrow["{\res_{H}^{G}}"', dashed, from=2-1, to=2-2]
\end{tikzcd} \qedhere\]
\popQED
\end{lem}

We may now prove \cref{prop: tower of o/u}.

\begin{proof}[Proof of \cref{prop: tower of o/u}]
The layers of the Weiss tower preserve fibre sequences so there is a fibre sequence 
\[
\bb{D}_n\sf{o/u} \longrightarrow \bb{D}_n\Bu \longrightarrow \bb{D}_n(r^\ast \BO)(-) \simeq \bb{D}_n\BO(r(-)),
\]
where the equivalence follows from~\cref{lem: tower of BAut}. 

It is a theorem of Arone \cite[Theorem 2, \S5]{Arone} that 
\begin{align*}
\bb{D}_n \BU(V) &\simeq \Map_\ast(L_n^\bb{C}, Q_{U(n)}(S^{\bb{C}^n \otimes V} \wedge EU(n)_+))^{U(n)} \\
\bb{D}_n \BO(V) &\simeq \Map_\ast(L_n^\bb{R}, Q_{O(n)}(S^{\bb{R}^n \otimes V} \wedge EO(n)_+))^{O(n)}
\end{align*}
where $Q_{\Aut(n)} \coloneq \Omega^\infty_{\Aut(n)}\Sigma^\infty_{\Aut(n)}$ denotes the $\Aut(n)$-equivariant stable homotopy functor associated to the incomplete $\Aut(n)$-universe $\es{U}_\bb{k} = \bigcup_{k} \bb{k}^{nk}$ which is a countable union of copies of the standard representation of $\Aut(n)$, see \cite[p. 480]{Arone}. It follows that 
\[
\bb{D}_n(r^\ast\Bo)(V) \simeq \bb{D}_n \BO(rV) \simeq \Map_\ast(L_n^\bb{R}, Q_{O(n)}(S^{\bb{R}^n \otimes rV} \wedge EO(n)_+))^{O(n)}.
\]
We now take our scheduled detour through the incomplete universes. Consider the diagram
\[\begin{tikzcd}
	{\es{S}_\ast^{U(n)}} & {\es{S}_\ast^{O(n)}} & {\s^{O(n)}_{\es{U}_\R}} \\
	{\s^{U(n)}_{\es{U}_\C}} & {\s^{U(n)}_{r\es{U}_\C}}
	\arrow["{\res_{O(n)}^{U(n)}}", from=1-1, to=1-2]
	\arrow["{\Sigma^\infty_{U(n)}}"', from=1-1, to=2-1]
	\arrow["{\Sigma^\infty_{O(n)}}", from=1-2, to=1-3]
	\arrow["\simeq", from=2-1, to=2-2]
	\arrow["{\res_{O(n)}^{U(n)}}"', dashed, from=2-2, to=1-3]
\end{tikzcd}\]
where the equivalence between $U(n)$-spectra on the complex universe $\es{U}_\C$ and $U(n)$-spectra on the real universe $r\es{U}_\C$ follows analogously to~\cite[\S8.2]{TaggartOCUC}, and the dotted arrow exists by~\cref{lem: restriction of universes}.

It follows that restriction and stabilisation commute, so we may identify the space $Q_{O(n)}(S^{\bb{R}^n \otimes rV} \wedge EO(n)_+)$ with the space $\res^{U(n)}_{O(n)} Q_{U(n)}(S^{\bb{C}^n \otimes V} \wedge EU(n)_+)$ and hence, using the $(\sf{Ind}_{O(n)}^{U(n)}, \res^{U(n)}_{O(n)})$-adjunction, we get an equivalence 
\begin{align*}
\bb{D}_n(r^\ast\BO)(V) 
&\simeq \bb{D}_n \BO(rV) \\
&\simeq \Map_\ast(L_n^\bb{R}, \res^{U(n)}_{O(n)} Q_{U(n)}(S^{\bb{C}^n \otimes V} \wedge EU(n)_+) )^{O(n)} \\
&\simeq \Map_\ast(\sf{Ind}_{O(n)}^{U(n)} L_n^\bb{R}, Q_{U(n)}(S^{\bb{C}^n \otimes V} \wedge EU(n)_+) )^{U(n)}.
\end{align*}
Unravelling these equivalences, one observes that the map $\bb{D}_n\BU(V) \to \bb{D}_n\BO(rV)$ is equivalent to the map 
\[
\Map_\ast(L_n^\bb{C}, Q_{U(n)}(S^{\bb{C}^n \otimes V} \wedge EU(n)_+))^{U(n)}  \longrightarrow \Map_\ast(\sf{Ind}^{U(n)}_{O(n)} L_n^\bb{R}, Q_{U(n)}(S^{\bb{C}^n \otimes V} \wedge EU(n)_+) )^{U(n)},
\]
induced by the canonical $U(n)$-equivariant map \eqref{eq:poset map}. In particular, the fibre is determined by mapping out of the cofibre sequence
\[
\sf{Ind}_{O(n)}^{U(n)} L_n^\bb{R} \longrightarrow L_n^\bb{C} \longrightarrow L_n^{\bb{C}/\bb{R}}
\]
(which defines $L_n^{\bb{C}/\bb{R}}$) i.e.,
\[
\bb{D}_n\sf{o/u}(V) \simeq \Map_\ast( L_n^{\bb{C}/\R}, Q_{U(n)}(S^{\bb{C}^n \otimes V} \wedge EU(n)_+))^{U(n)},
\]
Following the argument \cite[p. 480]{Arone} of Arone computing the derivatives of $\Ba$, the result follows by applying the Adams isomorphism for incomplete universes~\cite[Theorem 12.1]{Lewis}.

The more specific computation of the first derivative now follows from the fact that $L_1^\C \cong S^0 \cong L_1^\R$, while convergence follows from a crude connectivity argument using the estimates of~\cref{lem: tower of BAut}.
\end{proof}

\begin{rem}
A na\"ive computation using the \emph{Weiss cross-effects} of~\cite[\S2]{Weiss} would suggest that the first derivative of $\sf{o/u}$ is given by $\Sigma^{-1}\bb{S}$ with trivial $U(1)$-action. This cannot be the case, as it would imply that the map $\partial_1(\sf{o/u}) \to \partial_1\Bu$ is null. One can see that this map cannot be null by homology considerations for the map $D_1 \Bu(\C^d) \to D_1\Bo(\R^{2d})$.
\end{rem}

\begin{rem}
It would be conceptually much cleaner to have a poset representative for the cofibre $L_n^{\bb{C}/\R}$ i.e., a topological poset $\es{L}_n^{\C/\R}$ with an action of $U(n)$ such that there is a $U(n)$-equivariant equivalence $|\es{L}_n^{\C/\R}|^\diamond \simeq L_n^{\bb{C}/\R}$.
\end{rem}

\subsection{Digression: Complexification} \label{digression: complexification}
Every real vector space has an associated complex vector space obtained by tensoring over $\R$ with $\C$. Post-composition with complexification induces a functor 
\[
c^\ast : \Fun(\mathsf{Vect}_\bb{C}, \T) \longrightarrow \Fun(\mathsf{Vect}_\bb{R}, \T)
\]
on the level of functor categories. We again slightly extend work of the second author~\cite{TaggartOCUC}.

\begin{prop}\label{prop: complexification of towers}
Let $F: \vect{\bb{C}} \to \T$ be a unitary functor. If $F$ is weakly polynomial, then there are natural equivalences
\begin{align*}
P_n(c^\ast F) & \simeq c^\ast P_{\lfloor\frac{n}{2}\rfloor} F, \\
D_n (c^\ast F) &\simeq 
\begin{cases}
c^\ast D_{\frac{n}{2}} (F) & \text{if $2 \vert n$} \\
\ast & \text{otherwise}
\end{cases}
\end{align*}
of orthogonal functors. 
\end{prop}
\begin{proof}
As with \cref{thm: realification of towers}, we may assume that $F$ is reduced. The second equivalence follows from the first equivalence since~\cite[Lemma 4.2 and Theorem 4.4]{TaggartOCUC} imply that odd degree homogeneous layers of $c^\ast F$ vanish.

For $n$ even, the first equivalence follows from~\cite[Proposition 5.4]{TaggartOCUC}. For $n$ odd, say $n = 2k-1$, there is a commutative diagram
\[\begin{tikzcd}
	{P_{2k-1}(c^\ast F)} & {P_{2k-1}(c^\ast P_{k-1}F)} \\
	{P_{2k-2}(c^\ast F)} & {P_{2k-2}(c^\ast P_{k-1}F)}
	\arrow[from=1-1, to=1-2]
	\arrow[from=1-1, to=2-1]
	\arrow[from=1-2, to=2-2]
	\arrow[from=2-1, to=2-2]
\end{tikzcd}\]
induced the map $c^\ast F \to c^\ast P_{k-1}F$, in which the bottom map and right-hand vertical maps are equivalences by~\cite[Proposition 5.4]{TaggartOCUC}. By the case of $n$ even, it suffices to show that the left-most vertical map is an equivalence and hence it suffices to show that the top map is an equivalence. For this, note that since $F$ is weakly polynomial, for each complex vector space $V$ within the radius of convergence of $F$, the map $F(V) \to P_{k-1}F(V)$ is  $(k\dim_\R(V)-b)$-connected for some constant $b$. It follows that for each real vector space $W$ within the radius of convergence of $c^\ast F$, the map $c^\ast F(W) \to c^\ast P_{k-1}F(W)$ is $(2k\dim_\R(W)-b)$-connected, see e.g.,~\cite[Lemma 5.2]{TaggartOCUC}. In the language of \emph{op. cit.} the map $c^\ast F \to c^\ast P_{k-1}F$ is an order $2k$ orthogonal agreement, and hence~\cite[Lemma 2.22]{TaggartOCUC} yields the claim upon noting that if a natural transformation between polynomial functors is an equivalence when evaluated on vector spaces of large enough dimension, it is weak equivalence, see e.g.,~\cite[Proposition 2.31]{BarnesEldred}. \end{proof}

There is a natural transformation $\Bo \longrightarrow c^\ast\Bu$ of orthogonal functors induced by the subgroup inclusion $O(d) \subset U(d)$. As for realification, there is an induced natural transformation between Weiss towers but since $D_1(c^\ast\Bu)(\R^d)$ is contractible by~\cref{prop: complexification of towers}, we obtain the following.

\begin{prop}
In the metastable range, every stably trivial vector bundle has trivial complexification. \qed
\end{prop}

\begin{rem}
There exist real vector bundles which are not stably trivial but which have trivial complexification. To see this, it suffices to exhibit a non-trivial stable real bundle with trivial complexification. For example, take $\xi$ to be the real vector bundle over the circle corresponding to the non-trivial element in $\pi_1\sf{KO} \cong \bb{Z}/2$, i.e., $\xi$ is the Möbius bundle over $S^1$. Since $\pi_1\sf{KU} \cong 0$, the complexification of $\xi$ is trivial.
\end{rem}

\part{Computations}\label{part:computations}
At this juncture the main theoretical inputs are in place (\cref{main theorem: calc}). What remains is to prove~\cref{thm: CPell diagram} and~\cref{thm:sphere diagrams}. The starting point will be to compute the first few stable homotopy groups of stunted projective spaces, which we do via the Adams spectral sequence in~\cref{section: stunted projective}. In essence, this section is case-by-case resolution of differentials, in various Adams spectral sequences, and bouncing these different Adams spectral sequences computations off one another. 

Geometrically, the homotopy groups of stunted projective space correspond to bundles over spheres which are equipped with a stable trivialisation. In~\cref{section: B' implies B} we take this point of view seriously and deduce~\cref{thm:sphere diagrams}. The computations of~\cref{section: stunted projective} give the $E_2$-page of the Atiyah--Hirzebruch spectral sequences required to prove~\cref{thm: CPell diagram}. We carry out these spectral sequence computations in~\cref{section: B' implies A}. Again, this section becomes a case-by-case exercise in resolving differentials.

Since everything in this part is stable homotopy theory, we drop notation for suspension spectra and stable homotopy groups, i.e., $\pi_\ast(X)$ means $\pi_\ast^s(\Sigma^\infty X)$ from here on out. 

\section{Stable homotopy groups of stunted projective spaces}\label{section: stunted projective}

In this section we will prove the following, which is the analogue of \cref{thm:sphere diagrams} for bundles equipped with a stable trivialisation. It will be used in the proof of \cref{thm: CPell diagram} in \cref{section: B' implies A}, and in \cref{section: B' implies B}, we will deduce \cref{thm:sphere diagrams} from it.

\begin{thm} \label{thm: sphere diagrams 2} In small codimensions, the 2-local stable homotopy groups $\pi_m(\Sigma \bb{CP}^\infty_d)_{(2)}$ and $\pi_m(\bb{RP}^\infty_d)_{(2)}$ together with the realification and truncation maps between them, are as shown in \cref{fig:even sphere diagram 2} for $m=2n$ even and in \cref{fig:odd sphere diagram 2} for $m=2n+1$ odd.
\begin{figure}
\begin{center}
\begin{tikzpicture}[commutative diagrams/every diagram]
\matrix[matrix of math nodes, name=m, ampersand replacement=\&,
        row sep=1.5em, column sep=0.6cm,
        commutative diagrams/every cell]{
\&[-0.3cm]
\scalebox{0.75}{$\begin{cases}
  0        & n\equiv 1,5\ (8)\\
  \bb{Z}/2 & n\equiv 0,3,4\ (8)\\
  \bb{Z}/4 & n\equiv 6,7\ (8)\\
  \bb{Z}/8 & n\equiv 2\ (8)
\end{cases}$}
\& \&
\scalebox{0.75}{$\begin{cases}\bb{Z}/2 & n\text{ odd}\\ 0 & n\text{ even}\end{cases}$}
\& \&
\scalebox{0.75}{$0$} \\
\& \pi_{2n}(\Sigma\bb{CP}^\infty_{n-2}) \& \& \pi_{2n}(\Sigma\bb{CP}^\infty_{n-1}) \& \& \pi_{2n}(\Sigma\bb{CP}^\infty_{n}) \\
\& \& \& \& \& \\[1.8em]
\& \pi_{2n}(\bb{RP}^\infty_{2n-4}) \&[0.8cm] \pi_{2n}(\bb{RP}^\infty_{2n-3}) \&[3.2cm] \pi_{2n}(\bb{RP}^\infty_{2n-2}) \&[1.2cm] \pi_{2n}(\bb{RP}^\infty_{2n-1}) \& \pi_{2n}(\bb{RP}^\infty_{2n}) \\
\& \scalebox{0.75}{$\begin{cases}\bb{Z}/2 & n\not\equiv 1\ (4)\\ 0 & n\equiv 1\ (4)\end{cases}$}
\& \scalebox{0.75}{$\begin{cases}(\bb{Z}/2)^2 & n\equiv 3\ (4)\\ \bb{Z}/2 & \text{otherwise}\end{cases}$}
\& \scalebox{0.75}{$\begin{cases}(\bb{Z}/2)^2 & n\text{ odd}\\ 0 & n\text{ even}\end{cases}$}
\& \scalebox{0.75}{$\begin{cases}\bb{Z}/2 & n\text{ odd}\\ 0 & n\text{ even}\end{cases}$}
\& \scalebox{0.75}{$\bb{Z}$} \\};
\path[commutative diagrams/.cd, every arrow, every label]
(m-2-2) edge node {$0$} (m-2-4)
(m-2-4) edge node {$0$} (m-2-6)
(m-4-2) edge node {\scalebox{0.8}{$\begin{cases}
  \begin{psmallmatrix}1\\0\end{psmallmatrix} & n\equiv 3\ (4)\\
  \begin{psmallmatrix}1\end{psmallmatrix} & n\text{ even}
\end{cases}$}} (m-4-3)
(m-4-3) edge node[xshift=5mm] {\scalebox{0.8}{$\begin{cases}
  \begin{psmallmatrix}0&0\\0&1\end{psmallmatrix} & n\equiv 3\ (4)\\
  \begin{psmallmatrix}0\\1\end{psmallmatrix} & n\equiv 1\ (4)
\end{cases}$}} (m-4-4)
(m-4-4) edge node {\scalebox{0.8}{$\substack{\begin{psmallmatrix}1&0\end{psmallmatrix}\\[2pt]n\text{ odd}}$}} (m-4-5)
(m-4-5) edge node {$0$} (m-4-6)
(m-2-2) edge[commutative diagrams/two heads] (m-4-2)
(m-2-4) edge[commutative diagrams/hook] node {\scalebox{0.8}{$\begin{psmallmatrix}1\\0\end{psmallmatrix}\quad n\text{ odd}$}} (m-4-4)
(m-2-6) edge node {$0$} (m-4-6)
(m-1-2) edge[commutative diagrams/equal] (m-2-2)
(m-1-4) edge[commutative diagrams/equal] (m-2-4)
(m-1-6) edge[commutative diagrams/equal] (m-2-6)
(m-4-2) edge[commutative diagrams/equal] (m-5-2)
(m-4-3) edge[commutative diagrams/equal] (m-5-3)
(m-4-4) edge[commutative diagrams/equal] (m-5-4)
(m-4-5) edge[commutative diagrams/equal] (m-5-5)
(m-4-6) edge[commutative diagrams/equal] (m-5-6);
\end{tikzpicture}
\end{center}
    \caption{The realification and truncation maps on even spheres in low codimensions after localisation at 2 (proven as \cref{lem: sphere diagrams 2 even}). For compactness we omit zero maps, and abuse notation by writing $\pi_i(X) = \pi_i(X)_{(2)}$.}
    \label{fig:even sphere diagram 2}
\end{figure}

\begin{figure}
\begin{center}
\noindent\resizebox{\linewidth}{!}{%
\begin{tikzpicture}[commutative diagrams/every diagram]
\matrix[matrix of math nodes, name=m, ampersand replacement=\&,
        row sep=1.5em, column sep=0.6cm,
        commutative diagrams/every cell]{
\&[-0.3cm]
\scalebox{0.75}{$\begin{cases}\bb{Z}_{(2)}\oplus\bb{Z}/2 & n\text{ odd}\\ \bb{Z}_{(2)} & n\text{ even}\end{cases}$}
\& \& \scalebox{0.75}{$\bb{Z}_{(2)}$} \& \& \scalebox{0.75}{$0$} \\
\& \pi_{2n+1}(\Sigma\bb{CP}^\infty_{n-1}) \& \& \pi_{2n+1}(\Sigma\bb{CP}^\infty_{n}) \& \& \pi_{2n+1}(\Sigma\bb{CP}^\infty_{n+1}) \\
\& \& \& \& \& \\[1.8em]
\& \pi_{2n+1}(\bb{RP}^\infty_{2n-2}) \&[2.0cm] \pi_{2n+1}(\bb{RP}^\infty_{2n-1}) \&[2.8cm] \pi_{2n+1}(\bb{RP}^\infty_{2n}) \&[2.2cm] \pi_{2n+1}(\bb{RP}^\infty_{2n+1}) \&[1.2cm] \pi_{2n+1}(\bb{RP}^\infty_{2n+2}) \\[2.4em]
\& \scalebox{0.75}{$\begin{cases}(\bb{Z}/8)^2 & n\equiv 1\ (4)\\ \bb{Z}/4 & n\equiv 0,2\ (4)\\ \bb{Z}/16\oplus\bb{Z}/4 & n\equiv 3\ (4)\end{cases}$}
\& \scalebox{0.75}{$\begin{cases}\bb{Z}/8 & n\text{ odd}\\ \bb{Z}/2 & n\text{ even}\end{cases}$}
\& \scalebox{0.75}{$\begin{cases}\bb{Z}/4 & n\text{ odd}\\ (\bb{Z}/2)^2 & n\text{ even}\end{cases}$}
\& \scalebox{0.75}{$\bb{Z}/2$} \& \scalebox{0.75}{$0$} \\};
\path[commutative diagrams/.cd, every arrow, every label]
(m-2-2) edge node {\scalebox{0.8}{$\begin{cases}\begin{psmallmatrix}1&0\end{psmallmatrix} & n\text{ odd}\\ (2) & n\text{ even}\end{cases}$}} (m-2-4)
(m-2-4) edge node {$0$} (m-2-6)
(m-4-2) edge node[swap, yshift=-2.5mm, xshift=4mm] {\scalebox{0.8}{$\begin{cases}\begin{psmallmatrix}1&0\end{psmallmatrix} & n\equiv 1\ (4)\\ \begin{psmallmatrix}1&2\end{psmallmatrix} & n\equiv 3\ (4)\end{cases}$}} (m-4-3)
(m-4-3) edge node[swap, yshift=-2.5mm, xshift=2mm] {\scalebox{0.8}{$\begin{cases}(1) & n\text{ odd}\\ \begin{psmallmatrix}0\\1\end{psmallmatrix} & n\text{ even}\end{cases}$}} (m-4-4)
(m-4-4) edge node[swap, yshift=-2.5mm, xshift=2mm] {\scalebox{0.8}{$\begin{cases}(1) & n\text{ odd}\\ \begin{psmallmatrix}1&0\end{psmallmatrix} & n\text{ even}\end{cases}$}} (m-4-5)
(m-4-5) edge node[swap] {$0$} (m-4-6)
(m-2-2) edge node[swap] {\scalebox{0.8}{$\begin{cases}\begin{psmallmatrix}1&4\\0&0\end{psmallmatrix} & n\equiv 1\ (4)\\ \begin{psmallmatrix}1&0\\0&2\end{psmallmatrix} & n\equiv 3\ (4)\\ (1) & n\equiv 0\ (4)\\ (2) & n\equiv 6\ (8)\end{cases}$}} (m-4-2)
(m-2-4) edge node {\scalebox{0.8}{$\begin{cases}(1) & n\text{ odd}\\ \begin{psmallmatrix}1\\0\end{psmallmatrix} & n\text{ even}\end{cases}$}} (m-4-4)
(m-2-6) edge node {$0$} (m-4-6)
(m-1-2) edge[commutative diagrams/equal] (m-2-2)
(m-1-4) edge[commutative diagrams/equal] (m-2-4)
(m-1-6) edge[commutative diagrams/equal] (m-2-6)
(m-4-2) edge[commutative diagrams/equal] (m-5-2)
(m-4-3) edge[commutative diagrams/equal] (m-5-3)
(m-4-4) edge[commutative diagrams/equal] (m-5-4)
(m-4-5) edge[commutative diagrams/equal] (m-5-5)
(m-4-6) edge[commutative diagrams/equal] (m-5-6);
\end{tikzpicture}%
}
\end{center}
    \caption{The realification and truncation maps on odd spheres in low codimensions after localisation at 2 (proven as \cref{lem: sphere diagrams 2 odd}). For compactness we omit zero maps, and abuse notation by writing $\pi_i(X) = \pi_i(X)_{(2)}$.}
    \label{fig:odd sphere diagram 2}
\end{figure}
\end{thm}

This result will be proven using the Adams Spectral Sequence. The $E_2$-pages may be computed in the relevant range by Bruner's Ext software~\cite{Bruner} - they are shown in \cref{fig:SCPASS1,fig:SCPASS2,fig:RPASS1,fig:RPASS2,fig:RPASS3,fig:RPASS4}. With the Adams charts in hand we will determine the homotopy groups by determining the possible differentials. The groups in the top rows have already been computed by Hu as~\cite[Lemma 4.1]{Hu}. The case $m$ even will be proven as \cref{lem: sphere diagrams 2 even}, and the case $m$ odd (which is substantially more irritating) will be proven as \cref{lem: sphere diagrams 2 odd}. 

The potential differentials for the real projective spaces will turn out to all be forced by the maps. The differentials in the complex cases (\cref{fig:SCPASS1}) are harder. These were resolved by Hu \cite[p. 7807]{Hu}, using Matsunaga's~\cite{Matsunaga} calculation of the \emph{unstable} homotopy group $\pi_{2n+3}(U(n))$. It is desirable to avoid using detailed information about finite unitary groups, so we provide an alternative calculation of these differentials, using a result of Mosher~\cite{Mosher} on the cell structure of $\bb{CP}_{n}^{n+2}$. In fact, Hu uses this result in a different part of his argument. In~\cref{section: B' implies A} we will use Mosher's result again analogously to how it is implemented by Hu.

\subsection{The Adams charts} We have the following Adams charts for certain finite skeleta of the stable representing object $\Sigma^{\dim_\R(\bb{k})-1}\bb{kP}^\infty_{d}$ \cref{fig:SCPASS1,fig:SCPASS2,fig:RPASS1,fig:RPASS2,fig:RPASS3,fig:RPASS4}. In the complex case these charts were computed by hand in~\cite[p.7806-7807]{Hu}, and ours were generated using Bruner's Ext software~\cite{Bruner}. These charts coincide with those for $\Sigma^{\dim_\R(\bb{k})-1}\bb{kP}^\infty_{d}$ (for the relevant values of $d$) in the unshaded region, surjectz onto them in the light-shaded region, and are in principle unrelated to them in the dark-shaded region. In what follows we will use these to compute the homotopy groups in~\cref{thm: sphere diagrams 2}.

The possible differentials affecting the unshaded range (all $d_2$'s) are shown in blue in \cref{fig:SCPASS1,fig:SCPASS2,fig:RPASS1,fig:RPASS2,fig:RPASS3,fig:RPASS4}. At this juncture, ripping out the pages with the figures is strongly advised (or opening a second version of the pdf for eco-conscious readers).

\begin{figure}[ht]



\caption{Adams $E_2$-pages for $\Sigma \bb{CP}_{4k}^{4k+2}$ (left) and $\Sigma \bb{CP}_{4k+1}^{4k+3}$ (right). The possible differentials which may affect the unshaded region are indicated by a dotted blue arrow.}
\label{fig:SCPASS1}
\end{figure}

\begin{figure}[ht]
%
    \caption{Adams $E_2$-pages for $\Sigma \bb{CP}_{4k+2}^{4k+4}$ (left) and $\Sigma \bb{CP}_{4k+3}^{4k+5}$ (right).}
    \label{fig:SCPASS2}
\end{figure}

\begin{figure}[ht]
    %
    \caption{Adams $E_2$-pages for $\bb{RP}^{8k+5}_{8k}$ (left) and $\bb{RP}^{8k+5}_{8k+1}$ (right). The possible differentials which may affect the unshaded region are indicated by a dotted blue arrow.}
    \label{fig:RPASS1}
\end{figure}

\begin{figure}[ht]
%

    \caption{Adams $E_2$-pages for $\bb{RP}^{8k+7}_{8k+2}$ (left) and $\bb{RP}^{8k+7}_{8k+3}$ (right). The possible differentials which may affect the unshaded region are indicated by a dotted blue arrow.}
    \label{fig:RPASS2}
\end{figure}

\begin{figure}[ht]
%
\caption{Adams $E_2$-pages for $\bb{RP}^{8k+9}_{8k+4}$ (left) and $\bb{RP}^{8k+9}_{8k+5}$ (right). The possible differentials which may affect the unshaded region are indicated by a dotted blue arrow.}
    \label{fig:RPASS3}
\end{figure}

\begin{figure}[ht]
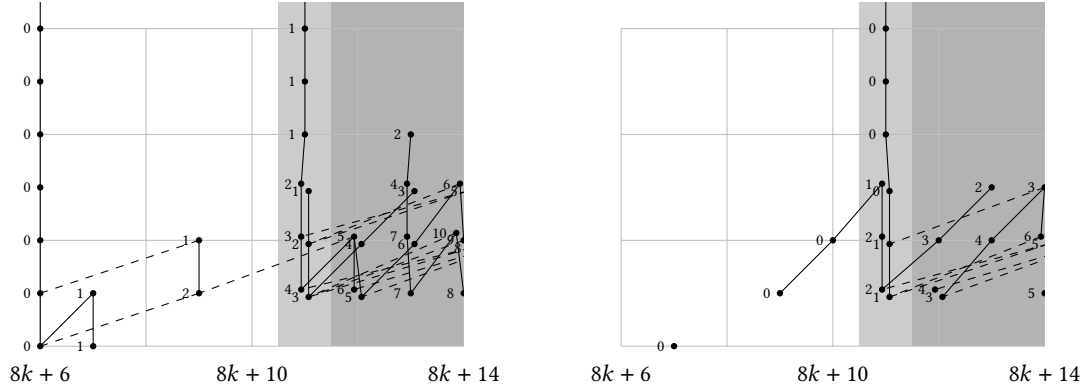

%
    \caption{Adams $E_2$-pages for $\bb{RP}^{8k+11}_{8k+6}$ (left) and $\bb{RP}^{8k+11}_{8k+7}$ (right).}
    \label{fig:RPASS4}
\end{figure}

Some explanation on interpreting these charts is in order. First, the mod-2 cohomology of $\bb{RP}^\infty$ is polynomial on a single generator $x$ in degree 1. There is only one non-zero class in each degree, so the Steenrod square $\mathrm{Sq}^i$ must act by $\mathrm{Sq}^i(x^d) = x^{d+i}$ or $\mathrm{Sq}^i(x^d) = 0$. Which of these possibilities arises depends on the congruence class of $d$ modulo the smallest power of 2 exceeding $i$ \cite[Example 4L.3]{HatcherAT}. It follows that up to a degree shift, the cohomology of $\bb{RP}_{d}^{d+4}$ or $\bb{RP}_{d}^{d+5}$ (as a module over the Steenrod algebra) depends only on the value of $d$ modulo 8. In \cref{fig:RPASS1,fig:RPASS2,fig:RPASS3,fig:RPASS4}, we have given the eight possible Adams charts, writing $d = 8k+\varepsilon$ for $0 \leq \varepsilon \leq 7$. Similar considerations in the complex case produce \cref{fig:SCPASS1,fig:SCPASS2}. Many more complicated periodicity phenomena are known for projective spaces, c.f.~\cite{MahowaldOnJames, HopkinsNotes}.

\subsection{Real differentials} We begin by resolving the real differentials, i.e., those in~\cref{fig:RPASS1,fig:RPASS2,fig:RPASS3,fig:RPASS4}. We will need a small amount of input in terms of the cell structure of $\bb{RP}^{n+1}_n$.

\subsubsection{Cell structure of $\bb{RP}_{n}^{n+1}$}
Consider the attaching map which forms $\bb{RP}_{n}^{n+1}$ as a cofibre:
$$S^{n} \xrightarrow{\ f \ } S^{n} \xrightarrow{\ i\ } \bb{RP}_{n}^{n+1}.$$
By knowledge of the cellular chain complex of $\bb{RP}^\infty$ \cite[Example 2.42]{HatcherAT}, $f$ must be as follows.

\begin{lem} \label{lem:firstRealCellAttachment} The attaching map $f$ for $\bb{RP}_{n}^{n+1}$ is of degree $2$ if $n$ is odd and is null if $n$ is even. \qed
\end{lem}

\begin{lem} \label{lem: real truncation hits eta}
Let $i : S^n \to \bb{RP}^{\infty}_n$ denote the inclusion of the bottom cell, and let $\eta$ denote the generator of $\pi_1(S^0)\cong \bb{Z}/2$. If $n$ is even, then $i \circ \eta \in \pi_{n+1}(\bb{RP}^{\infty}_n)$ is in the image of the truncation map $\bb{RP}^{\infty}_{n-1} \to \bb{RP}^{\infty}_n$.
\end{lem}

\begin{proof} Since $n$ is even, the $n$-skeleton of $\bb{RP}^{\infty}_{n-1}$ is a mod-2 Moore space $S^{n-1}/2$ by \cref{lem:firstRealCellAttachment}. Therefore, the restriction of the truncation map to $n$-skeleta is the top cell projection $p: S^{n-1}/2 \to S^n$, and it suffices to verify that $\eta \in \pi_{n+1}(S^n)$ is in the image of this latter map. This follows from the long exact sequence on homotopy coming from the cofibration
\[
S^{n-1} \longrightarrow S^{n-1}/2 \xrightarrow{\ p\ } S^n
\]
using that the connecting map is multiplication by $2$ and that $\eta$ is $2$-torsion.
\end{proof}

\subsubsection{The differentials} For the real cases, all (five) possible differentials are zero. Most of these differentials are easily resolved because of the action of realification.

\begin{lem} \label{lem: real ASS diffs} The five possible differentials shown in \cref{fig:RPASS1,fig:RPASS2,fig:RPASS3} are all zero.
\end{lem}

\begin{proof} First consider the possible differentials for $\bb{RP}^{8k+5}_{8k}$ and $\bb{RP}^{8k+9}_{8k+4}$ (shown left in \cref{fig:RPASS1,fig:RPASS3} respectively). We will give the argument in the first case, the second being identical. The realification map from $\Sigma \bb{CP}^{4k+2}_{4k}$ induces a map of Adams $E_2$-pages, `from \cref{fig:SCPASS1} to \cref{fig:RPASS1}'. By \cref{main theorem: calc}, this realification map is a surjection on mod-2 cohomology, hence an injection on the 0-lines (i.e.~the bottom row) of the Adams $E_2$-pages. In particular, the classes labelled 0 and 1 in the 0-line of \cref{fig:SCPASS1}-left are mapped respectively to the classes labelled 1 and 2 in \cref{fig:RPASS1}-left. The sources of both possible differentials therefore lie in the image of the realification map (for the right-hand differential this is because the map commutes with the action of $h_2$). The differentials therefore do not exist, because if they did then they would also have to exist in \cref{fig:SCPASS1}-left, and we see there that there can be no differentials for degree reasons. 

It remains to treat the differential in \cref{fig:RPASS2}. It suffices to prove that the truncation map $\bb{RP}^{8k+7}_{8k+3} \to \bb{RP}^{8k+9}_{8k+4}$ sends the source and target of this proposed differential to the source and target of the left hand differential in \cref{fig:RPASS3}, which we have already shown is zero. For the source, this follows because again the map on 0-lines is given by the map on cohomology, which is an isomorphism in this degree by~\cref{main theorem: calc}. For the target, \cref{lem: real truncation hits eta} shows that truncation hits the class $y$ in bidegree $(8k+5,1)$, since $y$ detects $\eta$ on the bottom cell. In the chart (\cref{fig:RPASS2}) for $\bb{RP}^{8k+7}_{8k+3}$, there is only one class in the correct bidegree, which must therefore map to $y$. Multiplying by $h_1$ gives the result.
\end{proof}

\subsection{Complex differentials} As with the real case, in order to resolve the complex differentials (i.e., those in~\cref{fig:SCPASS1}) we need some preliminary information on the cell structure of various stunted projective spaces. Understanding these cell structures will also be necessary to resolve the differentials in the Atiyah--Hirzebruch spectral sequences of~\cref{section: B' implies A}.

\subsubsection{Cell structure of $\bb{CP}_{n}^{n+1}$}\label{subsection: CPn+1n} Consider the (stable) attaching map which forms $\bb{CP}_{n}^{n+1}$ as a cofibre:
$$S^{2n+1} \xrightarrow{\ f\ } S^{2n} \xrightarrow{\ i\ } \bb{CP}_{n}^{n+1}.$$
The map $f$ lies in the first stable stem, and $\pi_1(S^0)=\bb{Z}/2\{\eta\}$, so $f$ is either $\eta$ or null. It is $\eta$ if and only if $\bb{CP}_{n}^{n+1}$ carries a non-zero $\mathrm{Sq}^2$, so the Steenrod algebra structure of $H^*(\bb{CP}_{n}^{n+1};\bb{Z}/2)$ (which follows from that of $H^*(\bb{CP}^{\infty};\bb{Z}/2)$) immediately gives the following.

\begin{lem} \label{lem:firstCellAttachment} The attaching map $f$ for $\bb{CP}_{n}^{n+1}$ is $\eta$ if $n$ is odd and is null if $n$ is even. \qed
\end{lem}

In much of what follows it will be vital to fix (compatible) bases for the groups we know. The following will be useful momentarily and follows because $\pi_2(S^0)=\bb{Z}/2\{\eta^2\}$, $\pi_3(S^0) \cong \bb{Z}/8\{\nu\}$ and $\eta^3=4\nu$.

\begin{lem} \label{lem:pi3Cnu} 
For $n$ odd, we have $$\pi_{2n+3}(\bb{CP}^{n+1}_n) \cong \bb{Z}/4\{i \circ \nu\},$$ where $i$ is the inclusion of the bottom cell. \qed
\end{lem}

\subsubsection{Cell structure of $\bb{CP}_{n}^{n+2}$}\label{subsection: CPn+2n}

In \cite{Mosher}, Mosher computes the stable homotopy groups of $\bb{CP}^{\infty}$ using the spectral sequence arising from the cellular filtration. Roughly speaking, the columns of their $E_0$-page are shifted copies of $\pi_*(S^0)$ (since $\bb{CP}^{\infty}$ has one cell in each even dimension), the $d^1$-differential `corresponds to' the attaching map for $\bb{CP}_{n}^{n+1}$ described in \cref{lem:firstCellAttachment} (c.f. \cite[Proposition 5.1]{Mosher}), and the $d^2$-differential `corresponds to' the attaching map which forms $\bb{CP}_{n}^{n+2}$ as the cofibre
$$S^{2n+3} \xrightarrow{\ g \ } \bb{CP}_{n}^{n+1} \longrightarrow \bb{CP}_{n}^{n+2}.$$

\cref{lem:firstCellAttachment} now describes both $\bb{CP}_{n}^{n+1}$ and $\bb{CP}_{n+1}^{n+2}$, and we can deduce:
\begin{itemize}
    \item When $n$ is odd, $\bb{CP}_{n}^{n+1}$ is a suspension $\Sigma^{2n}C_\eta$ of the cofibre of $\eta$, alias $\Sigma^{-2}\bb{CP}^2$, and by \cref{lem:pi3Cnu} the class $g$ is equal to $\lambda' (i \circ \nu)$ for some $\lambda' \in \bb{Z}/4$, and therefore lifts along $i:S^{2n} \to \bb{CP}_{n}^{n+1}$ to a class $\widetilde{g} = \lambda \nu$ for any $\lambda \in \bb{Z}/8$ reducing to $\lambda '$ modulo $4$.
    \item When $n$ is even, $\bb{CP}_{n}^{n+1} \simeq S^{2n} \oplus S^{2n+2}$, and $g$ takes the form $(\lambda \nu, \eta)$ for some $\lambda \in \bb{Z}/8$.
\end{itemize}

In particular, one of $\bb{CP}_{n}^{n+1}$ and $\bb{CP}_{n+1}^{n+2}$ always splits as a wedge of spheres. We therefore have the following, which we record for later use.

\begin{lem} \label{lem:MosherImplementation}
Let $\lambda$ be as above. When $n$ is odd, there is a map $\Sigma^{2n} C_{\lambda \nu} \to \bb{CP}^{n+2}_n$ which is surjective on mod-2 cohomology. When $n$ is even, there is a map $ \bb{CP}^{n+2}_n \to \Sigma^{2n} C_{\lambda \nu}$ which is injective on mod-2 cohomology. \qed
\end{lem}

Mosher (motivated by determining the $d^2$ differentials in their spectra sequence) determines the values of $\lambda$ as follows.

\begin{lem}[{\cite[Proposition 5.2]{Mosher}}] \label{lem:Mosher}
    The value of $\lambda$ mod 8 is as follows.  \begin{center}
\begin{tabular}{|c | c c c c c c c c  |} 
 \hline
 $n$ \textrm{ mod 8} & 0 & 1 & 2 & 3 & 4 & 5 & 6 & 7   \\ 
 \hline
$\lambda$ & 0 & 2 & 1 & 1 & 2 & 0 & 1 & 1  \\
 \hline
\end{tabular}
\end{center}
\end{lem}

A similar discussion is given by Hu~\cite[p. 7798]{Hu} but we encourage the reader to proceed with caution: our indexing coincides with Mosher, but not with Hu, who considers $\bb{CP}^\ell_{\ell-2}$ rather than $\bb{CP}^{n+2}_n$, see e.g.,~\cite[Lemma 4.4]{Hu}; this is natural in light of the fact that they use this lemma only later, as discussed above. Caution also that in light of the above discussion, when $n$ is odd, $\lambda$ is really only determined modulo $4$.

Mosher also shows the following about the stable Hurewicz image in $\bb{CP}^\infty$, which recasts earlier work of Toda \cite{Toda}. We will use it to resolve some of the complex differentials, hence identify the relevant homotopy groups. The integral homology $H_\ast(\bb{CP}^{\infty};\bb{Z})$ of $\bb{CP}^\infty$ is a divided power algebra on a generator in degree 2, which in particular is a torsion free graded group. It follows that, the stable Hurewicz map $\pi_*(\bb{CP}^\infty) \to H_*(\bb{CP}^{\infty};\bb{Z})$ descends to the torsion-free quotient $\pi_*(\bb{CP}^\infty)/T$.

\begin{thm} \label{thm: MosherHurewicz}
    Taken modulo torsion, the stable Hurewicz map $$\pi_*(\bb{CP}^\infty)/T \longrightarrow H_*(\bb{CP}^{\infty};\bb{Z})$$
    is injective, and its image in $H_{2n}(\bb{CP}^{\infty};\bb{Z}) \cong \bb{Z}$ is the subgroup $n! \cdot \bb{Z}$.
\end{thm}
\begin{proof}
Injectivity is the statement of \cite[Theorem 2.1]{Mosher}, and the identification of the image is given in the proof, see~\cite[p.182]{Mosher}.
\end{proof}

What we will need is (the $n=2$ case of) the following corollary.

\begin{cor} \label{cor: degreeOfTopCellProj}
    The map $\pi_{2n}(\bb{CP}^n) \to \pi_{2n}(S^{2n})$ induced by the projection $\bb{CP}^n \to S^{2n}$ to the top cell has image the subgroup $n! \cdot \bb{Z} \subset \bb{Z} \cong \pi_{2n}(S^{2n})$.
\end{cor}

\begin{proof}
Consider the naturality square for the Hurewicz map
\[\begin{tikzcd}
	\pi_{2n}(\bb{CP}^n) & \pi_{2n}(S^{2n}) \\
	H_{2n}(\bb{CP}^n; \bb{Z}) & H_{2n}(S^{2n}; \bb{Z})
	\arrow[from=1-1, to=1-2]
	\arrow["h", from=1-1, to=2-1]
	\arrow["h", from=1-2, to=2-2]
	\arrow[from=2-1, to=2-2]
\end{tikzcd}\]
in which the bottom and right maps are isomorphisms. It suffices to show that the left-hand Hurewicz map $h: \pi_{2n}(\bb{CP}^n) \to H_{2n}(\bb{CP}^n; \bb{Z}) \cong \bb{Z}$ has image $n! \cdot \bb{Z}$. This follows from \cref{thm: MosherHurewicz} because $\bb{CP}^n$ is the $(2n+1)$-skeleton of $\bb{CP}^{\infty}$.
\end{proof}

We can use the $n=2$ case, and our Adams $E_2$-page for $\bb{CP}_{4k}^{4k+2}$ (\cref{fig:SCPASS1}), to deduce the following, the content of which is really that the map on stable homotopy groups is non-zero with no jump in Adams filtration. Unsurprisingly, we will use this to resolve the differentials in~\cref{fig:SCPASS1}.

\begin{lem} \label{lem: AdamsFiltrationControl} Consider the map induced on Adams $E_2$-pages by the projection to the top cell $$\Sigma \bb{CP}_{4k}^{4k+2} \longrightarrow S^{8k+5}.$$
The class $x$ in bidegree $(8k+5, 1)$ of~\cref{fig:SCPASS1} which corresponds to the top cell of $\Sigma \bb{CP}_{4k}^{4k+2}$ has non-zero image in the $E_2$-page for $S^{8k+5}$.
\end{lem}

\begin{proof} Since $4k$ is even, we may factor the projection to the top cell over the cofibre of $\eta$ (i.e.~by first crushing the bottom cell, then the middle one) writing it as
$$ \Sigma \bb{CP}_{4k}^{4k+2} \xrightarrow{\ p\ } \Sigma^{8k+3} C_{\eta} \xrightarrow{\ q\ } S^{8k+5}$$
where $p$ and $q$ are both mod-$2$ cohomology injections.

Compare the Adams $E_2$-page for $\bb{CP}_{4k}^{4k+2}$ (\cref{fig:SCPASS1}) with the $E_2$-page for the cofibre of $\eta$ (\cref{fig:mod-2-Moore and hook} (right)). Since $p$ is a cohomology injection, it is a surjection on 0-lines. The class $x$ generates the $h_0$-tower corresponding to the top cell of $\bb{CP}_{4k}^{4k+2}$, and chasing around $h_i$-multiplications shows that this $h_0$-tower surjects onto the $h_0$-tower coming from the top cell of $\Sigma^{8k+3} C_{\eta}$. In particular $x$ has non-zero image under $p$.

The spectrum $C_{\eta}$ is equivalently $\Sigma^{-2}\bb{CP}^2$ (e.g.,~by \cref{lem:firstCellAttachment}), so by \cref{cor: degreeOfTopCellProj} the map $q$ has degree 2 on homotopy groups. Since the $h_0$-tower for the top cell of $C_{\eta}$ begins in the 1-line, generated by $p_*(x)$, the only way that this can happen is if $q$ sends $p_*(x)$ to $h_0$ times the generator of the $h_0$-tower for the sphere. In particular, $q_*p_*(x)$ is still non-zero, and the result follows.
\end{proof}

\subsubsection{The differentials}For the complex cases (both in \cref{fig:SCPASS1}), the next two lemmas resolve the possible differentials.

\begin{lem} \label{lem: complex ASS diff 1}
    The Adams $d^2$-differential for $\Sigma \bb{CP}^{4k+2}_{4k}$ shown (left) in \cref{fig:SCPASS1} is zero if and only if $k$ is even.
\end{lem}

\begin{proof} \cref{lem:MosherImplementation} gives a mod-2 cohomology injection $\Sigma \bb{CP}^{4k+2}_{4k} \to \Sigma^{8k+1} C_{\lambda \nu}$. The top cell projection for $\Sigma \bb{CP}^{4k+2}_{4k}$ factors over this map, so by \cref{lem: AdamsFiltrationControl}, the class $x$ generating the $h_0$-tower corresponding to the top cell must have non-zero image in the $E_2$-page for $\Sigma^{8k+1} C_{\lambda \nu}$.
    
\cref{lem:Mosher} tells us that if $k$ is odd (so $n=4k \equiv 4 \pmod{8}$) then $\lambda=2$, and if $k$ is even ($n \equiv 0 \pmod{8}$) then $\lambda = 0$. Either way, the Steenrod algebra action on $H^*(\Sigma^{8k+3} C_{\lambda \nu})$ is trivial, so (ignoring the differentials) the $E_2$-page of its Adams spectral sequence is a direct sum of two shifted copies of the $E_2$-page for the sphere. In particular, the image of $x$ must be $h_0$ times the second generator, because there is no other class in correct bidegree. It follows that the differential exists if and only if $2 \nu$ times the bottom cell of $\Sigma^{8k+1} C_{\lambda \nu}$ is null, which (rather tautologically) happens when $\lambda = 2$, and doesn't when $\lambda = 0$.
\end{proof}

\begin{lem} \label{lem: complex ASS diff 2}
    The Adams $d^2$-differential for $\Sigma \bb{CP}^{4k+3}_{4k+1}$ shown (right) in \cref{fig:SCPASS1} is zero if and only if $k$ is odd.
\end{lem}
\begin{proof}
\cref{lem:MosherImplementation} gives a mod-2 cohomology surjection $\Sigma^{8k+3} C_{\lambda \nu} \to \Sigma \bb{CP}^{4k+3}_{4k+1}$. \cref{lem:Mosher} tells us that if $k$ is odd (so $n=4k+1 \equiv 5 \pmod{8}$) then $\lambda=0$, and if $k$ is even ($n \equiv 1 \pmod{8}$) then $\lambda = 2$. As above, in each case, the $E_2$-page for $\Sigma^{8k+3} C_{\lambda \nu}$ is a direct sum of two shifted copies of the $E_2$-page for the sphere. Since our map is a mod-2 cohomology surjection, it is an isomorphism on the 0-line on $E_2$-pages, and chasing the action of the $h_i$'s shows that both source and target of the proposed differential lie in its image. The conclusion now follows as in the previous lemma.
\end{proof}

This completes the resolution of the differentials, and hence implicitly calculates the relevant stable homotopy groups. We will deduce the groups explicitly along with the maps in the following sections.

\subsection{Truncation and realification over even spheres} We now identify the action of the truncation and realification maps for bundles over spheres equipped with a stable trivialisation. We begin with the easier case of spheres of even dimension.

\begin{lem} \label{lem: sphere diagrams 2 even} In small codimensions, the $2$-local stable homotopy groups $\pi_{2n}(\Sigma \bb{CP}^\infty_d)_{(2)}$ and $\pi_{2n}(\bb{RP}^\infty_d)_{(2)}$, together with the realification and truncation maps between them, are as shown in \cref{fig:even sphere diagram 2}.
\end{lem}

\begin{proof} We work $2$-locally throughout the proof. The groups are as claimed by the charts (\cref{fig:SCPASS1,fig:SCPASS2,fig:RPASS1,fig:RPASS2,fig:RPASS3,fig:RPASS4}), and the calculation of possible differentials (Lemmas \ref{lem: real ASS diffs}, \ref{lem: complex ASS diff 1}, and \ref{lem: complex ASS diff 2}). Our task is to determine the maps and to do so we will work roughly right-to-left in \cref{fig:even sphere diagram 2}. 

\textbf{Truncation} $\pi_{2n}(\bb{RP}^{\infty}_{2n-1}) \to \pi_{2n}(\bb{RP}^{\infty}_{2n})$ is a map from a torsion group to a free group, so must be zero.

\textbf{Truncation} $\pi_{2n}(\bb{RP}^{\infty}_{2n-2}) \to \pi_{2n}(\bb{RP}^{\infty}_{2n-1})$. When $n$ is even both domain and codomain are zero. When $n$ is odd, we are dealing with a map $(\bb{Z}/2)^2 \to \bb{Z}/2$. Since $2n-2$ is of the form $8k$ or $8k+4$ in the mod-8 indexing for the figures, we must consider the left-to-right maps in \cref{fig:RPASS1,fig:RPASS3}. The basis element of the domain detected in the 2-line must map to zero because there is nothing in high enough Adams filtration for it to hit. By \cref{main theorem: calc}, truncation is a mod-2 cohomology injection, so it hits the classes labelled 0 and 1 in the 0-line of the target. Chasing $h_1$ multiplications now shows that either choice of the other basis element has non-zero image (we will fix one choice in the next step). Ordering the basis `by Adams filtration' we get $\begin{psmallmatrix}
    1 & 0
\end{psmallmatrix}$.

\textbf{Realification} $\pi_{2n}(\Sigma \bb{CP}^{\infty}_{n-1}) \to \pi_{2n}(\bb{RP}^{\infty}_{2n-2})$. The codomain is non-zero when ($n$ is odd so) $2n-2$ is of the form $8k$ or $8k+4$, hence we are looking at the Adams charts of~\cref{fig:SCPASS1,fig:SCPASS2} (left in both) for the complex groups and ~\cref{fig:RPASS1,fig:RPASS3} (left in both) for the real groups. In both cases we are dealing with a map $\bb{Z}/2 \to (\bb{Z}/2)^2$. Since realification is a mod-$2$ cohomology surjection (\cref{main theorem: calc}), chasing $h_1$ multiplications shows that the map is injective. The domain is generated by an $\eta$-multiple of any class detected in the 0-line. We take the basis of the target, $\pi_{2n}(\bb{RP}^{\infty}_{2n-2})$ to be the image of this class and the unique class detected on the 2-line. We get $\begin{psmallmatrix}
    1 \\ 0
\end{psmallmatrix}$.

\textbf{Truncation} $\pi_{2n}(\bb{RP}^{\infty}_{2n-3}) \to \pi_{2n}(\bb{RP}^{\infty}_{2n-2})$. To find the right place in the charts, it is helpful to note that we are now mapping \emph{into} the groups we were mapping \emph{out of} in the previous truncation step. When $n$ is even, the codomain is zero. When $n \equiv 3 \pmod{4}$, $2n-3$ is of the form $8k+3$ and the map is $(\bb{Z}/2)^2 \to (\bb{Z}/2)^2$ (\cref{fig:RPASS2}-right to \cref{fig:RPASS3}-left). The $h_1$ and $h_2$ multiplications in the domain show that $\pi_{2n}(\bb{RP}^{\infty}_{2n-3}) \cong (\bb{Z}/2)^2$ has a unique basis where one class is a $\nu$-multiple $x \nu$ and the other class is a $\eta$-multiple $y \eta$. Then $x$ maps to zero for degree reasons, so $x \nu$ also maps to zero, and by \cref{lem: real truncation hits eta} $y$ maps to the unique non-zero class detected in the 1-line, so $y \eta$ maps to the unique non-zero class detected in the 2-line. In total, we get $\begin{psmallmatrix}
    0 & 0 \\ 0 & 1
\end{psmallmatrix}$. In the case $n \equiv 1 \pmod{4}$, the domain is of the form $\bb{Z}/2\{y \eta\}$, and so (again using \cref{lem: real truncation hits eta}) the map is $\begin{psmallmatrix}
    0 \\  1
\end{psmallmatrix}$.

\textbf{Truncation} $\pi_{2n}(\bb{RP}^{\infty}_{2n-4}) \to \pi_{2n}(\bb{RP}^{\infty}_{2n-3})$. It suffices to notice that in the three cases where the domain is non-zero, its generator is a $\nu$-multiple. 

\textbf{Realification} $\pi_{2n}(\Sigma \bb{CP}^{\infty}_{n-2}) \to \pi_{2n}(\bb{RP}^{\infty}_{2n-4})$. Analogously to the previous point, is a surjection.

\textbf{Truncation} $\pi_{2n}(\Sigma \bb{CP}^{\infty}_{n-2}) \to \pi_{2n}(\Sigma \bb{CP}^{\infty}_{n-1})$. We can chase round the left-hand square in \cref{fig:even sphere diagram 2}: we have deduced all of the maps apart from the top horizontal, which we wish to show is zero. The result follows because the composite through the bottom left is zero, and the right-hand vertical map is injective.
\end{proof}

\subsection{Truncation and realification over odd spheres} We repeat the above discussion now for odd-dimensional spheres. The main result is the following lemma. Once it is proven, \cref{thm: sphere diagrams 2} will follow immediately, since it is just the combination of the statements of \cref{lem: sphere diagrams 2 odd} and \cref{lem: sphere diagrams 2 even}.

\begin{lem} \label{lem: sphere diagrams 2 odd} In small codimension, the $2$-local stable homotopy groups $\pi_{2n+1}(\Sigma \bb{CP}^\infty_d)_{(2)}$ and $\pi_{2n+1}(\bb{RP}^\infty_d)_{(2)}$ together with the realification and truncation maps between them, are as shown in \cref{fig:odd sphere diagram 2}.
\end{lem}

The proof is much more intricate that its even-dimensional counterpart. There are two subcases that require quite a lot of extra work. Here we find it clearest to write not in dependency order, so we give the proof here, feeling free to cite the relevant results from (slightly) later in the paper (all proven in the next two sub-subsections,~\cref{subsection: codim 3 odd,subsection: codim 3 even}). This way, the reader who does want to read these details will first see the use case, and the reader who doesn't can skip them.

\begin{proof} 
The groups are as claimed by the charts (\cref{fig:SCPASS1,fig:SCPASS2,fig:RPASS1,fig:RPASS2,fig:RPASS3,fig:RPASS4}), and the calculation of possible differentials (Lemmas \ref{lem: real ASS diffs}, \ref{lem: complex ASS diff 1}, and \ref{lem: complex ASS diff 2}). When a group has multiple summands, we will order them such that summands with generators in lower Adams filtration come first (since it happens that this is always possible in our range). For the maps, it suffices to show that the maps are as claimed `up to units', i.e.~that (bases can be chosen such that) the matrices are as shown in \cref{fig:odd sphere diagram 2} up to possible multiplication of each entry by some unit modulo 2. Starting in the top left of the picture, one can change basis in the target of each map to eliminate all of these units. We will work from the right to the left in \cref{fig:odd sphere diagram 2}, i.e.,~increasing the codimension. The rightmost truncation maps, which go into the stable range, are automatically zero, because their codomains are zero.

\textbf{Truncation} $\pi_{2n+1}(\bb{RP}^\infty_{2n}) \to \pi_{2n+1}(\bb{RP}^\infty_{2n+1})$. When $n$ is odd, $2n$ is of the form $8k+2$ or $8k+6$ in the mod-8 indexing for the figures. The two cases are identical, so we describe only $2n=8k+2$. The truncation map is then from \cref{fig:RPASS2}-left to \cref{fig:RPASS2}-right. We are interested in the component $\bb{Z}/4 \to \bb{Z}/2$ in dimension $8k+3$. Both summands begin on the 0-line, so since the map is a cohomology isomorphism in this degree, it must send the generator to the generator, as claimed. When $n$ is even, $2n = 8k$ or $8k+4$, so we consider the left-to-right map in \cref{fig:RPASS1,fig:RPASS3} respectively. Letting $i$ be the inclusion of the bottom cell, and letting $j$ be any choice of inclusion of the next cell (in the next part, we will make a specific choice), we see that $j$ and $i \circ \eta$ give a basis, that $s_\ast$ sends $j$ to the generator for cohomology reasons, and that $s_\ast(i \circ \eta)= s_\ast(i) \circ \eta = 0$, where $t$ is the truncation map (\cref{main theorem: calc}).

\textbf{Realification} $\pi_{2n+1}(\Sigma \bb{CP}^\infty_{n}) \to \pi_{2n+1}(\bb{RP}^\infty_{2n})$ is as claimed by cohomology considerations (using \cref{main theorem: calc} as normal), after possibly modifying the second generator $j$ by adding $i \circ \eta$.

\textbf{Truncation} $\pi_{2n+1}(\bb{RP}^\infty_{2n -1}) \to \pi_{2n+1}(\bb{RP}^\infty_{2n})$. For $n$ odd ($2n = 8k+2$ or $8k+6$, corresponding respectively to truncation \cref{fig:RPASS1}-right-to-\cref{fig:RPASS2}-left and \cref{fig:RPASS3}-right-to-\cref{fig:RPASS4}-left), knowing the map on cohomology again suffices. When $n$ is even ($2n = 8k$ or $8k+4$), \cref{lem: real truncation hits eta} shows that this map always hits $i \circ \eta$, where as usual $i : S^{2n} \to \bb{RP}^\infty_{2n}$ is the inclusion of the bottom cell, so the result follows.

\textbf{Truncation} $\pi_{2n+1}(\bb{RP}^\infty_{2n -2}) \to \pi_{2n+1}(\bb{RP}^\infty_{2n-1})$. When $n$ is even (left-to-right in  \cref{fig:RPASS2,fig:RPASS4}, dimensions $8k+5$ and $8k+9$) the generator is a $\nu$-multiple of a class that goes to zero for dimension reasons, so the map is zero. When $n$ is odd (\cref{fig:RPASS1,fig:RPASS3}), we proceed as follows. Writing $i : S^{2n-2} \to \bb{RP}^\infty_{2n-2}$ for the inclusion of the bottom cell, we have $t \circ i \simeq *$ for connectivity reasons where $t$ is the truncation map (\cref{main theorem: calc}), and so $s_\ast(i \circ \nu) \simeq *$. Cohomology shows that $s_\ast(x)$ is a generator of the target $\bb{Z}/8$, where $x \in \pi_{2n+1}(\bb{RP}^\infty_{2n -2})$ is any class detected on the 0-line. \cref{cor: codimension 3 bases} (with $n = 2m+1$) now gives bases in terms of the classes defined in \cref{lem: splitting RP4m+3}: in the case $n \equiv 1 \pmod{4}$, we have
\[
\pi_{2n+1}(\bb{RP}^\infty_{2n -2}) \cong \bb{Z}/8 \{k\} \oplus \bb{Z}/8 \{i \circ \nu\}
\]
and the above considerations imply that in this basis the truncation map $s_\ast$ is $\begin{pmatrix} 1 & 0 \end{pmatrix}$. In the case $n \equiv 3 \pmod{4}$, 
\[
\pi_{2n+1}(\bb{RP}^\infty_{2n -2}) \cong \bb{Z}/16 \{k\} \oplus \bb{Z}/4\{a\cdot(i\nu) - (2u) \cdot k\}
\]
for $a$ and $u$ units mod 2, so, up to units, the map is given by $\begin{pmatrix} 1 & 2 \end{pmatrix}$.

\textbf{Truncation} $\pi_{2n+1}(\Sigma \bb{CP}^\infty_{n-1}) \to \pi_{2n+1}(\Sigma\bb{CP}^\infty_{n})$. The target is torsion free, so the map must be zero on the torsion subgroup of the domain, and descends to a map $\bb{Z} \to \bb{Z}$ on the torsion-free quotient. 
When $n$ is odd (left-to-right in \cref{fig:SCPASS1,fig:SCPASS2}), this restriction is a surjection for cohomology reasons. When $n$ is even (right in each of these figures to left in the other), note first that cellular approximation identifies
$$\pi_{2n+1}(\Sigma \bb{CP}^{n}_{n-1}) \cong \pi_{2n+1}(\Sigma \bb{CP}^\infty_{n-1}) \longrightarrow \pi_{2n+1}(\Sigma\bb{CP}^{\infty}_{n}) \cong \pi_{2n+1}(\Sigma\bb{CP}^n_{n})$$
that $\Sigma \bb{CP}^{n}_{n-1} \simeq \Sigma^{2n-1} C_{\eta} \simeq \Sigma^{2n-1} \bb{CP}^2_1$ (\cref{lem:firstCellAttachment}), and that $\Sigma\bb{CP}^n_{n} \simeq S^{2n+1}$. In other words, our map is equivalently the top cell projection for $\bb{CP}^2$, which by \cref{cor: degreeOfTopCellProj} has degree 2.

\textbf{Realification} $\pi_{2n+1}(\Sigma \bb{CP}^\infty_{n-1}) \to \pi_{2n+1}(\bb{RP}^\infty_{2n-2})$. This map is the most difficult to identify. The case $n$ odd is established as \cref{prop: difficult map 1}. The case $n$ even is \cref{prop: difficult map 2}.
\end{proof}

\subsubsection{The case of codimension $3$ real bundles with $n$ odd}\label{subsection: codim 3 odd}

Knowing the effect of a map on Adams $E_2$-pages is in general not enough to deduce its effect on homotopy groups, even once one knows all of the differentials - we only learn the effect of the map `modulo classes of higher Adams filtration'. For the even homotopy groups, we got lucky and this didn't cause any problems. For the odd homotopy groups, we have to work harder. In this section we will resolve the realification map in codimension $3$ with $n$ odd (\cref{prop: difficult map 1}) and the truncation map out of codimension $3$ (\cref{cor: codimension 3 bases}) both of which essentially just provide preferred bases.

Consider realification in real codimension 3 (the leftmost vertical in \cref{fig:odd sphere diagram 2}): 
$$r_*: \pi_{2n+1}(\Sigma \bb{CP}^{\infty}_{n-1})_{(2)} \longrightarrow \pi_{2n+1}( \bb{RP}^{\infty}_{2n-2})_{(2)}.$$

In this section we will treat the case $n$ odd, (in the following section we will treat the case $n$ even), so in the remainder of this section we will set $n = 2m+1$ and consider
$$r_*: \pi_{4m+3}(\Sigma \bb{CP}^{\infty}_{2m})_{(2)}  \longrightarrow \pi_{4m+3}( \bb{RP}^{\infty}_{4m})_{(2)} .$$

The goal of this section is then to prove the following.  For definitions of the classes $i$, $j$, $k$, $j'$, and $k'$ see \cref{lem: splitting RP4m+3}, and for $v_1$ see \cref{lem: v1}.

\begin{prop} \label{prop: difficult map 1}
The map $$\bb{Z}_{(2)}\{k'\} \oplus \bb{Z}/2\{j'\eta^2\} \cong \pi_{2m+3}(\Sigma \bb{CP}^{\infty}_{2m})_{(2)} \xrightarrow{\ r_* \ } \pi_{2m+3}(\bb{RP}^{\infty}_{4m})_{(2)} \cong \begin{cases}
        \bb{Z}/8\{k\} \oplus \bb{Z}/8\{i\} & m \textrm{ even} \\
        \bb{Z}/16\{k\} \oplus \bb{Z}/4\{j v_1\} & m \textrm{ odd}
    \end{cases}$$ is given (up to units) by
$\begin{psmallmatrix}
    1 & 4 \\
    0 & 0
\end{psmallmatrix}$ when $m$ is even, and  $\begin{psmallmatrix}
    1 & 0 \\
    0 & 2
\end{psmallmatrix}$
when $m$ is odd.
\end{prop}

In some sense this result can \emph{almost} be read off the charts, which leave two possibilities. For example, for $m$ even, either the mod-2 generator goes to 4 times where the infinite order generator goes, or it does not (in which case there is only one place it can go). The key thing will be to understand why the groups
 $$\pi_{4m+3}(\bb{RP}^{\infty}_{4m}) \cong \begin{cases}
    (\bb{Z}/8)^2 & m \textrm{ even } (4m = 8k), \\
    \bb{Z}/16 \oplus \bb{Z}/4 &  m \textrm{ odd } (4m = 8k+4).
\end{cases}$$
are what they are. We know these groups from \cref{fig:RPASS1,fig:RPASS3}, where there can be no differentials by \cref{lem: real ASS diffs}, but a geometric understanding of them may be obtained from the attaching map for the $(4m+4)$-cell: 
\begin{equation}\label{eq: attaching 4m+3}
  S^{4m+3} \xrightarrow{\ f \ } \bb{RP}^{4m+3}_{4m} \longrightarrow \bb{RP}^{4m+4}_{4m}.
\end{equation}
This attaching map is a class in $\pi_{4m+3}(\bb{RP}^{4m+3}_{4m})$, which we now study.

We know (\cref{lem:firstCellAttachment}) that we have a splitting
$j' \oplus k' : S^{4m+1} \oplus S^{4m+3} \xrightarrow{\simeq} \Sigma \bb{CP}^{2m+1}_{2m}$. Our first job is to establish a splitting of $\bb{RP}^{4m+3}_{4m}$ which is compatible with this one under realification. Note first that any such splitting $j' \oplus k'$ must restrict to an equivalence $j' :S^{4m+1} \xrightarrow{\simeq} \Sigma \bb{CP}^{2m}_{2m}$.

\begin{lem} \label{lem: splitting RP4m+2} Fix a $2$-local equivalence $j' : S^{4m+1}  \xrightarrow{\simeq} \Sigma \bb{CP}^{2m}_{2m}$. There is a $2$-local splitting $$i \oplus j: S^{4m} \oplus S^{4m+1}/2 \xrightarrow{\ \simeq \ } \bb{RP}^{4m+2}_{4m}$$
in which $i: S^{4m} \to \bb{RP}^{4m+2}_{4m}$ is the inclusion of the bottom cell and where $j$ corresponds to $j'$ up to units in the sense that $rj' \simeq u \cdot j \iota$, where $u$ is a 2-local unit and $\iota: S^{4m+1} \to S^{4m+1}/2$ is the inclusion of the bottom cell and generates $\pi_{4m+1}(S^{4m+1}/2) \cong \bb{Z}/2$.
\end{lem}

\begin{proof}
We work $2$-locally thoughout the proof. Consider the composite of $rj'$ with the top cell projection for $\bb{RP}^{4m+1}_{4m}$, as shown in the following diagram where the bottom row is a cofibre sequence:
\[\begin{tikzcd}
	& \Sigma \bb{CP}^{2m}_{2m} & S^{4m+1} \\
    S^{4m} & \bb{RP}^{4m+1}_{4m} & S^{4m+1}.
	 \arrow["r",swap,from=1-2, to=2-2]
     \arrow["j'", from=1-3, to=1-2]
     \arrow[swap, "\simeq", from=1-3, to=1-2]
     \arrow[from=2-1, to=2-2]
     \arrow[from=2-2, to=2-3]
     \arrow[dashed, swap, "u^{-1}", from=2-3, to=1-3]
\end{tikzcd}\]

Since $r$ is a mod-2 cohomology surjection (\cref{main theorem: calc}), the degree of this composite (which is a map between spheres) must be a 2-local unit $u$. The composite $r j'u^{-1}$ splits the top cell projection by construction, so the map $i \oplus (r j'u^{-1}) : S^{4m} \oplus S^{4m+1} \xrightarrow{\simeq} \bb{RP}^{4m+1}_{4m}$, where $i$ is the inclusion of the bottom cell is a homotopy equivalence. 

To obtain $j$, hence the desired equivalence $i \oplus j$, it suffices to argue that under any homotopy equivalence $ \bb{RP}^{4m+1}_{4m}  \simeq S^{4m} \oplus S^{4m+1}$ (so in particular our equivalence above), the attaching map $g$ forming $\bb{RP}^{4m+2}_{4m}$ as the cofibre
\[
S^{4m+1} \xrightarrow{\ g\ }\bb{RP}^{4m+1}_{4m} \longrightarrow \bb{RP}^{4m+2}_{4m}
\]
is identified with $2$ times the inclusion of $S^{4m+1}$ (up to a $2$-local unit). This will follow because the only non-trivial Steenrod operation on $H^*(\bb{RP}^{4m+2}_{4m}; \bb{F}_2)$ is a $\mathrm{Sq}^1$ connecting degrees $4m+1$ and $4m+2$.

Fix any homotopy equivalence $\bb{RP}^{4m+1}_{4m}  \simeq S^{4m} \oplus S^{4m+1}$. The projection of $g$ to $S^{4m}$ is a class in the first stable stem, so is either $\eta$ or null. The absence of a $\mathrm{Sq}^2$ on $H^*(\bb{RP}^{4m+2}_{4m}; \bb{F}_2)$ means it must be null. The presence of the $\mathrm{Sq}^1$ implies that the projection of $g$ to $S^{4m+1}$ must be $2$ (up to units), as required.
\end{proof}

Now we extend this description of realification over one more cell.

\begin{lem} \label{lem: splitting RP4m+3} 
There are $2$-local splittings
$$j' \oplus k' : S^{4m+1} \oplus S^{4m+3} \xrightarrow{\ \simeq\ } \Sigma \bb{CP}^{2m+1}_{2m}$$
and
$$i \oplus j \oplus k: S^{4m} \oplus S^{4m+1}/2 \oplus S^{4m+3} \xrightarrow{\ \simeq \ } \bb{RP}^{4m+3}_{4m}$$
such that $rj' \simeq u \cdot j \iota$ and $r k' = v \cdot k$, where $u$ and $v$ are 2-local units and $\iota$ generates $\pi_{4m+1}(S^{4m+1}/2) \cong \bb{Z}/2$.
\end{lem}

\begin{proof}
As always, we work $2$-locally. Realification gives a map of cofibre sequences
\[\begin{tikzcd}
S^{4m+1} & & S^{4m+3} \\
	 \Sigma \bb{CP}^{2m}_{2m} & \Sigma \bb{CP}^{2m+1}_{2m} & \Sigma \bb{CP}^{2m+1}_{2m+1}\\
    \bb{RP}^{4m+2}_{4m} & \bb{RP}^{4m+3}_{4m} & \bb{RP}^{4m+3}_{4m+3}  \\
    S^{4m} \oplus S^{4m+1}/2  & & S^{4m+3}.
	\arrow[equal,from=1-1, to=2-1]
    \arrow[from=2-1,to=2-2]
    \arrow["j'", from=1-1,to=2-2]
    \arrow[from=2-2,to=2-3]
    \arrow[equal,from=1-3, to=2-3]
    \arrow["r",from=2-1, to=3-1]
    \arrow["r",from=2-2, to=3-2]
    \arrow["r",from=2-3, to=3-3]
    \arrow["i \oplus j",from=4-1,to=3-1]
    \arrow["\simeq",swap,from=4-1,to=3-1]
    \arrow[from=3-1,to=3-2]
    \arrow[from=3-2,to=3-3]
    \arrow[equal,from=3-3, to=4-3]
    \arrow[swap, dashed,"k'",from=1-3,to=2-2]
    \arrow[ dashed,"k",from=4-3,to=3-2]
\end{tikzcd}\]
By \cref{lem:firstCellAttachment}, we know that a section $k'$ of the top-right horizontal map exists, and we take the first splitting to be $j' \oplus k'$. \cref{lem: splitting RP4m+2} gives a compatible splitting in the bottom left: an equivalence
 $$i \oplus j: S^{4m} \oplus S^{4m+1}/2 \xrightarrow{\ \simeq \ } \bb{RP}^{4m+2}_{4m}$$
such that $rj' \simeq u \cdot j \iota$ for some 2-local unit $u$.

By \cref{main theorem: calc}, the composite $S^{4m+3} \to S^{4m+3}$ of the maps in the rightmost column is a mod-2 cohomology isomorphism, so it must have degree some 2-local unit $v$. Setting $k = r \circ k' \circ v^{-1}$ gives a splitting of the bottom right horizontal map, such that the resulting equivalence
$$i \oplus j \oplus k : S^{4m} \oplus S^{4m+1}/2 \oplus S^{4m+3} \xrightarrow{\ \simeq\ } \bb{RP}^{4m+3}_{4m}$$
satisfies $r k' \simeq v \cdot k$, as required.
\end{proof}

\cref{lem: splitting RP4m+3} allows us to view the attaching map of \eqref{eq: attaching 4m+3} is a class in $\pi_{4m+3}(S^{4m} \oplus S^{4m+1}/2 \oplus S^{4m+3})$, whose possible values we now analyse. We will use Adams $E_2$-pages for two auxiliary complexes (\cref{fig:mod-2-Moore and hook}), which we again obtained using Bruner's Ext software \cite{Bruner}. The first thing is to determine the group. 

\begin{figure}[ht]
    \begin{tikzpicture}[scale=0.7]

\tikzset{radius= 0.05}

\clip (-1.0,-1.0) rectangle (5.5,  5.5);
\draw[color=lightgray] (0,0) grid [step=2]  (16,8);

\foreach \n in {0,2,4,6}
{
\def\nn{\n-0};
\node at (\nn,-0.5) {$\n$};
}

\foreach \s in {0,2,4,6}
{
\def\ss{\s-0};
\node at (-0.5,\ss) {$\s$};
}


\draw [fill] ( 0.00, 0.00) circle;
\node [left] at  ( 0.00, 0.00) {\tiny{0}};

\draw [fill] ( 1.00, 1.00) circle;
\node [left] at  ( 1.00, 1.00) {\tiny{0}};

\draw [fill] ( 2.00, 1.00) circle;
\node [left] at  ( 2.00, 1.00) {\tiny{1}};

\draw [fill] ( 3.00, 1.00) circle;
\node [left] at  ( 3.00, 1.00) {\tiny{2}};

\draw [fill] ( 7.00, 1.00) circle;
\node [left] at  ( 7.00, 1.00) {\tiny{3}};

\draw [fill] (15.00, 1.00) circle;
\node [left] at  (15.00, 1.00) {\tiny{4}};

\draw [fill] ( 2.00, 2.00) circle;
\node [left] at  ( 2.00, 2.00) {\tiny{0}};

\draw [fill] ( 3.00, 2.00) circle;
\node [left] at  ( 3.00, 2.00) {\tiny{1}};

\draw [fill] ( 6.00, 2.00) circle;
\node [left] at  ( 6.00, 2.00) {\tiny{2}};

\draw [fill] ( 7.00, 2.00) circle;
\node [left] at  ( 7.00, 2.00) {\tiny{3}};

\draw [fill] ( 8.00, 2.00) circle;
\node [left] at  ( 8.00, 2.00) {\tiny{4}};

\draw [fill] ( 9.00, 2.00) circle;
\node [left] at  ( 9.00, 2.00) {\tiny{5}};

\draw [fill] (14.00, 2.00) circle;
\node [left] at  (14.00, 2.00) {\tiny{6}};

\draw [fill] (16.00, 2.00) circle;
\node [left] at  (16.00, 2.00) {\tiny{7}};

\draw [fill] ( 4.00, 3.00) circle;
\node [left] at  ( 4.00, 3.00) {\tiny{0}};

\draw [fill] ( 8.00, 3.00) circle;
\node [left] at  ( 8.00, 3.00) {\tiny{1}};

\draw [fill] ( 9.07, 2.93) circle;
\node [left] at  ( 9.07, 2.93) {\tiny{2}};

\draw [fill] ( 8.93, 3.07) circle;
\node [left] at  ( 8.93, 3.07) {\tiny{3}};

\draw [fill] (10.00, 3.00) circle;
\node [left] at  (10.00, 3.00) {\tiny{4}};

\draw [fill] (15.00, 3.00) circle;
\node [left] at  (15.00, 3.00) {\tiny{5}};

\draw [fill] ( 8.00, 4.00) circle;
\node [left] at  ( 8.00, 4.00) {\tiny{0}};

\draw [fill] ( 9.00, 4.00) circle;
\node [left] at  ( 9.00, 4.00) {\tiny{1}};

\draw [fill] (10.00, 4.00) circle;
\node [left] at  (10.00, 4.00) {\tiny{2}};

\draw [fill] (14.00, 4.00) circle;
\node [left] at  (14.00, 4.00) {\tiny{3}};

\draw [fill] ( 9.00, 5.00) circle;
\node [left] at  ( 9.00, 5.00) {\tiny{0}};

\draw [fill] (10.00, 5.00) circle;
\node [left] at  (10.00, 5.00) {\tiny{1}};

\draw [fill] (11.00, 5.00) circle;
\node [left] at  (11.00, 5.00) {\tiny{2}};

\draw [fill] (15.00, 5.00) circle;
\node [left] at  (15.00, 5.00) {\tiny{3}};

\draw [fill] (16.00, 5.00) circle;
\node [left] at  (16.00, 5.00) {\tiny{4}};

\draw [fill] (10.00, 6.00) circle;
\node [left] at  (10.00, 6.00) {\tiny{0}};

\draw [fill] (11.00, 6.00) circle;
\node [left] at  (11.00, 6.00) {\tiny{1}};

\draw [fill] (15.00, 6.00) circle;
\node [left] at  (15.00, 6.00) {\tiny{2}};

\draw [fill] (16.00, 6.00) circle;
\node [left] at  (16.00, 6.00) {\tiny{3}};

\draw [fill] (12.00, 7.00) circle;
\node [left] at  (12.00, 7.00) {\tiny{0}};

\draw [fill] (16.00, 7.00) circle;
\node [left] at  (16.00, 7.00) {\tiny{1}};

\draw [fill] (16.00, 8.00) circle;
\node [left] at  (16.00, 8.00) {\tiny{0}};


\draw ( 1.00, 1.00) --( 0.00, 0.00);

\draw [dashed]  ( 3.00, 1.00) --( 0.00, 0.00);

\draw ( 2.00, 2.00) --( 1.00, 1.00);

\draw ( 2.00, 2.00) --( 2.00, 1.00);

\draw ( 3.00, 2.00) --( 2.00, 1.00);

\draw [dashed]  ( 6.00, 2.00) --( 3.00, 1.00);

\draw ( 8.00, 2.00) --( 7.00, 1.00);

\draw (16.00, 2.00) --(15.00, 1.00);

\draw [dashed]  (18.00, 2.00) --(15.00, 1.00);

\draw ( 4.00, 3.00) --( 3.00, 2.00);

\draw ( 8.00, 3.00) --( 7.00, 2.00);

\draw [dashed]  ( 8.93, 3.07) --( 6.00, 2.00);

\draw ( 8.93, 3.07) --( 8.00, 2.00);

\draw ( 8.93, 3.07) --( 9.00, 2.00);

\draw [dashed]  (10.00, 3.00) --( 7.00, 2.00);

\draw (10.00, 3.00) --( 9.00, 2.00);

\draw (17.00, 3.00) --(16.00, 2.00);

\draw ( 9.00, 4.00) --( 8.00, 3.00);

\draw ( 9.00, 4.00) --( 9.07, 2.93);

\draw (10.00, 4.00) --( 9.07, 2.93);

\draw [dashed]  (18.00, 4.00) --(15.00, 3.00);

\draw ( 9.00, 5.00) --( 8.00, 4.00);

\draw [dashed]  (11.00, 5.00) --( 8.00, 4.00);

\draw (11.00, 5.00) --(10.00, 4.00);

\draw (15.00, 5.00) --(14.00, 4.00);

\draw (10.00, 6.00) --( 9.00, 5.00);

\draw (10.00, 6.00) --(10.00, 5.00);

\draw (11.00, 6.00) --(10.00, 5.00);

\draw (16.00, 6.00) --(15.00, 5.00);

\draw (16.00, 6.00) --(16.00, 5.00);

\draw (17.00, 6.00) --(16.00, 5.00);

\draw (12.00, 7.00) --(11.00, 6.00);

\draw (16.00, 7.00) --(15.00, 6.00);
\end{tikzpicture}
\quad
\begin{tikzpicture}[scale=0.7]

\tikzset{radius= 0.05}

\node at (0,-0.5) {$4m$};

\foreach \n in {2,4}
{
\def\nn{\n-0};
\node at (\nn,-0.5) {$4m+\n$};
}

\clip (-1.0,-1.0) rectangle ( 5.5,  5.5);
\draw[color=lightgray] (0,0) grid [step=2]  (16,8);

\foreach \s in {0,2,4,6}
{
\def\ss{\s-0};
\node at (-0.5,\ss) {$\s$};
}


\draw [fill] ( 1.00, 0.00) circle;
\node [left] at  ( 1.00, 0.00) {\tiny{0}};

\draw [fill] ( 2.00, 1.00) circle;
\node [left] at  ( 2.00, 1.00) {\tiny{0}};

\draw [fill] ( 4.00, 1.00) circle;
\node [left] at  ( 4.00, 1.00) {\tiny{1}};

\draw [fill] ( 7.00, 1.00) circle;
\node [left] at  ( 7.00, 1.00) {\tiny{2}};

\draw [fill] ( 8.00, 1.00) circle;
\node [left] at  ( 8.00, 1.00) {\tiny{3}};

\draw [fill] (16.00, 1.00) circle;
\node [left] at  (16.00, 1.00) {\tiny{4}};

\draw [fill] ( 4.00, 2.00) circle;
\node [left] at  ( 4.00, 2.00) {\tiny{0}};

\draw [fill] ( 7.07, 1.93) circle;
\node [left] at  ( 7.07, 1.93) {\tiny{1}};

\draw [fill] ( 6.93, 2.07) circle;
\node [left] at  ( 6.93, 2.07) {\tiny{2}};

\draw [fill] ( 8.00, 2.00) circle;
\node [left] at  ( 8.00, 2.00) {\tiny{3}};

\draw [fill] ( 9.00, 2.00) circle;
\node [left] at  ( 9.00, 2.00) {\tiny{4}};

\draw [fill] (10.00, 2.00) circle;
\node [left] at  (10.00, 2.00) {\tiny{5}};

\draw [fill] (15.00, 2.00) circle;
\node [left] at  (15.00, 2.00) {\tiny{6}};

\draw [fill] ( 4.00, 3.00) circle;
\node [left] at  ( 4.00, 3.00) {\tiny{0}};

\draw [fill] ( 7.00, 3.00) circle;
\node [left] at  ( 7.00, 3.00) {\tiny{1}};

\draw [fill] ( 9.00, 3.00) circle;
\node [left] at  ( 9.00, 3.00) {\tiny{2}};

\draw [fill] (10.00, 3.00) circle;
\node [left] at  (10.00, 3.00) {\tiny{3}};

\draw [fill] (11.00, 3.00) circle;
\node [left] at  (11.00, 3.00) {\tiny{4}};

\draw [fill] (13.00, 3.00) circle;
\node [left] at  (13.00, 3.00) {\tiny{5}};

\draw [fill] (16.00, 3.00) circle;
\node [left] at  (16.00, 3.00) {\tiny{6}};

\draw [fill] ( 4.00, 4.00) circle;
\node [left] at  ( 4.00, 4.00) {\tiny{0}};

\draw [fill] ( 9.00, 4.00) circle;
\node [left] at  ( 9.00, 4.00) {\tiny{1}};

\draw [fill] (10.00, 4.00) circle;
\node [left] at  (10.00, 4.00) {\tiny{2}};

\draw [fill] (11.00, 4.00) circle;
\node [left] at  (11.00, 4.00) {\tiny{3}};

\draw [fill] (15.00, 4.00) circle;
\node [left] at  (15.00, 4.00) {\tiny{4}};

\draw [fill] ( 4.00, 5.00) circle;
\node [left] at  ( 4.00, 5.00) {\tiny{0}};

\draw [fill] (10.00, 5.00) circle;
\node [left] at  (10.00, 5.00) {\tiny{1}};

\draw [fill] (11.00, 5.00) circle;
\node [left] at  (11.00, 5.00) {\tiny{2}};

\draw [fill] (15.00, 5.00) circle;
\node [left] at  (15.00, 5.00) {\tiny{3}};

\draw [fill] (16.00, 5.00) circle;
\node [left] at  (16.00, 5.00) {\tiny{4}};

\draw [fill] ( 4.00, 6.00) circle;
\node [left] at  ( 4.00, 6.00) {\tiny{0}};

\draw [fill] (11.00, 6.00) circle;
\node [left] at  (11.00, 6.00) {\tiny{1}};

\draw [fill] (15.00, 6.00) circle;
\node [left] at  (15.00, 6.00) {\tiny{2}};

\draw [fill] (16.00, 6.00) circle;
\node [left] at  (16.00, 6.00) {\tiny{3}};

\draw [fill] ( 4.00, 7.00) circle;
\node [left] at  ( 4.00, 7.00) {\tiny{0}};

\draw [fill] (15.00, 7.00) circle;
\node [left] at  (15.00, 7.00) {\tiny{1}};

\draw [fill] ( 4.00, 8.00) circle;
\node [left] at  ( 4.00, 8.00) {\tiny{0}};


\draw ( 2.00, 1.00) --( 1.00, 0.00);

\draw [dashed]  ( 4.00, 1.00) --( 1.00, 0.00);

\draw [dashed]  ( 7.07, 1.93) --( 4.00, 1.00);

\draw [dashed]  ( 6.93, 2.07) --( 4.00, 1.00);

\draw ( 6.93, 2.07) --( 7.00, 1.00);

\draw ( 8.00, 2.00) --( 7.00, 1.00);

\draw ( 9.00, 2.00) --( 8.00, 1.00);

\draw [dashed]  (10.00, 2.00) --( 7.00, 1.00);

\draw (17.00, 2.00) --(16.00, 1.00);

\draw [dashed]  (19.00, 2.00) --(16.00, 1.00);

\draw ( 4.00, 3.00) --( 4.00, 2.00);

\draw [dashed]  ( 7.00, 3.00) --( 4.00, 2.00);

\draw ( 7.00, 3.00) --( 7.07, 1.93);

\draw ( 7.00, 3.00) --( 6.93, 2.07);

\draw ( 9.00, 3.00) --( 8.00, 2.00);

\draw [dashed]  (10.00, 3.00) --( 7.07, 1.93);

\draw [dashed]  (10.00, 3.00) --( 6.93, 2.07);

\draw (10.00, 3.00) --(10.00, 2.00);

\draw [dashed]  (13.00, 3.00) --(10.00, 2.00);

\draw ( 4.00, 4.00) --( 4.00, 3.00);

\draw [dashed]  (10.00, 4.00) --( 7.00, 3.00);

\draw (10.00, 4.00) --( 9.00, 3.00);

\draw (10.00, 4.00) --(10.00, 3.00);

\draw (11.00, 4.00) --(11.00, 3.00);

\draw [dashed]  (19.07, 3.93) --(16.00, 3.00);

\draw ( 4.00, 5.00) --( 4.00, 4.00);

\draw (10.00, 5.00) --( 9.00, 4.00);

\draw (11.00, 5.00) --(11.00, 4.00);

\draw (16.00, 5.00) --(15.00, 4.00);

\draw ( 4.00, 6.00) --( 4.00, 5.00);

\draw (11.00, 6.00) --(10.00, 5.00);

\draw (11.00, 6.00) --(11.00, 5.00);

\draw (15.00, 6.00) --(15.00, 5.00);

\draw (16.00, 6.00) --(15.00, 5.00);

\draw [dashed]  (18.00, 6.00) --(15.00, 5.00);

\draw ( 4.00, 7.00) --( 4.00, 6.00);

\draw (15.00, 7.00) --(15.00, 6.00);

\draw (17.00, 7.00) --(16.00, 6.00);

\draw ( 4.00, 8.00) --( 4.00, 7.00);

\draw ( 4.00, 9.00) --( 4.00, 8.00);
\end{tikzpicture}
\quad
\begin{tikzpicture}[scale=0.7]

\tikzset{radius= 0.05}

\foreach \n in {0,2,4}
{
\def\nn{\n+7};
\node at (\nn,-0.5) {$\n$};
}

\foreach \s in {0,2,4}
{
\def\ss{\s-0};
\node at (6.5,\ss) {$\s$};
}

\clip (6.5,-1.0) rectangle ( 12.5,  5.5);
\begin{scope}[xshift=1cm]
\draw[color=lightgray] (5,0) grid [step=2]  (15,8);
\end{scope}


\draw [fill] ( 7.00, 0.00) circle;
\node [left] at  ( 7.00, 0.00) {\tiny{0}};

\draw [fill] ( 7.00, 1.00) circle;
\node [left] at  ( 7.00, 1.00) {\tiny{0}};

\draw [fill] ( 9.00, 1.00) circle;
\node [left] at  ( 9.00, 1.00) {\tiny{1}};

\draw [fill] (10.00, 1.00) circle;
\node [left] at  (10.00, 1.00) {\tiny{2}};

\draw [fill] (12.00, 1.00) circle;
\node [left] at  (12.00, 1.00) {\tiny{3}};

\draw [fill] (14.00, 1.00) circle;
\node [left] at  (14.00, 1.00) {\tiny{4}};

\draw [fill] ( 7.00, 2.00) circle;
\node [left] at  ( 7.00, 2.00) {\tiny{0}};

\draw [fill] ( 9.00, 2.00) circle;
\node [left] at  ( 9.00, 2.00) {\tiny{1}};

\draw [fill] (10.00, 2.00) circle;
\node [left] at  (10.00, 2.00) {\tiny{2}};

\draw [fill] (12.00, 2.00) circle;
\node [left] at  (12.00, 2.00) {\tiny{3}};

\draw [fill] (13.00, 2.00) circle;
\node [left] at  (13.00, 2.00) {\tiny{4}};

\draw [fill] (14.00, 2.00) circle;
\node [left] at  (14.00, 2.00) {\tiny{5}};

\draw [fill] (15.00, 2.00) circle;
\node [left] at  (15.00, 2.00) {\tiny{6}};

\draw [fill] (16.00, 2.00) circle;
\node [left] at  (16.00, 2.00) {\tiny{7}};

\draw [fill] ( 7.00, 3.00) circle;
\node [left] at  ( 7.00, 3.00) {\tiny{0}};

\draw [fill] ( 9.00, 3.00) circle;
\node [left] at  ( 9.00, 3.00) {\tiny{1}};

\draw [fill] (12.00, 3.00) circle;
\node [left] at  (12.00, 3.00) {\tiny{2}};

\draw [fill] (14.00, 3.00) circle;
\node [left] at  (14.00, 3.00) {\tiny{3}};

\draw [fill] (15.00, 3.00) circle;
\node [left] at  (15.00, 3.00) {\tiny{4}};

\draw [fill] (16.00, 3.00) circle;
\node [left] at  (16.00, 3.00) {\tiny{5}};

\draw [fill] ( 7.00, 4.00) circle;
\node [left] at  ( 7.00, 4.00) {\tiny{0}};

\draw [fill] ( 9.00, 4.00) circle;
\node [left] at  ( 9.00, 4.00) {\tiny{1}};

\draw [fill] (14.00, 4.00) circle;
\node [left] at  (14.00, 4.00) {\tiny{2}};

\draw [fill] (16.00, 4.00) circle;
\node [left] at  (16.00, 4.00) {\tiny{3}};

\draw [fill] ( 7.00, 5.00) circle;
\node [left] at  ( 7.00, 5.00) {\tiny{0}};

\draw [fill] ( 9.00, 5.00) circle;
\node [left] at  ( 9.00, 5.00) {\tiny{1}};

\draw [fill] (16.00, 5.00) circle;
\node [left] at  (16.00, 5.00) {\tiny{2}};

\draw [fill] ( 7.00, 6.00) circle;
\node [left] at  ( 7.00, 6.00) {\tiny{0}};

\draw [fill] ( 9.00, 6.00) circle;
\node [left] at  ( 9.00, 6.00) {\tiny{1}};

\draw [fill] ( 7.00, 7.00) circle;
\node [left] at  ( 7.00, 7.00) {\tiny{0}};

\draw [fill] ( 9.00, 7.00) circle;
\node [left] at  ( 9.00, 7.00) {\tiny{1}};

\draw [fill] ( 7.00, 8.00) circle;
\node [left] at  ( 7.00, 8.00) {\tiny{0}};

\draw [fill] ( 9.00, 8.00) circle;
\node [left] at  ( 9.00, 8.00) {\tiny{1}};


\draw ( 7.00, 1.00) --( 7.00, 0.00);

\draw [dashed]  (10.00, 1.00) --( 7.00, 0.00);

\draw ( 7.00, 2.00) --( 7.00, 1.00);

\draw ( 9.00, 2.00) --( 9.00, 1.00);

\draw [dashed]  (10.00, 2.00) --( 7.00, 1.00);

\draw (10.00, 2.00) --( 9.00, 1.00);

\draw (10.00, 2.00) --(10.00, 1.00);

\draw [dashed]  (12.00, 2.00) --( 9.00, 1.00);

\draw (12.00, 2.00) --(12.00, 1.00);

\draw [dashed]  (13.00, 2.00) --(10.00, 1.00);

\draw (13.00, 2.00) --(12.00, 1.00);

\draw (14.00, 2.00) --(14.00, 1.00);

\draw [dashed]  (15.00, 2.00) --(12.00, 1.00);

\draw ( 7.00, 3.00) --( 7.00, 2.00);

\draw ( 9.00, 3.00) --( 9.00, 2.00);

\draw [dashed]  (12.00, 3.00) --( 9.00, 2.00);

\draw (12.00, 3.00) --(12.00, 2.00);

\draw (14.00, 3.00) --(14.00, 2.00);

\draw [dashed]  (15.00, 3.00) --(12.00, 2.00);

\draw (15.00, 3.00) --(15.00, 2.00);

\draw (16.00, 3.00) --(16.00, 2.00);

\draw [dashed]  (18.00, 3.00) --(15.00, 2.00);

\draw ( 7.00, 4.00) --( 7.00, 3.00);

\draw ( 9.00, 4.00) --( 9.00, 3.00);

\draw (14.00, 4.00) --(14.00, 3.00);

\draw (16.00, 4.00) --(16.00, 3.00);

\draw ( 7.00, 5.00) --( 7.00, 4.00);

\draw ( 9.00, 5.00) --( 9.00, 4.00);

\draw (16.00, 5.00) --(16.00, 4.00);

\draw ( 7.00, 6.00) --( 7.00, 5.00);

\draw ( 9.00, 6.00) --( 9.00, 5.00);

\draw ( 7.00, 7.00) --( 7.00, 6.00);

\draw ( 9.00, 7.00) --( 9.00, 6.00);

\draw ( 7.00, 8.00) --( 7.00, 7.00);

\draw ( 9.00, 8.00) --( 9.00, 7.00);

\draw ( 7.00, 9.00) --( 7.00, 8.00);

\draw ( 9.00, 9.00) --( 9.00, 8.00);
\end{tikzpicture}

    \caption{Adams $E_2$-pages for the mod-2 Moore spectrum $S^0/2$ (left), the space $C_{pf}$ appearing in the proof of \cref{lem: 4m+4-cell attaching map} (middle), and the cofibre of $\eta$ (right).}
    \label{fig:mod-2-Moore and hook}
\end{figure}
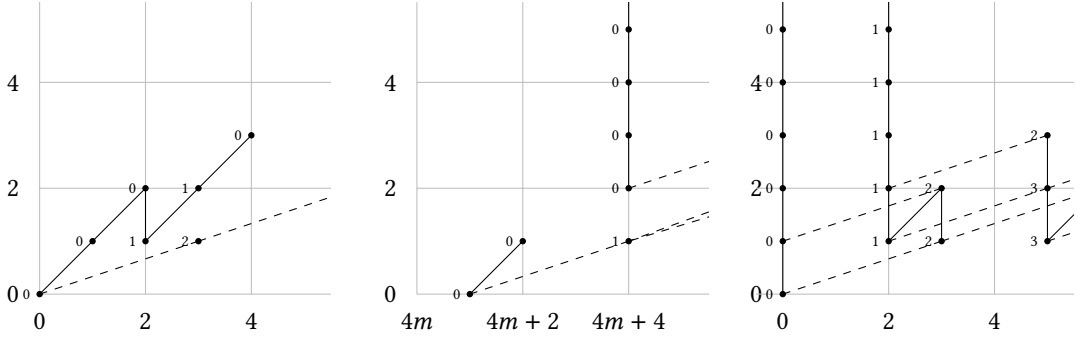

The following is immediate from the $E_2$-page for the mod-2 Moore spectrum $S^0/2$, see \cref{fig:mod-2-Moore and hook}-left.

\begin{lem} \label{lem: v1} A basis for $\pi_2(S^0/2)$ is given by
    $$\pi_2(S^0/2)_{(2)} \cong \bb{Z}/4 \{v_1\}$$
    for a class $v_1$ such that $ 2v_1 \simeq \iota \eta^2$, where $\iota: S^0 \to S^0/2$ is the inclusion of the bottom cell. \qed
\end{lem}

Combining \cref{lem: v1} with \cref{lem: splitting RP4m+3} gives the following immediate corollary.

\begin{cor} \label{cor: pi 4m+3}  A basis for $\pi_{4m+3}(\bb{RP}^{4m+3}_{4m})_{(2)}$ is given by
\[
\pushQED{\qed} 
\pi_{4m+3}(\bb{RP}^{4m+3}_{4m})_{(2)} \cong \bb{Z}/8 \{i \nu\} \oplus \bb{Z}/4\{jv_1\} \oplus \bb{Z}_{(2)}\{k\}.\qedhere
\popQED
\] 
\end{cor}

The attaching map of \eqref{eq: attaching 4m+3} is an element of this group, which we now determine.

\begin{lem} \label{lem: 4m+4-cell attaching map}
Let $f \in \pi_{4m+3}(\bb{RP}^{4m+3}_{4m})_{(2)}$ denote the attaching map of \eqref{eq: attaching 4m+3} and (using \cref{cor: pi 4m+3}) write $f = a \cdot (i \nu) + b \cdot (j v_1) + c \cdot k$, where $a \in \bb{Z}/8$, $b \in \bb{Z}/4$, and $c \in \bb{Z}_{(2)}$. Then:
\begin{enumerate}
    \item $c = 2 u$ for some unit $u \in \bb{Z}_{(2)}$,
    \item $b$ is not divisible by 2 (i.e.~it generates $\bb{Z}/4$), and
    \item $a$ is a generator of $\bb{Z}/8$ when $m$ is odd, and is zero when $m$ is even.
\end{enumerate}
\end{lem}

\begin{proof} We work $2$-locally throughout the proof. The structure of $H^*(\bb{RP}^{4m+4}_{4m}; \bb{F}_2)$ as a module over the Steenrod algebra is:
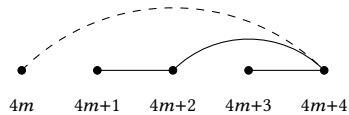
\begin{center}
        \begin{tikzpicture}
        \draw[fill=black]  (0,0) circle [radius=0.05];
            \draw (0,-0.25)  node[below] {$\scriptstyle{4m}$};
        \foreach \x in {1,2,3,4}{
            \draw[fill=black]  (\x,0) circle [radius=0.05];
            \draw (\x,-0.25)  node[below] {$\scriptstyle{4m+\x}$};
        }
        \draw (1,0)--(2,0);
        \draw (3,0)--(4,0);
        \draw (2,0) to[in =135, out =45] (4,0);
        \draw[dashed] (0,0) to[in =135, out =45] (4,0);
    \end{tikzpicture}
    \end{center}
where the dashed $\mathrm{Sq}^4$ is present if and only if $m$ is odd. Recall the splitting of~\cref{lem: splitting RP4m+3}, and write $q : \bb{RP}^{4m+3}_{4m} \to S^{4m+3}$ for the projection, so that 
$q \circ k \simeq \mathrm{id}_{S^{4m+3}}$, $qi \simeq *,$ and $qj \simeq *$.
Then 
\[
qf \simeq q(c \cdot k) \simeq c (qk) \simeq c \cdot \mathrm{id}_{S^{4m+3}}
\]
and the cofibre of $qf$ is $\bb{RP}^{4m+4}_{4m+3} \simeq S^{4m+3}/2$, so (1) follows.

For (2), write $p : \bb{RP}^{4m+3}_{4m} \to S^{4m+1}/2$ for the projection, so that
$pj \simeq \mathrm{id}_{S^{4m+1}/2}$, $pk \simeq *$, and $pi \simeq *.$
Consider the composite $pf \simeq b \cdot (v_1)$ using the diagram
\[\begin{tikzcd}
S^{4m+3} & \bb{RP}^{4m+3}_{4m} & \bb{RP}^{4m+4}_{4m}\\
S^{4m+3} & S^{4m+1}/2 & C_{pf} 
	\arrow[equal,from=1-1, to=2-1]
    \arrow["f",from=1-1, to=1-2]
    \arrow[from=1-2, to=1-3]
    \arrow["p",from=1-2, to=2-2]
    \arrow[from=1-3, to=2-3]
    \arrow["pf",from=2-1, to=2-2]
    \arrow[from=2-2, to=2-3]
\end{tikzcd}\]
where the rows are cofibre sequences. As a module over the Steenrod algebra, $H^*(C_{pf}; \bb{F}_2)$ is the submodule of $H^*(\bb{RP}^{4m+4}_{4m}; \bb{F}_2)$ consisting of degrees $4m+1,4m+2,$ and $4m+4$ (a `sickle' shape). Consulting the Adams chart for $C_{pf}$ in \cref{fig:mod-2-Moore and hook}-middle, we see that $\pi_{4m+3}(C_{pf})=0$, so $pf \simeq b \cdot (v_1)$ must be a surjection onto $\pi_{4m+3}(S^{4m+1}/2) \cong \bb{Z}/4\{v_1\}$ (\cref{lem: v1}). This implies that $b$ must be a unit mod 4, proving (2).

For (3), proceed similarly, this time considering the two-cell complex $C$ obtained as the cofibre of the projection of $f$ to the bottom cell $S^{4m}$. We see that $a$ is a generator of $\bb{Z}/8$ if and only if the projection of $f$ is an odd multiple of $\nu$, if and only if $C \simeq \Sigma^{4m} C_{\nu}$, if and only if the two cells of $C$ are connected by a $\mathrm{Sq}^4$, which we saw above happens if and only if $m$ is odd.

If $m$ is even, then James periodicity (see e.g.,~\cite{MahowaldOnJames}) says that $\bb{RP}^{4m+4}_{4m} \cong \Sigma^{4m} \bb{RP}^4_+$, and in particular that the bottom cell splits off, so the attaching map $f$ must have first component $a=0$.
\end{proof}

As a corollary, we can in each case write down a basis for $\pi_{4m+3}(\bb{RP}^{\infty}_{4m})_{(2)}$, the group we actually care about. Note that the following description is precisely what we needed to resolve the truncation out of codimension $3$ in the proof of~\cref{lem: sphere diagrams 2 odd}.

\begin{cor} \label{cor: codimension 3 bases} A basis for the $2$-local stable homotopy group $\pi_{4m+3}(\bb{RP}^{\infty}_{4m})_{(2)}$ is given by
    $$\pi_{4m+3}(\bb{RP}^{\infty}_{4m})_{(2)} \cong \begin{cases}
        \bb{Z}/8\{k\} \oplus \bb{Z}/8\{i \nu \} & m \textrm{ even} \\
        \bb{Z}/16\{k\} \oplus \bb{Z}/4\{j v_1\} & m \textrm{ odd.}
    \end{cases}$$
Furthermore, in the second case $jv_1 = a\cdot(i\nu) - (2u) \cdot k$ where $u$ and $a$ are $2$-local units. \qed
\end{cor}

We are now ready to resolve codimension $3$ realification for $n=2m+1$ odd.

\begin{proof}[Proof of \cref{prop: difficult map 1}]
Again, we work $2$-locally. We are considering the $(4m+3)$-rd homotopy groups, so we may restrict to $(4m+4)$-skeleta. Since $\Sigma \bb{CP}^\infty_{2m}$ has no even-dimensional cells, this restriction factors as
$$\Sigma \bb{CP}^{2m+1}_{2m} \xrightarrow{ \ r\ } \bb{RP}^{4m+3}_{4m} \hookrightarrow  \bb{RP}^{4m+4}_{4m}.$$
Consider the first map in this composite. \cref{lem: splitting RP4m+3} says that we have (2-local) splittings
$$j' \oplus k' : S^{4m+1} \oplus S^{4m+3} \xrightarrow{ \ \simeq \ } \Sigma \bb{CP}^{2m+1}_{2m},$$
and
$$i \oplus j \oplus k: S^{4m} \oplus S^{4m+1}/2 \oplus S^{4m+3} \xrightarrow{\ \simeq \  } \bb{RP}^{4m+3}_{4m},$$
such that $rj' \simeq u \cdot j \iota$ and $r k' = v \cdot k$, where $u$ and $v$ are units and $\iota$ generates $\pi_{4m+1}(S^{4m+1}/2) \cong \bb{Z}/2$. It follows (c.f.~\cref{cor: pi 4m+3}) that this first map $$\pi_{4m+3}(\bb{CP}^{2m+1}_{2m}) \cong \bb{Z}\{k'\}\oplus \bb{Z}/2\{j' \circ \eta^2\} \xrightarrow{\ r_* \ } \pi_{4m+3}(\bb{RP}^{4m+3}_{4m}) \cong \bb{Z}/8 \{i \nu\} \oplus \bb{Z}/4\{jv_1\} \oplus \bb{Z}\{k\}$$
acts on the first generator as $r_* k' = v \cdot k$, and on the second as $r_* (j'\eta^2) = r j'\eta^2 = u \cdot j \iota \eta^2 = 2 \cdot j v_1$, where the equation $ u \cdot \iota \eta^2 = 2 v_1$ holds by \cref{lem: v1}.

Now consider the second map. From the cofibre sequence \eqref{eq: attaching 4m+3} we see that
$$\pi_{4m+3}(\bb{RP}^{4m+4}_{4m}) \cong \pi_{4m+3}(\bb{RP}^{4m+3}_{4m})/\langle f\rangle \cong (\bb{Z}/8 \{i \nu\} \oplus \bb{Z}/4\{jv_1\} \oplus \bb{Z}\{k\}) / \langle f\rangle.$$
We computed $f$ in \cref{lem: 4m+4-cell attaching map}: $f = a \cdot (i \nu) + b \cdot (j v_1) + c \cdot k$, where $c$ is 2 times a generator of $\bb{Z}$, $b$ is a generator of $\bb{Z}/4$, and $a$ is either zero (when $m$ is even), or a generator of $\bb{Z}/8$ (when $m$ is odd). In the first case, $2f = 2 \cdot (j v_1) - 2c \cdot k$, so the quotient map identifies $2 j v_1 = r_*(j'\eta^2)$ with $4 r_*(k') = 4v \cdot k$. In the second case, the quotient does not introduce any linear dependence between these two classes, and the result is in any case immediate in this case from our choice of basis in \cref{cor: codimension 3 bases}.
\end{proof}

\subsubsection{The case of codimension $3$ real bundles with $n$ even}\label{subsection: codim 3 even}

In this section we will set $n = 2m$ and complete the proof of~\cref{lem: sphere diagrams 2 odd} by considering
$$r_*: \pi_{4m+1}(\Sigma \bb{CP}^{\infty}_{2m-1})_{(2)} \longrightarrow \pi_{4m+1}( \bb{RP}^{\infty}_{4m-2})_{(2)}.$$

When $m$ is odd, $4m-2$ is of the form $8k+2$, and the map is the degree $8k+5$ component \cref{fig:SCPASS1}-right to \cref{fig:RPASS2}-left. When $m$ is even, $4m-2$ is of the form $8k+6$, and the map is the degree $8k+8$ component \cref{fig:SCPASS2}-right to \cref{fig:RPASS4}-left. In either case, the domain is $\bb{Z}$ and the codomain is $\bb{Z}/4$. The result of this section is then as follows, and uses the following two additional Adams charts.

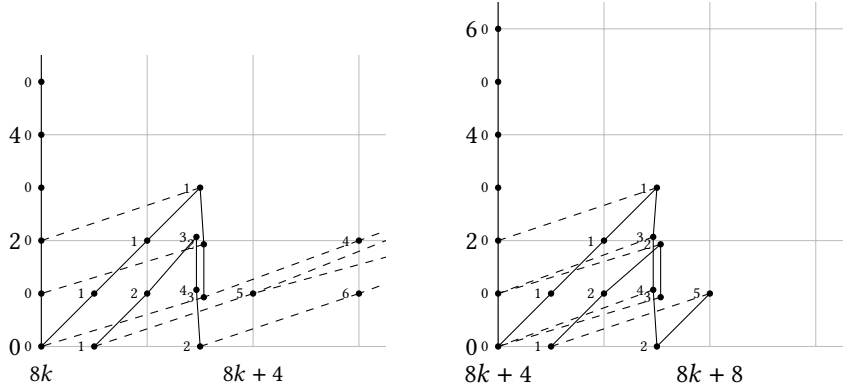
\begin{figure}[ht]

\begin{tikzpicture}[scale=0.7]

\tikzset{radius= 0.05}

\clip (-1.0,-1.0) rectangle ( 6.5,  5.5);
\draw[color=lightgray] (0,0) grid [step=2]  (16,8);

\node at (0,-0.5) {\small{$8k$}};
\node at (4,-0.5) {\small{$8k+4$}};

\foreach \s in {0,2,4,6}
{
\def\ss{\s-0};
\node at (-0.5,\ss) {$\s$};
}


\draw [fill] ( 0.00, 0.00) circle;
\node [left] at  ( 0.00, 0.00) {\tiny{0}};

\draw [fill] ( 1.00, 0.00) circle;
\node [left] at  ( 1.00, 0.00) {\tiny{1}};

\draw [fill] ( 3.00, 0.00) circle;
\node [left] at  ( 3.00, 0.00) {\tiny{2}};

\draw [fill] ( 7.00, 0.00) circle;
\node [left] at  ( 7.00, 0.00) {\tiny{3}};

\draw [fill] (15.00, 0.00) circle;
\node [left] at  (15.00, 0.00) {\tiny{4}};

\draw [fill] ( 0.00, 1.00) circle;
\node [left] at  ( 0.00, 1.00) {\tiny{0}};

\draw [fill] ( 1.00, 1.00) circle;
\node [left] at  ( 1.00, 1.00) {\tiny{1}};

\draw [fill] ( 2.00, 1.00) circle;
\node [left] at  ( 2.00, 1.00) {\tiny{2}};

\draw [fill] ( 3.07, 0.93) circle;
\node [left] at  ( 3.07, 0.93) {\tiny{3}};

\draw [fill] ( 2.93, 1.07) circle;
\node [left] at  ( 2.93, 1.07) {\tiny{4}};

\draw [fill] ( 4.00, 1.00) circle;
\node [left] at  ( 4.00, 1.00) {\tiny{5}};

\draw [fill] ( 6.00, 1.00) circle;
\node [left] at  ( 6.00, 1.00) {\tiny{6}};

\draw [fill] ( 7.07, 0.93) circle;
\node [left] at  ( 7.07, 0.93) {\tiny{7}};

\draw [fill] ( 6.93, 1.07) circle;
\node [left] at  ( 6.93, 1.07) {\tiny{8}};

\draw [fill] ( 8.07, 0.93) circle;
\node [left] at  ( 8.07, 0.93) {\tiny{9}};

\draw [fill] ( 7.93, 1.07) circle;
\node [left] at  ( 7.93, 1.07) {\tiny{10}};

\draw [fill] (10.00, 1.00) circle;
\node [left] at  (10.00, 1.00) {\tiny{11}};

\draw [fill] (14.00, 1.00) circle;
\node [left] at  (14.00, 1.00) {\tiny{12}};

\draw [fill] (15.07, 0.93) circle;
\node [left] at  (15.07, 0.93) {\tiny{13}};

\draw [fill] (14.93, 1.07) circle;
\node [left] at  (14.93, 1.07) {\tiny{14}};

\draw [fill] (16.07, 0.93) circle;
\node [left] at  (16.07, 0.93) {\tiny{15}};

\draw [fill] (15.93, 1.07) circle;
\node [left] at  (15.93, 1.07) {\tiny{16}};

\draw [fill] ( 0.00, 2.00) circle;
\node [left] at  ( 0.00, 2.00) {\tiny{0}};

\draw [fill] ( 2.00, 2.00) circle;
\node [left] at  ( 2.00, 2.00) {\tiny{1}};

\draw [fill] ( 3.07, 1.93) circle;
\node [left] at  ( 3.07, 1.93) {\tiny{2}};

\draw [fill] ( 2.93, 2.07) circle;
\node [left] at  ( 2.93, 2.07) {\tiny{3}};

\draw [fill] ( 6.00, 2.00) circle;
\node [left] at  ( 6.00, 2.00) {\tiny{4}};

\draw [fill] ( 7.14, 1.86) circle;
\node [left] at  ( 7.14, 1.86) {\tiny{5}};

\draw [fill] ( 7.00, 2.00) circle;
\node [left] at  ( 7.00, 2.00) {\tiny{6}};

\draw [fill] ( 6.86, 2.14) circle;
\node [left] at  ( 6.86, 2.14) {\tiny{7}};

\draw [fill] ( 8.07, 1.93) circle;
\node [left] at  ( 8.07, 1.93) {\tiny{8}};

\draw [fill] ( 7.93, 2.07) circle;
\node [left] at  ( 7.93, 2.07) {\tiny{9}};

\draw [fill] ( 9.07, 1.93) circle;
\node [left] at  ( 9.07, 1.93) {\tiny{10}};

\draw [fill] ( 8.93, 2.07) circle;
\node [left] at  ( 8.93, 2.07) {\tiny{11}};

\draw [fill] (10.00, 2.00) circle;
\node [left] at  (10.00, 2.00) {\tiny{12}};

\draw [fill] (14.07, 1.93) circle;
\node [left] at  (14.07, 1.93) {\tiny{13}};

\draw [fill] (13.93, 2.07) circle;
\node [left] at  (13.93, 2.07) {\tiny{14}};

\draw [fill] (15.14, 1.86) circle;
\node [left] at  (15.14, 1.86) {\tiny{15}};

\draw [fill] (15.00, 2.00) circle;
\node [left] at  (15.00, 2.00) {\tiny{16}};

\draw [fill] (14.86, 2.14) circle;
\node [left] at  (14.86, 2.14) {\tiny{17}};

\draw [fill] (16.00, 2.00) circle;
\node [left] at  (16.00, 2.00) {\tiny{18}};

\draw [fill] ( 0.00, 3.00) circle;
\node [left] at  ( 0.00, 3.00) {\tiny{0}};

\draw [fill] ( 3.00, 3.00) circle;
\node [left] at  ( 3.00, 3.00) {\tiny{1}};

\draw [fill] ( 7.07, 2.93) circle;
\node [left] at  ( 7.07, 2.93) {\tiny{2}};

\draw [fill] ( 6.93, 3.07) circle;
\node [left] at  ( 6.93, 3.07) {\tiny{3}};

\draw [fill] ( 8.00, 3.00) circle;
\node [left] at  ( 8.00, 3.00) {\tiny{4}};

\draw [fill] ( 9.07, 2.93) circle;
\node [left] at  ( 9.07, 2.93) {\tiny{5}};

\draw [fill] ( 8.93, 3.07) circle;
\node [left] at  ( 8.93, 3.07) {\tiny{6}};

\draw [fill] (10.00, 3.00) circle;
\node [left] at  (10.00, 3.00) {\tiny{7}};

\draw [fill] (14.07, 2.93) circle;
\node [left] at  (14.07, 2.93) {\tiny{8}};

\draw [fill] (13.93, 3.07) circle;
\node [left] at  (13.93, 3.07) {\tiny{9}};

\draw [fill] (15.07, 2.93) circle;
\node [left] at  (15.07, 2.93) {\tiny{10}};

\draw [fill] (14.93, 3.07) circle;
\node [left] at  (14.93, 3.07) {\tiny{11}};

\draw [fill] (16.00, 3.00) circle;
\node [left] at  (16.00, 3.00) {\tiny{12}};

\draw [fill] ( 0.00, 4.00) circle;
\node [left] at  ( 0.00, 4.00) {\tiny{0}};

\draw [fill] ( 7.00, 4.00) circle;
\node [left] at  ( 7.00, 4.00) {\tiny{1}};

\draw [fill] ( 9.07, 3.93) circle;
\node [left] at  ( 9.07, 3.93) {\tiny{2}};

\draw [fill] ( 8.93, 4.07) circle;
\node [left] at  ( 8.93, 4.07) {\tiny{3}};

\draw [fill] (11.00, 4.00) circle;
\node [left] at  (11.00, 4.00) {\tiny{4}};

\draw [fill] (14.07, 3.93) circle;
\node [left] at  (14.07, 3.93) {\tiny{5}};

\draw [fill] (13.93, 4.07) circle;
\node [left] at  (13.93, 4.07) {\tiny{6}};

\draw [fill] (15.14, 3.86) circle;
\node [left] at  (15.14, 3.86) {\tiny{7}};

\draw [fill] (15.00, 4.00) circle;
\node [left] at  (15.00, 4.00) {\tiny{8}};

\draw [fill] (14.86, 4.14) circle;
\node [left] at  (14.86, 4.14) {\tiny{9}};

\draw [fill] ( 0.00, 5.00) circle;
\node [left] at  ( 0.00, 5.00) {\tiny{0}};

\draw [fill] ( 9.00, 5.00) circle;
\node [left] at  ( 9.00, 5.00) {\tiny{1}};

\draw [fill] (10.00, 5.00) circle;
\node [left] at  (10.00, 5.00) {\tiny{2}};

\draw [fill] (11.07, 4.93) circle;
\node [left] at  (11.07, 4.93) {\tiny{3}};

\draw [fill] (10.93, 5.07) circle;
\node [left] at  (10.93, 5.07) {\tiny{4}};

\draw [fill] (14.07, 4.93) circle;
\node [left] at  (14.07, 4.93) {\tiny{5}};

\draw [fill] (13.93, 5.07) circle;
\node [left] at  (13.93, 5.07) {\tiny{6}};

\draw [fill] (15.14, 4.86) circle;
\node [left] at  (15.14, 4.86) {\tiny{7}};

\draw [fill] (15.00, 5.00) circle;
\node [left] at  (15.00, 5.00) {\tiny{8}};

\draw [fill] (14.86, 5.14) circle;
\node [left] at  (14.86, 5.14) {\tiny{9}};

\draw [fill] (16.00, 5.00) circle;
\node [left] at  (16.00, 5.00) {\tiny{10}};

\draw [fill] ( 0.00, 6.00) circle;
\node [left] at  ( 0.00, 6.00) {\tiny{0}};

\draw [fill] (10.00, 6.00) circle;
\node [left] at  (10.00, 6.00) {\tiny{1}};

\draw [fill] (11.07, 5.93) circle;
\node [left] at  (11.07, 5.93) {\tiny{2}};

\draw [fill] (10.93, 6.07) circle;
\node [left] at  (10.93, 6.07) {\tiny{3}};

\draw [fill] (14.00, 6.00) circle;
\node [left] at  (14.00, 6.00) {\tiny{4}};

\draw [fill] (15.07, 5.93) circle;
\node [left] at  (15.07, 5.93) {\tiny{5}};

\draw [fill] (14.93, 6.07) circle;
\node [left] at  (14.93, 6.07) {\tiny{6}};

\draw [fill] (16.07, 5.93) circle;
\node [left] at  (16.07, 5.93) {\tiny{7}};

\draw [fill] (15.93, 6.07) circle;
\node [left] at  (15.93, 6.07) {\tiny{8}};

\draw [fill] ( 0.00, 7.00) circle;
\node [left] at  ( 0.00, 7.00) {\tiny{0}};

\draw [fill] (11.00, 7.00) circle;
\node [left] at  (11.00, 7.00) {\tiny{1}};

\draw [fill] (15.07, 6.93) circle;
\node [left] at  (15.07, 6.93) {\tiny{2}};

\draw [fill] (14.93, 7.07) circle;
\node [left] at  (14.93, 7.07) {\tiny{3}};

\draw [fill] (16.00, 7.00) circle;
\node [left] at  (16.00, 7.00) {\tiny{4}};

\draw [fill] ( 0.00, 8.00) circle;
\node [left] at  ( 0.00, 8.00) {\tiny{0}};

\draw [fill] (15.00, 8.00) circle;
\node [left] at  (15.00, 8.00) {\tiny{1}};


\draw ( 0.00, 1.00) --( 0.00, 0.00);

\draw ( 1.00, 1.00) --( 0.00, 0.00);

\draw ( 2.00, 1.00) --( 1.00, 0.00);

\draw [dashed]  ( 3.07, 0.93) --( 0.00, 0.00);

\draw ( 2.93, 1.07) --( 3.00, 0.00);

\draw [dashed]  ( 4.00, 1.00) --( 1.00, 0.00);

\draw [dashed]  ( 6.00, 1.00) --( 3.00, 0.00);

\draw ( 6.93, 1.07) --( 7.00, 0.00);

\draw ( 7.93, 1.07) --( 7.00, 0.00);

\draw (14.93, 1.07) --(15.00, 0.00);

\draw (15.93, 1.07) --(15.00, 0.00);

\draw [dashed]  (17.93, 1.07) --(15.00, 0.00);

\draw ( 0.00, 2.00) --( 0.00, 1.00);

\draw ( 2.00, 2.00) --( 1.00, 1.00);

\draw [dashed]  ( 3.07, 1.93) --( 0.00, 1.00);

\draw ( 3.07, 1.93) --( 3.07, 0.93);

\draw ( 2.93, 2.07) --( 2.00, 1.00);

\draw ( 2.93, 2.07) --( 2.93, 1.07);

\draw [dashed]  ( 6.00, 2.00) --( 3.07, 0.93);

\draw [dashed]  ( 7.14, 1.86) --( 4.00, 1.00);

\draw ( 7.00, 2.00) --( 7.07, 0.93);

\draw [dashed]  ( 6.86, 2.14) --( 4.00, 1.00);

\draw ( 6.86, 2.14) --( 6.93, 1.07);

\draw ( 7.93, 2.07) --( 7.07, 0.93);

\draw ( 9.07, 1.93) --( 8.07, 0.93);

\draw [dashed]  ( 8.93, 2.07) --( 6.00, 1.00);

\draw ( 8.93, 2.07) --( 7.93, 1.07);

\draw (10.00, 2.00) --(10.00, 1.00);

\draw (13.93, 2.07) --(14.00, 1.00);

\draw (15.00, 2.00) --(15.07, 0.93);

\draw (14.86, 2.14) --(14.93, 1.07);

\draw (16.00, 2.00) --(15.07, 0.93);

\draw (17.00, 2.00) --(16.07, 0.93);

\draw (16.86, 2.14) --(15.93, 1.07);

\draw [dashed]  (18.14, 1.86) --(15.07, 0.93);

\draw [dashed]  (17.86, 2.14) --(14.93, 1.07);

\draw [dashed]  (18.93, 2.07) --(16.07, 0.93);

\draw ( 0.00, 3.00) --( 0.00, 2.00);

\draw [dashed]  ( 3.00, 3.00) --( 0.00, 2.00);

\draw ( 3.00, 3.00) --( 2.00, 2.00);

\draw ( 3.00, 3.00) --( 3.07, 1.93);

\draw ( 7.07, 2.93) --( 7.00, 2.00);

\draw ( 6.93, 3.07) --( 7.14, 1.86);

\draw ( 6.93, 3.07) --( 6.86, 2.14);

\draw ( 9.07, 2.93) --( 8.07, 1.93);

\draw [dashed]  ( 8.93, 3.07) --( 6.00, 2.00);

\draw ( 8.93, 3.07) --( 7.93, 2.07);

\draw [dashed]  (10.00, 3.00) --( 7.14, 1.86);

\draw (10.00, 3.00) --( 9.07, 1.93);

\draw (10.00, 3.00) --(10.00, 2.00);

\draw (13.93, 3.07) --(14.07, 1.93);

\draw (15.07, 2.93) --(15.00, 2.00);

\draw (14.93, 3.07) --(15.14, 1.86);

\draw (14.93, 3.07) --(14.86, 2.14);

\draw (16.93, 3.07) --(16.00, 2.00);

\draw [dashed]  (18.07, 2.93) --(15.00, 2.00);

\draw [dashed]  (17.79, 3.21) --(15.14, 1.86);

\draw [dashed]  (17.79, 3.21) --(14.86, 2.14);

\draw ( 0.00, 4.00) --( 0.00, 3.00);

\draw ( 7.00, 4.00) --( 7.07, 2.93);

\draw ( 8.93, 4.07) --( 8.00, 3.00);

\draw (13.93, 4.07) --(14.07, 2.93);

\draw (15.14, 3.86) --(14.07, 2.93);

\draw (15.00, 4.00) --(15.07, 2.93);

\draw (14.86, 4.14) --(14.93, 3.07);

\draw [dashed]  (16.93, 4.07) --(14.07, 2.93);

\draw [dashed]  (17.86, 4.14) --(15.07, 2.93);

\draw [dashed]  (19.00, 4.00) --(16.00, 3.00);

\draw ( 0.00, 5.00) --( 0.00, 4.00);

\draw (10.00, 5.00) --( 9.07, 3.93);

\draw (10.93, 5.07) --(11.00, 4.00);

\draw (14.07, 4.93) --(14.07, 3.93);

\draw [dashed]  (13.93, 5.07) --(11.00, 4.00);

\draw (13.93, 5.07) --(13.93, 4.07);

\draw (15.14, 4.86) --(14.07, 3.93);

\draw (15.00, 5.00) --(15.00, 4.00);

\draw (14.86, 5.14) --(14.86, 4.14);

\draw (16.00, 5.00) --(15.14, 3.86);

\draw [dashed]  (17.07, 4.93) --(14.07, 3.93);

\draw [dashed]  (16.93, 5.07) --(13.93, 4.07);

\draw ( 0.00, 6.00) --( 0.00, 5.00);

\draw (10.00, 6.00) --( 9.00, 5.00);

\draw (11.07, 5.93) --(11.07, 4.93);

\draw (10.93, 6.07) --(10.00, 5.00);

\draw (10.93, 6.07) --(10.93, 5.07);

\draw [dashed]  (14.00, 6.00) --(11.07, 4.93);

\draw (14.00, 6.00) --(14.07, 4.93);

\draw (15.07, 5.93) --(15.00, 5.00);

\draw (14.93, 6.07) --(14.86, 5.14);

\draw (15.93, 6.07) --(15.14, 4.86);

\draw [dashed]  (17.07, 5.93) --(14.07, 4.93);

\draw [dashed]  (16.93, 6.07) --(13.93, 5.07);

\draw (16.93, 6.07) --(16.00, 5.00);

\draw ( 0.00, 7.00) --( 0.00, 6.00);

\draw (11.00, 7.00) --(10.00, 6.00);

\draw (11.00, 7.00) --(11.07, 5.93);

\draw (15.07, 6.93) --(15.07, 5.93);

\draw (14.93, 7.07) --(14.93, 6.07);

\draw (17.07, 6.93) --(16.07, 5.93);

\draw [dashed]  (16.93, 7.07) --(14.00, 6.00);

\draw (16.93, 7.07) --(15.93, 6.07);

\draw ( 0.00, 8.00) --( 0.00, 7.00);

\draw (15.00, 8.00) --(15.07, 6.93);

\draw ( 0.00, 9.00) --( 0.00, 8.00);
\end{tikzpicture}
\quad \quad
\begin{tikzpicture}[scale=0.7]

\tikzset{radius= 0.05}

\clip (3,-1.0) rectangle ( 10.5,  6.5);
\draw[color=lightgray] (4,0) grid [step=2]  (16,8);

\node at (4,-0.5) {$8k+4$};
\node at (8,-0.5) {$8k+8$};

\foreach \s in {0,2,4,6}
{
\def\ss{\s-0};
\node at (3.5,\ss) {$\s$};
}


\draw [fill] ( 4.00, 0.00) circle;
\node [left] at  ( 4.00, 0.00) {\tiny{0}};

\draw [fill] ( 5.00, 0.00) circle;
\node [left] at  ( 5.00, 0.00) {\tiny{1}};

\draw [fill] ( 7.00, 0.00) circle;
\node [left] at  ( 7.00, 0.00) {\tiny{2}};

\draw [fill] ( 4.00, 1.00) circle;
\node [left] at  ( 4.00, 1.00) {\tiny{0}};

\draw [fill] ( 5.00, 1.00) circle;
\node [left] at  ( 5.00, 1.00) {\tiny{1}};

\draw [fill] ( 6.00, 1.00) circle;
\node [left] at  ( 6.00, 1.00) {\tiny{2}};

\draw [fill] ( 7.07, 0.93) circle;
\node [left] at  ( 7.07, 0.93) {\tiny{3}};

\draw [fill] ( 6.93, 1.07) circle;
\node [left] at  ( 6.93, 1.07) {\tiny{4}};

\draw [fill] ( 8.00, 1.00) circle;
\node [left] at  ( 8.00, 1.00) {\tiny{5}};

\draw [fill] (11.07, 0.93) circle;
\node [left] at  (11.07, 0.93) {\tiny{6}};

\draw [fill] (10.93, 1.07) circle;
\node [left] at  (10.93, 1.07) {\tiny{7}};

\draw [fill] (12.07, 0.93) circle;
\node [left] at  (12.07, 0.93) {\tiny{8}};

\draw [fill] (11.93, 1.07) circle;
\node [left] at  (11.93, 1.07) {\tiny{9}};

\draw [fill] (14.00, 1.00) circle;
\node [left] at  (14.00, 1.00) {\tiny{10}};

\draw [fill] ( 4.00, 2.00) circle;
\node [left] at  ( 4.00, 2.00) {\tiny{0}};

\draw [fill] ( 6.00, 2.00) circle;
\node [left] at  ( 6.00, 2.00) {\tiny{1}};

\draw [fill] ( 7.07, 1.93) circle;
\node [left] at  ( 7.07, 1.93) {\tiny{2}};

\draw [fill] ( 6.93, 2.07) circle;
\node [left] at  ( 6.93, 2.07) {\tiny{3}};

\draw [fill] (11.14, 1.86) circle;
\node [left] at  (11.14, 1.86) {\tiny{4}};

\draw [fill] (11.00, 2.00) circle;
\node [left] at  (11.00, 2.00) {\tiny{5}};

\draw [fill] (10.86, 2.14) circle;
\node [left] at  (10.86, 2.14) {\tiny{6}};

\draw [fill] (12.14, 1.86) circle;
\node [left] at  (12.14, 1.86) {\tiny{7}};

\draw [fill] (12.00, 2.00) circle;
\node [left] at  (12.00, 2.00) {\tiny{8}};

\draw [fill] (11.86, 2.14) circle;
\node [left] at  (11.86, 2.14) {\tiny{9}};

\draw [fill] (13.07, 1.93) circle;
\node [left] at  (13.07, 1.93) {\tiny{10}};

\draw [fill] (12.93, 2.07) circle;
\node [left] at  (12.93, 2.07) {\tiny{11}};

\draw [fill] (14.07, 1.93) circle;
\node [left] at  (14.07, 1.93) {\tiny{12}};

\draw [fill] (13.93, 2.07) circle;
\node [left] at  (13.93, 2.07) {\tiny{13}};

\draw [fill] ( 4.00, 3.00) circle;
\node [left] at  ( 4.00, 3.00) {\tiny{0}};

\draw [fill] ( 7.00, 3.00) circle;
\node [left] at  ( 7.00, 3.00) {\tiny{1}};

\draw [fill] (11.14, 2.86) circle;
\node [left] at  (11.14, 2.86) {\tiny{2}};

\draw [fill] (11.00, 3.00) circle;
\node [left] at  (11.00, 3.00) {\tiny{3}};

\draw [fill] (10.86, 3.14) circle;
\node [left] at  (10.86, 3.14) {\tiny{4}};

\draw [fill] (12.00, 3.00) circle;
\node [left] at  (12.00, 3.00) {\tiny{5}};

\draw [fill] (13.07, 2.93) circle;
\node [left] at  (13.07, 2.93) {\tiny{6}};

\draw [fill] (12.93, 3.07) circle;
\node [left] at  (12.93, 3.07) {\tiny{7}};

\draw [fill] (14.07, 2.93) circle;
\node [left] at  (14.07, 2.93) {\tiny{8}};

\draw [fill] (13.93, 3.07) circle;
\node [left] at  (13.93, 3.07) {\tiny{9}};

\draw [fill] ( 4.00, 4.00) circle;
\node [left] at  ( 4.00, 4.00) {\tiny{0}};

\draw [fill] (11.07, 3.93) circle;
\node [left] at  (11.07, 3.93) {\tiny{1}};

\draw [fill] (10.93, 4.07) circle;
\node [left] at  (10.93, 4.07) {\tiny{2}};

\draw [fill] (13.07, 3.93) circle;
\node [left] at  (13.07, 3.93) {\tiny{3}};

\draw [fill] (12.93, 4.07) circle;
\node [left] at  (12.93, 4.07) {\tiny{4}};

\draw [fill] (14.00, 4.00) circle;
\node [left] at  (14.00, 4.00) {\tiny{5}};

\draw [fill] (15.00, 4.00) circle;
\node [left] at  (15.00, 4.00) {\tiny{6}};

\draw [fill] ( 4.00, 5.00) circle;
\node [left] at  ( 4.00, 5.00) {\tiny{0}};

\draw [fill] (11.00, 5.00) circle;
\node [left] at  (11.00, 5.00) {\tiny{1}};

\draw [fill] (13.00, 5.00) circle;
\node [left] at  (13.00, 5.00) {\tiny{2}};

\draw [fill] (14.00, 5.00) circle;
\node [left] at  (14.00, 5.00) {\tiny{3}};

\draw [fill] (15.07, 4.93) circle;
\node [left] at  (15.07, 4.93) {\tiny{4}};

\draw [fill] (14.93, 5.07) circle;
\node [left] at  (14.93, 5.07) {\tiny{5}};

\draw [fill] ( 4.00, 6.00) circle;
\node [left] at  ( 4.00, 6.00) {\tiny{0}};

\draw [fill] (11.00, 6.00) circle;
\node [left] at  (11.00, 6.00) {\tiny{1}};

\draw [fill] (14.00, 6.00) circle;
\node [left] at  (14.00, 6.00) {\tiny{2}};

\draw [fill] (15.07, 5.93) circle;
\node [left] at  (15.07, 5.93) {\tiny{3}};

\draw [fill] (14.93, 6.07) circle;
\node [left] at  (14.93, 6.07) {\tiny{4}};

\draw [fill] ( 4.00, 7.00) circle;
\node [left] at  ( 4.00, 7.00) {\tiny{0}};

\draw [fill] (11.00, 7.00) circle;
\node [left] at  (11.00, 7.00) {\tiny{1}};

\draw [fill] (15.00, 7.00) circle;
\node [left] at  (15.00, 7.00) {\tiny{2}};

\draw [fill] ( 4.00, 8.00) circle;
\node [left] at  ( 4.00, 8.00) {\tiny{0}};

\draw [fill] (11.00, 8.00) circle;
\node [left] at  (11.00, 8.00) {\tiny{1}};


\draw ( 4.00, 1.00) --( 4.00, 0.00);

\draw ( 5.00, 1.00) --( 4.00, 0.00);

\draw ( 6.00, 1.00) --( 5.00, 0.00);

\draw [dashed]  ( 7.07, 0.93) --( 4.00, 0.00);

\draw [dashed]  ( 6.93, 1.07) --( 4.00, 0.00);

\draw ( 6.93, 1.07) --( 7.00, 0.00);

\draw [dashed]  ( 8.00, 1.00) --( 5.00, 0.00);

\draw ( 8.00, 1.00) --( 7.00, 0.00);

\draw ( 4.00, 2.00) --( 4.00, 1.00);

\draw ( 6.00, 2.00) --( 5.00, 1.00);

\draw [dashed]  ( 7.07, 1.93) --( 4.00, 1.00);

\draw ( 7.07, 1.93) --( 6.00, 1.00);

\draw ( 7.07, 1.93) --( 7.07, 0.93);

\draw [dashed]  ( 6.93, 2.07) --( 4.00, 1.00);

\draw ( 6.93, 2.07) --( 6.93, 1.07);

\draw (11.00, 2.00) --(11.07, 0.93);

\draw (10.86, 2.14) --(10.93, 1.07);

\draw (12.00, 2.00) --(11.07, 0.93);

\draw (11.86, 2.14) --(10.93, 1.07);

\draw (13.07, 1.93) --(12.07, 0.93);

\draw (12.93, 2.07) --(11.93, 1.07);

\draw [dashed]  (14.07, 1.93) --(10.93, 1.07);

\draw (13.93, 2.07) --(14.00, 1.00);

\draw ( 4.00, 3.00) --( 4.00, 2.00);

\draw [dashed]  ( 7.00, 3.00) --( 4.00, 2.00);

\draw ( 7.00, 3.00) --( 6.00, 2.00);

\draw ( 7.00, 3.00) --( 6.93, 2.07);

\draw (11.14, 2.86) --(11.14, 1.86);

\draw (11.00, 3.00) --(11.00, 2.00);

\draw (10.86, 3.14) --(10.86, 2.14);

\draw (12.00, 3.00) --(11.14, 1.86);

\draw (13.07, 2.93) --(12.14, 1.86);

\draw (12.93, 3.07) --(11.86, 2.14);

\draw [dashed]  (14.07, 2.93) --(11.14, 1.86);

\draw [dashed]  (13.93, 3.07) --(10.86, 2.14);

\draw (13.93, 3.07) --(12.93, 2.07);

\draw (13.93, 3.07) --(14.07, 1.93);

\draw ( 4.00, 4.00) --( 4.00, 3.00);

\draw (11.07, 3.93) --(11.00, 3.00);

\draw (10.93, 4.07) --(10.86, 3.14);

\draw (12.93, 4.07) --(12.00, 3.00);

\draw [dashed]  (14.00, 4.00) --(11.14, 2.86);

\draw (14.00, 4.00) --(13.07, 2.93);

\draw (14.00, 4.00) --(14.07, 2.93);

\draw ( 4.00, 5.00) --( 4.00, 4.00);

\draw (11.00, 5.00) --(10.93, 4.07);

\draw (14.00, 5.00) --(13.07, 3.93);

\draw (14.93, 5.07) --(15.00, 4.00);

\draw [dashed]  (17.93, 5.07) --(15.00, 4.00);

\draw ( 4.00, 6.00) --( 4.00, 5.00);

\draw (11.00, 6.00) --(11.00, 5.00);

\draw (14.00, 6.00) --(13.00, 5.00);

\draw (15.07, 5.93) --(14.00, 5.00);

\draw (15.07, 5.93) --(15.07, 4.93);

\draw (14.93, 6.07) --(14.00, 5.00);

\draw (14.93, 6.07) --(14.93, 5.07);

\draw ( 4.00, 7.00) --( 4.00, 6.00);

\draw (11.00, 7.00) --(11.00, 6.00);

\draw (15.00, 7.00) --(14.00, 6.00);

\draw (15.00, 7.00) --(15.07, 5.93);

\draw (15.00, 7.00) --(14.93, 6.07);

\draw ( 4.00, 8.00) --( 4.00, 7.00);

\draw (11.00, 8.00) --(11.00, 7.00);

\draw ( 4.00, 9.00) --( 4.00, 8.00);

\draw (11.00, 9.00) --(11.00, 8.00);
\end{tikzpicture}

\caption{Adams $E_2$-pages for $\bb{RP}^{\infty}_{8k}$ (left) and $\bb{RP}^{\infty}_{8k+4}$ (right). In both cases it is immediate that the codimension 5 homotopy group is zero.}
\label{fig: codimension 5 useful zero}
\end{figure}

\begin{prop} \label{prop: difficult map 2} The map
$$r_*: \bb{Z}_{(2)} \cong \pi_{4m+1}(\Sigma \bb{CP}^{\infty}_{2m-1})_{(2)} \longrightarrow \pi_{4m+1}( \bb{RP}^{\infty}_{4m-2})_{(2)}  \cong \bb{Z}/4$$
is surjective when $m$ is even, zero when $m \equiv 1\pmod{4}$, and multiplication by $2$ when $m \equiv 3 \pmod{4}$.
\end{prop}
\begin{proof}
We work $2$-locally. Consider the diagram
\[\begin{tikzcd}
	{\Sigma \bb{CP}^\infty_{2m-2}} & {\Sigma \bb{CP}^\infty_{2m-1}} & {\Sigma^2 \bb{CP}^{2m-2}_{2m-2} \simeq S^{4m-2}} \\
	{\bb{RP}^\infty_{4m-4}} & {\bb{RP}^\infty_{4m-2}} & {\Sigma \bb{RP}^{4m-3}_{4m-4} \simeq S^{4m-3} \oplus S^{4m-2}}
	\arrow[from=1-1, to=1-2]
	\arrow["r", from=1-1, to=2-1]
	\arrow[from=1-2, to=1-3]
	\arrow["r", from=1-2, to=2-2]
	\arrow["{\Sigma r}", from=1-3, to=2-3]
	\arrow[from=2-1, to=2-2]
	\arrow[from=2-2, to=2-3]
\end{tikzcd}\]
where the rows are cofibrations and the vertical maps are realification. By \cref{main theorem: calc}, realification is a mod-2 cohomology surjection, so the projection of $\Sigma r$ to the second factor has degree a 2-local unit, so is a 2-local homotopy equivalence.

Now apply $(4m+1)$-st homotopy groups. We claim that the group $\pi_{4m+1}(\bb{RP}^{\infty}_{4m-4})$ is zero. This follows immediately from the Adams charts of~\cref{fig: codimension 5 useful zero}. Using this fact in conjunction with the fourth stable stem being zero, and \cref{fig:SCPASS1,fig:SCPASS2}, the result of applying $(4m+1)$-st homotopy groups to the above diagram of cofibre sequences is
\[\begin{tikzcd}
\bb{Z} & \bb{Z} & \bb{Z}/8 \\
0 & \bb{Z}/4 & \bb{Z}/8.
    \arrow["r_*",from=1-1, to=2-1]
    \arrow["r_*",from=1-2, to=2-2]
    \arrow["(\Sigma r)_*",from=1-3, to=2-3]
    \arrow[swap, "\cong",from=1-3, to=2-3]
    \arrow[from=2-1, to=2-2]
    \arrow[hook, from=2-2, to=2-3]
    \arrow[from=1-1, to=1-2]
    \arrow[from=1-2, to=1-3]
\end{tikzcd}\]
By a diagram chase, it suffices to show that the top left horizontal map is $\bb{Z} \xrightarrow{\cdot a} \bb{Z}$ where $a$ is a multiple of 4 when $m$ is even, $\pm 2$ when $m \equiv 3 \pmod{4}$, and $\pm 1$ when $m \equiv 1 \pmod{4}$.

When $m$ is even ($2m-1 = 4k+3$), this map comes from degree $(8k+9)$ in the left-to-right truncation map in \cref{fig:SCPASS2}. The $h_0$-tower in the domain starts from Adams filtration 3, and the one in the codomain from Adams filtration 1. Since no differentials can affect this degree, the map on homotopy groups must be multiplication by a multiple of 4, hence have cokernel of size a multiple of 4, as required.

When $m$ is odd, our map is left-to-right in \cref{fig:SCPASS1}, in degree $8k+5$. Use the fact that the map is a cohomology injection and chase around $h_i$-multiplications to see that the map on $E_2$ pages sends one $h_0$-tower isomorphically to the other. The result then follows by using \cref{lem: complex ASS diff 1}: the differential in the source happens if $k$ is odd ($m \equiv 3 \pmod{4}$), and does not happen if $k$ is even ($m \equiv 1 \pmod{4}$). 
\end{proof}

\section{Calculations at odd primes} \label{section: odd primes}

In this section we study the situation after localisation at an odd prime $p$, which is much simpler. Combining this section with \cref{thm: sphere diagrams 2} will provide the reader with the integral version of \cref{thm:sphere diagrams} and with a little extra work also the integral statement of \cref{thm: CPell diagram}. We leave these deductions to the reader. 

Since 
$$\widetilde{H}_i(\bb{RP}^{\infty};\bb{Z}) \cong \begin{cases}
    \bb{Z}/2 & i \textrm{ odd} \\
    0 & i \textrm{ even}
\end{cases}
$$
it is immediate that
$$(\bb{RP}^\infty_{d})_{(p)} \simeq \begin{cases}
* & d \textrm{ odd}\\
    S^d_{(p)} & d \textrm{ even}
\end{cases}$$
so that in the metastable range, the set of stably trivial real vector bundles is identified with stable cohomotopy. The first few stable stems are given integrally by
$$\pi_{i+d}(S^d) \cong \begin{cases}
\bb{Z}/2\{\eta\} & i=1 \\
\bb{Z}/2\{\eta^2\} & i=2 \\
\bb{Z}/24\{\nu\} & i=3 \\
0 & i=4,5
\end{cases}$$
so we obtain the following calculation of the stable homotopy groups of $\bb{RP}^\infty_d$ in our dimensional range

\begin{lem} \label{lem: odd primary homotopy of RP} The group $\pi_m(\bb{RP}^\infty_d)_{(3)}$ is zero for $m-d = 1,2,4,5$. When $m-d =3$ we have:
$$\pi_{d+3}(\bb{RP}^\infty_d)_{(3)} \cong \begin{cases} 
\bb{Z}/3\{i\nu\} & d \textrm{ even} \\
0 & d \textrm{ odd} \\
\end{cases}$$ where $i$ is the inclusion of the bottom cell. If $p \geq 5$, then $\pi_m(\bb{RP}^\infty_d)_{(p)}$ is zero for $0<m-d \leq 5$. \qed 
\end{lem}

This tells us the odd-primary version of the bottom row of \cref{fig:even sphere diagram 2,fig:odd sphere diagram 2}. The top row was calculated by Hu \cite[Lemma 4.5]{Hu}, as follows.

\begin{lem} \label{lem: odd primary homotopy of CP} The group $\pi_m(\Sigma \bb{CP}^\infty_d)$ contains no 3-torsion for $m \leq 2d+3$. For $m=2d+4$ we have
$$\pi_{2d+4}(\Sigma \bb{CP}^\infty_d)_{(3)} \cong \begin{cases} 
\bb{Z}/3\{j' \nu\} & d \equiv 0 \pmod{3} \\
0 & \textrm{otherwise}
\end{cases}$$
where $j'$ is the inclusion of the bottom cell. If $p \geq 5$, then $\pi_m(\Sigma \bb{CP}^\infty_d)_{(p)}$ contains no $p$-torsion for $m \leq 2d+4$. \qed
\end{lem}

This tells us the integral values of the \emph{groups} in \cref{fig:even sphere diagram 2,fig:odd sphere diagram 2}: we add a $\bb{Z}/3$-summand to the bottom left of \cref{fig:odd sphere diagram 2}, and perhaps (depending on the parity of $n \pmod{3}$) to the top left of \cref{fig:even sphere diagram 2}, and otherwise nothing changes. For the most part, these groups are surrounded by groups which are $p$-locally zero for all odd primes, so the maps in and out can have no $p$-primary components. The first exception is in the leftmost column of \cref{fig:odd sphere diagram 2}, where realification is 3-locally a map $\bb{Z}_{(3)} \to \bb{Z}/3$ (c.f.~\cref{lem: odd primary homotopy of RP}).

\begin{lem} \label{lem: 3primaryMap} The 3-local realification map 
$$\bb{Z}_{(3)} \cong \pi_{2n+1}(\Sigma \bb{CP}^\infty_{n-1})_{(3)} \xrightarrow{\ r_* \ } \pi_{2n+1}(\bb{RP}^{\infty}_{2n-2})_{(3)} \cong \bb{Z}/3$$ 
is surjective when $n \equiv 0,1 \pmod{3}$ and zero when $n \equiv 2 \pmod{3}$.
\end{lem}

Though it looks rather different, the proof of this lemma is inspired by Kachi's proof of the corresponding unstable statement (see especially the diagram after Equation 4.2 in \cite{Kachi}).

\begin{proof} Consider the map induced on cofibre sequences by realification (the leftmost `realification' map is obtained as the restriction to $(2n-3)$-skeleta of the one to its right):
\[\begin{tikzcd}
\Sigma \bb{CP}^{n-2}_{n-2} & \Sigma \bb{CP}^{\infty}_{n-2} & \Sigma \bb{CP}^{\infty}_{n-1} & \Sigma^2 \bb{CP}^{n-2}_{n-2} & \Sigma^2 \bb{CP}^{\infty}_{n-2} \\
\bb{RP}^{2n-3}_{2n-3} & \bb{RP}^{\infty}_{2n-3} & \bb{RP}^{\infty}_{2n-2} & \Sigma \bb{RP}^{2n-3}_{2n-3} &
	\arrow["r",from=1-1, to=2-1]
    \arrow["r",from=1-2, to=2-2]
    \arrow["r",from=1-3, to=2-3]
    \arrow["r",from=1-4, to=2-4]
    \arrow[swap,"\simeq",from=1-4, to=2-4]
    \arrow[from=1-1, to=1-2]
    \arrow[from=1-2, to=1-3]
    \arrow["\delta_{\bb{C}}",from=1-3, to=1-4]
    \arrow[from=1-4, to=1-5]
    \arrow[from=2-1, to=2-2]
    \arrow[from=2-2, to=2-3]
    \arrow["\delta_{\bb{R}}",from=2-3, to=2-4]
    \arrow[swap, "\simeq",from=2-3, to=2-4]
\end{tikzcd}\]
By \cref{thm: realification vbs calc}, the leftmost realification map is an integral homotopy equivalence, so the rightmost is too. Since $2n-3$ is odd, we have $(\bb{RP}^{\infty}_{2n-3})_{(3)} \simeq *$, so $\delta_{\bb{R}}$ is a 3-local homotopy equivalence. It follows that (the third-from-left) $r_*$ is a 3-local surjection on $\pi_{2n+1}$ if and only if $(\delta_{\bb{C}})_*$ is, if and only if $(\delta_{\bb{C}})_*$ hits the generator $\nu \in \pi_{2n+1}(\Sigma^2\bb{CP}^{n-2}_{n-2})_{(3)} \cong \pi_{2n+1}(S^{2n-2})_{(3)} \cong \bb{Z}/3$. This happens if and only if the generator of $H^{2n-2}(\Sigma^2 \bb{CP}^{\infty}_{n-2};\bb{Z}/3)$ supports a Steenrod power $P^1$, which happens if and only if $n-2$ is not 3-divisible.
\end{proof}

The second thing we have to compute (for each odd $p$) is the degree of the $p$-primary truncation map $\bb{Z}_{(p)} \cong \pi_{2n+1}(\Sigma \bb{CP}^{\infty}_{n-1})_{(p)} \to \pi_{2n+1} (\Sigma \bb{CP}^{\infty}_n)_{(p)} \cong \bb{Z}_{(p)}$. Because $\eta$ is $p$-locally null, we always have $(\Sigma \bb{CP}^{n}_{n-1})_{(p)} \simeq S^{2n-1}_{(p)} \oplus S^{2n+1}_{(p)}$, and under this identification, the truncation
\[
(\Sigma \bb{CP}^{n}_{n-1})_{(p)} \longrightarrow (\Sigma \bb{CP}^{n}_{n})_{(p)} \simeq S^{2n+1}_{(p)}
\]
must be projection to the second factor. The following lemma is now immediate.

\begin{lem} For any odd $p$, the truncation map
$\pi_{2n+1}(\Sigma \bb{CP}^{\infty}_{n-1})_{(p)} \to \pi_{2n+1} (\Sigma \bb{CP}^{\infty}_n)_{(p)}$
is surjective. \qed
\end{lem}

\section{Vector bundles over spheres} \label{section: B' implies B}

In this section we will deduce \cref{thm:sphere diagrams} from \cref{thm: sphere diagrams 2}. We will need the following additional ingredient. Bott periodicity~\cite{Bott} provides an identification
$$\pi_{n}(SO) \cong \begin{cases}
    \bb{Z} & n \equiv 3 \pmod{4}\\
    \bb{Z}/2 & n \equiv 0,1 \pmod{8} \\
    0 & \textrm{otherwise.}
\end{cases}$$

Davis and Mahowald calculate the `$SO(d)$ of origin' for these classes, which is to say the dimension of the summand that is left after splitting off the maximum possible number of trivial line bundles.

\begin{thm}[{\cite[Theorem 1.1]{DavisMahowaldSOnOfOrigin}}]\label{thm:SOnOfOrigin}
   Let $\alpha \neq 0 \in \pi_{n}(SO) $. The minimal $d$ for which $\alpha$ is in the image of the map $\pi_n(SO(d)) \to \pi_n(SO)$ is 
   \pushQED{\qed}
   \[d = \begin{cases}
       2 m+1 & \text{if} \  n = 4m - 1 > 15\\
       6 & \text{if} \ n \equiv 0, 1\pmod{8} \ \text{and} \  n > 1. 
   \end{cases} \qedhere \] 
   \popQED
\end{thm}

They also give answers for small values of $n$ not covered by the above, which are more complicated to state (we omit these for simplicity). For real vector bundles, the `$SO(d)$ of origin' implies that in a certain dimensional range, being stably trivial is a property, not a structure, so we obtain the following corollary.

\begin{cor}\label{cor: SOnOfOrigin}
For $n > 15$ and $d \geq \frac{n+3}{2}$, there is an isomorphism
\[
\mathrm{Vect}_{\bb{R},d}^{0}(S^n) \cong \pi_n(SO/SO(d))
\]
between the group of stably trivial rank $d$ real bundles over $S^n$ and the $n$-th homotopy group of $SO/SO(d)$.
\end{cor}
\begin{proof}
Since $n \geq 2$, we have $\pi_n(O(d)) \cong \pi_n(SO(d))$, hence an equality
\[
\mathrm{Vect}_{\bb{R},d}^{0}(S^n) = \coker(\pi_{n+1}(BSO) \longrightarrow \pi_n(SO/SO(d))).
\]
Now consider the long exact sequence on homotopy associated to the fibration
\[
SO/SO(d) \longrightarrow BSO(d) \longrightarrow BSO.
\]
The map $\pi_{n+1}(BSO) \to \pi_n(SO/SO(d))$ is zero since~\cref{thm:SOnOfOrigin} implies that $\pi_n(SO(d)) \to \pi_n(SO)$ is surjective for $d \geq \frac{n+3}{2}$, hence the result.
\end{proof}

In the complex case, Bott periodicity~\cite{Bott} says that 
$$\pi_n(U) \cong \begin{cases}
    \bb{Z} & n \textrm{ odd}\\
    0 & n \textrm{ even.}
\end{cases}$$
so over even-dimensional spheres $S^{2n}$, the set of stably trivial vector bundles is again identified with $\pi_{2n}(U/U(d))$. For odd-dimensional spheres, the situation is more complicated. 

\begin{thm}[{\cite{Yokota, Toda}}] \label{thm: YokotaToda} There exists a map $$f : \Sigma \bb{CP}^{\infty} \longrightarrow SU$$
which is a surjection on integral cohomology, and on unstable homotopy groups.
\end{thm}

\begin{proof} The map is constructed by Yokota \cite[Section 4]{Yokota}, who shows that it is onto on cohomology. Toda \cite[Proposition 4.2]{Toda} shows that any such map must be a surjection on homotopy groups (he also explicitly constructs classes $\zeta_n$ such that $f_*(\zeta_n)$ generates $\pi_{2n+1}(SU)\cong \bb{Z}$).
\end{proof}

Ideally, one would prove a relationship between the map $f_d : \Sigma \bb{CP}^{\infty}_d \to SU/SU(d) \simeq U/U(d)$ induced on the quotient by Yokota's map, and the value at $\bb{C}^d$ of the natural transformation
\[
U/U(-) \longrightarrow P_1 U/U(-) \simeq D_1 BU(-)
\]
where the equivalence follows since $U/U(-)$ is reduced see e.g.,~\cite[Example 10.1]{Weiss}. It seems possible that this is related to the splitting of \cite[Example 10.2]{Weiss}. We will make do without understanding this relationship. We also need a tiny amount of cohomology information.

\begin{lem} \label{lem: CohomologyOfUn}
    The quotient map $U \to U/U(d)$ induces a surjection $H^{2d+1}(U/U(d); \bb{Z}) \longtwoheadrightarrow H^{2d+1}(U; \bb{Z})$.
\end{lem}

\begin{proof} The cohomology of $U(d)$ is an exterior algebra
    $H^*(U(d)) \cong \Lambda[e_1, \dots, e_d],$
    where $|e_i| = 2i-1$, see e.g., \cite[Corollary 4D.3]{HatcherAT}. The result follows by (e.g.) a Serre spectral sequence argument.
\end{proof}

We now prove~\cref{thm:sphere diagrams}.

\begin{proof}[Proof of \cref{thm:sphere diagrams}] 
Recall (\cref{thm: stably trivial stunted projective}) that for $\bb{k}=\bb{R},\bb{C}$, and any finite cell complex $X$ of dimension at least $2\dim_\R(\bb{k})$ we have a natural isomorphism
\[
\vect{\bb{k}, d}^0(X) \cong \coker \left( [\Sigma X, \BAut] \xrightarrow{ \ \delta_\bb{k} \ } [\Sigma^\infty X, \Sigma^\infty  \Sigma^{\dim_\bb{R}(\bb{k})-1}\bb{kP}^\infty_d] \right)
\] in the metastable range (recall that $\BAut$ is our chosen notation for $\BO$ or $\BU$ as appropriate). For $X = S^{2n}$ or $X = S^{2n+1}$, all of the groups we consider are metastable when $n \geq 5$, which is implied by our assumption that $n \geq 8$.

To deduce \cref{thm:sphere diagrams} from \cref{thm: sphere diagrams 2}, we have to take \cref{fig:even sphere diagram 2,fig:odd sphere diagram 2}, and quotient each group by the image of the above map (setting $X =S^{2n}$ or $S^{2n+1}$, $\bb{k}=\bb{R},\bb{C}$, and $d \geq 2n-4$ as appropriate), to obtain \cref{fig:even sphere diagram,fig:odd sphere diagram}.

For the real bundles, we wish to use \cref{cor: SOnOfOrigin} to see that this image is trivial, so nothing changes when we quotient. To use this corollary for even spheres $S^{2n}$, we must verify its hypotheses: that $2n > 15$ and $d \geq \frac{2n+3}{2}$. The first hypothesis follows immediately because $n \geq 8$. For the second, note that we are considering bundles of rank $d \geq 2n-4$, and since $n \geq 8$, we have $n \geq 6$ and hence $d \geq 2n - 4 \geq n+2 > \frac{2n+3}{2}$. The check for odd spheres is similar.

For complex bundles on even spheres, $[\Sigma S^{2n}, BU] \cong \pi_{2n}(U)=0$ by Bott periodicity, so the top row of \cref{fig:even sphere diagram} is again the same as the top row of \cref{fig:even sphere diagram 2}. 

It remains to treat the complex cases on odd spheres. This time we will see interesting differences between the top rows of \cref{fig:odd sphere diagram} and \cref{fig:odd sphere diagram 2}.

We begin in codimension 1, where $\vect{\bb{C},n}^0(S^{2n+1})$ is the cokernel of the composite 
\[
\delta_{\bb{C}} : \pi_{2n+1}(U) \longrightarrow \pi_{2n+1}(U/U(n)) \xrightarrow{\ \cong\ } \pi_{2n+1}(\Omega^\infty\Sigma^{\infty + 1} \bb{CP}^{\infty}_{n})
\]
where the first map is the boundary map in the fibration $U/U(n) \to BU(n) \to BU$ and the second map is the value at $\bb{C}^n$ of the linear approximation
\[
U/\sf{u} \longrightarrow P_1(U/\sf{u}) \simeq D_1(U/\sf{u})
\]
which is an equivalence on $\pi_{2n+1}$ by (the proof of) \cref{lem: tower of BAut} since $U/\sf{u}$ is the reduced approximation to $\Bu$ and hence has the same convergence behaviour, i.e., the linear approximation map is $(2\dim_\R(\C^n)+1)$-connected. The Hurewicz map
\[
h: \pi_{2n+1}(\Omega^\infty\Sigma^{\infty+1} \bb{CP}^{\infty}_{n}) \longrightarrow H_{2n+1}(\Omega^\infty\Sigma^{\infty+1} \bb{CP}^{\infty}_{n}; \bb{Z}) \cong H_{2n+1}(\Sigma^{\infty+1} \bb{CP}^{\infty}_{n}; \bb{Z}) \cong \bb{Z}
\]
is an isomorphism, where the isomorphism $H_{2n+1}(\Omega^\infty\Sigma^{\infty+1} \bb{CP}^{\infty}_{n}; \bb{Z}) \cong H_{2n+1}(\Sigma^{\infty+1} \bb{CP}^{\infty}_{n}; \bb{Z})$ follows from the Freudenthal Suspension Theorem. Bott periodicity says that, $\pi_{2n+1}(U) \cong \bb{Z}$, hence we must show that
\[
(\delta_{\bb{C}})_*: \bb{Z} \cong \pi_{2n+1}(U) \longrightarrow \pi_{2n+1}(\Omega^\infty\Sigma^{\infty+1}\bb{CP}^\infty_n) \cong \bb{Z}
\]
has image $n! \cdot \bb{Z} \subset \bb{Z}$. Consider the following diagram, formed from the naturality squares for the Hurewicz map applied our map $\delta_{\bb{C}}$ and Yokota's map $f$ (\cref{thm: YokotaToda}):
\[\begin{tikzcd}
\pi_{2n+1}(\Sigma \bb{CP}^{\infty}) & \pi_{2n+1}(U) & \pi_{2n+1}(U/U(n)) & \pi_{2n+1}(\Omega^\infty\Sigma^{\infty+1} \bb{CP}^{\infty}_{n}) \\
H_{2n+1}(\Sigma \bb{CP}^{\infty}; \bb{Z}) & H_{2n+1}(U; \bb{Z}) & H_{2n+1}(U/U(n); \bb{Z}) & H_{2n+1}(\Sigma^{\infty+1} \bb{CP}^{\infty}_{n}; \bb{Z})
    \arrow["h",from=1-1, to=2-1]
    \arrow["h",from=1-2, to=2-2]
    \arrow["h",from=1-3, to=2-3]
    \arrow[swap, "\cong",from=1-3, to=2-3]
    \arrow["h",from=1-4, to=2-4]
    \arrow[swap, "\cong",from=1-4, to=2-4]
    \arrow["f_*",from=2-1, to=2-2]
    \arrow[from=2-2, to=2-3]
    \arrow["\cong",from=2-3, to=2-4]
    \arrow["f_*",from=1-1, to=1-2]
    \arrow[from=1-2, to=1-3]
    \arrow["\cong", from=1-3, to=1-4]
    \arrow[bend left = 15,  "(\delta_{\bb{C}})_*", from=1-2, to=1-4]
\end{tikzcd}\]
The rightmost Hurewicz map is an isomorphism. By \cref{thm: YokotaToda}, the top left map $f_*$ on homotopy is a surjection. It follows that the image of $\delta_{\bb{C}}$ in the top right group $\pi_{2n+1}(\Omega^\infty\Sigma^{\infty+1} \bb{CP}^{\infty}_{n}) \cong \bb{Z}$ is a subgroup of the same index as the image of the whole composite $h \circ \delta_{\bb{C}} \circ f_*$ in $H_{2n+1}(\Sigma^{\infty+1} \bb{CP}^{\infty}_{n}; \bb{Z})$. By commutativity of the diagram, this composite is the same as the composite through the bottom left corner of the diagram, which we must now show has image a subgroup of index $n!$. This follows because the image of the leftmost Hurewicz map is $n! \cdot \bb{Z} \subset \bb{Z} \cong H_{2n+1}(\Sigma \bb{CP}^{\infty})$, (\cref{thm: MosherHurewicz}), and by \cref{thm: YokotaToda,lem: CohomologyOfUn} and the universal coefficient theorem, the bottom composite is a surjection.

In codimension 2, $\vect{\bb{C},n-1}^0(S^{2n+1})$ is the cokernel of the composite $$\delta_{\bb{C}} : \bb{Z} \cong \pi_{2n+1}(U) \longrightarrow \pi_{2n+1}(U/U(n-1)) \xrightarrow{\ \cong\ } \pi_{2n+1}(\Omega^\infty\Sigma^{\infty+1} \bb{CP}^{\infty}_{n-1}) \cong \begin{cases}
    \bb{Z} \oplus \bb{Z}/2 & n \textrm{ odd} \\
    \bb{Z} & n\textrm{ even.}
\end{cases}$$

We compare with the codimension 1 case by considering the diagram
\[\begin{tikzcd}
\bb{Z} \cong \pi_{2n+1}(\Sigma \bb{CP}^{\infty}) & \pi_{2n+1}(U) & \pi_{2n+1}(U/U(n-1)) & \pi_{2n+1}(\Omega^\infty\Sigma^{\infty+1} \bb{CP}^{\infty}_{n-1}) \\
& \pi_{2n+1}(U) & \pi_{2n+1}(U/U(n)) & \pi_{2n+1}(\Omega^\infty\Sigma^{\infty+1} \bb{CP}^{\infty}_{n})
    \arrow[equal,from=1-2, to=2-2]
    \arrow[from=1-3, to=2-3]
    \arrow[from=1-4, to=2-4]
    \arrow[from=1-4, to=2-4]
    \arrow["f_*",two heads,from=1-1, to=1-2]
    \arrow[from=2-2, to=2-3]
    \arrow["\cong",from=2-3, to=2-4]
    \arrow[from=1-2, to=1-3]
    \arrow["\cong", from=1-3, to=1-4]
    \arrow[bend left = 15,  "(\delta_{\bb{C}})_*", from=1-2, to=1-4]
    \arrow[bend right = 15, swap, "(\delta_{\bb{C}})_*", from=2-2, to=2-4]
\end{tikzcd}\]

We know from the codimension 1 case that the composite top left to bottom right is $\bb{Z} \xrightarrow{\cdot n!} \bb{Z}$. The right hand vertical map is truncation, which by \cref{thm: sphere diagrams 2} (\cref{fig:odd sphere diagram 2}) is $\bb{Z} \xrightarrow{\cdot 2} \bb{Z}$ when $n$ is even, so in this case the desired conclusion follows: $(\delta_{\bb{C}})_*$ in the top row is $\bb{Z} \xrightarrow{\cdot (n!/2)} \bb{Z}$.

For $n$ odd, the image $(\delta_{\bb{C}})_*(1)$ of the generator under any choice of isomorphism $\pi_{2n+1}(\Omega^\infty\Sigma^{\infty+1} \bb{CP}^{\infty}_{n-1}) \cong \bb{Z} \oplus \bb{Z}/2 $ (c.f.~\cref{fig:odd sphere diagram 2}) must be of the form $(\pm n!,x)$ for some $x \in \bb{Z}/2$. Since $n \geq 2$, the value of $x \in \bb{Z}/2$ is independent of the choice of isomorphism: $x = 0$ if and only if $(\delta_{\bb{C}})_*(1)$ is 2-divisible.

Consider the naturality square for realification (\cref{main theorem: calc}), where the bottom map is an isomorphism by the real case:
\[\begin{tikzcd}
\bb{Z} \oplus \bb{Z}/2 &[-0.4cm]\pi_{2n+1}^{\mathrm{s}}(\Sigma^{\infty+1} \bb{CP}^{\infty}_{n-1}) & \vect{\bb{C},n-1}^0(S^{2n+1})\\
& \pi_{2n+1}^{\mathrm{s}}(\Sigma^{\infty} \bb{RP}^{\infty}_{2n-2}) & \vect{\bb{R},2n-2}^0(S^{2n+1})
\arrow["r_*",from=1-2, to=2-2]
\arrow["r_*",from=1-3, to=2-3]
\arrow[two heads,from=1-2, to=1-3]
\arrow["\cong",from=2-2, to=2-3]
\arrow["\cong",from=1-1, to=1-2]
\end{tikzcd}\]

The top map is precisely the quotient by $(\delta_{\bb{C}})_*(1)=(\pm n!,x)$; in particular this class maps to zero in the bottom right, so $r_*(\pm n! , x) = 0 \in \pi_{2n+1}^{\mathrm{s}}(\Sigma^{\infty} \bb{RP}^{\infty}_{2n-2})$. Consulting \cref{fig:odd sphere diagram 2}, we see that $r_*(0,1)$ has non-zero realification in $\pi_{2n+1}^{\mathrm{s}}(\Sigma^{\infty} \bb{RP}^{\infty}_{2n-2})$, and that $r_*(1,0)$ has order 8 (when $n \equiv 1\pmod{4}$), and order 16 (when $n \equiv 3\pmod{4}$). Since $n!$ is divisible by 8 for $n \geq 4$, and divisible by 16 for $n \geq 6$, it follows that $r_*(n!,0)=n! \cdot r_*(1,0) = 0$ for $n \geq 5$ odd, so $0 = r_*(\pm n!,x) = x \cdot r_*(0,1)$. Since $r_*(0,1)$ is non-zero, we must have $x=0 \in \bb{Z}/2$. So the top left corner of \cref{fig:odd sphere diagram} is as claimed, completing the derivation of this diagram, and hence the proof.
\end{proof}

\section{Vector bundles over complex projective space} \label{section: B' implies A}
In this section we use the computations of~\cref{thm: sphere diagrams 2} to prove~\cref{thm: CPell diagram}. Since the negative first reduced real and complex $K$-theory of complex projective spaces are zero, see e.g.,~\cite[Theorem 2]{Fujii}, \cref{main theorem: calc} tells us that stable triviality for a bundle over complex projective space is a property not a structure. That is, the groups of stably trivial vector bundles $\vect{\bb{k}, d}^0(\mathbb{CP}^\ell)$ are equivalently vector bundles over $\bb{CP}^\ell$ equipped with a stable trivialisation.

We will do this by studying the corresponding Atiyah--Hirzebruch spectral sequences. In the complex case, writing $[X,Y]$ for the group of stable maps between spaces $X$ and $Y$, the spectral sequence has signature
\begin{equation}\label{fig:E2AHSSComplex}
E^{p,q}_2 = H^p(\bb{CP}^\ell, \pi_{-q}^s(\Sigma\bb{CP}^\infty_{\ell-j})) \cong \begin{cases}
    \pi_{-q}^s(\Sigma\bb{CP}^\infty_{\ell-j}) & 0 \leq p \leq 2\ell, \ p\ \text{even} \\
    0 & \textrm{otherwise}
\end{cases} \Longrightarrow [\bb{CP}^\ell, \Sigma^{p+q+1}\bb{CP}^\infty_{\ell-j}]    
\end{equation}
and we will consider it for $0 \leq j \leq 2$. In the real case, this spectral sequence has signature
\begin{equation}\label{fig:E2AHSS}
E^{p,q}_2 = H^p(\bb{CP}^\ell, \pi_{-q}^s(\bb{RP}^\infty_{2\ell-i})) \cong \begin{cases}
    \pi_{-q}^s(\bb{RP}^\infty_{2\ell-i}) & 0 \leq p \leq 2\ell, \ p\ \text{even} \\
    0 & \textrm{otherwise}
\end{cases} \Longrightarrow [\bb{CP}^\ell, \Sigma^{p+q}\bb{RP}^\infty_{2\ell-i}]
\end{equation}
and we will consider it for $0 \leq i \leq 4$. The spectral sequences are of cohomological type, with $d_r$-differentials of bidegree $(r,1-r)$.

Roughly speaking, we already know the $E_2$-pages, since we have computed the homotopy groups in \cref{section: stunted projective}, and we know the cohomology of projective spaces, and our task will be to determine differentials and hence obtain the stable mapping sets, alias bundles equipped with stable trivialisation, which proves \cref{thm: CPell diagram}.

Note that the spectral sequence of~\cref{fig:E2AHSSComplex} has been computed by Hu~\cite{Hu} within the relevant range, and we will not redo these computations. The spectral sequence of~\cref{fig:E2AHSS} will let us deduce the corresponding groups of stably trivial real bundles, and the maps of spectral sequences induced by realification and stabilisation/truncation will allow us to deduce the effect of those operations.
 
These Atiyah--Hirzebruch spectral sequences are induced by the cellular filtration on the space $\bb{CP}^\ell$, so their differentials are governed by its cell structure. We recall now how these differentials are computed, see also~\cite[p.7793, p.7798]{Hu} but proceed with caution since our indexing choices differ.

\begin{rem}\label{rmk: differentials}The $d_2$ differentials are determined by the cell structure of $\bb{CP}^{n+1}_n$, and in particular it follows from the definitions that the $d_2$ differential $d_2: E_2^{2n, -m} \to E_2^{2n+2, -m-1}$ is zero if $n$ is even and is given by $\eta^*$ if $n$ is odd. Since both spectral sequences are concentrated in even columns ($\bb{CP}^{\ell}$ only having even-dimensional cells), there can be no $d_3$ differentials. Similarly, the differential $d_4: E_4^{2n, -m} \to E_4^{2n+4, -m-3}$ is given by pullback along $(\lambda\nu)^*$ where $\lambda$ is as in~\cref{lem:Mosher}. We will use both of these facts frequently in what follows.
\end{rem}

We now prove \cref{thm: CPell diagram}, which is to say that we deduce \cref{fig:CPell diagram}.

\begin{proof}[Proof of~\cref{thm: CPell diagram}]
We will proceed on a case-by-case basis, working from right-to-left along~\cref{fig:CPell diagram}. By \cref{main theorem: calc} (and the $K$-theory vanishing noted above), the mapping sets computed by the spectral sequences are the desired groups of stably trivial bundles, provided that we are in the metastable range. We are considering groups of complex codimension at most 2 and real codimension at most 4, and $\dim(\bb{CP}^\ell) = 2 \ell$, so metastability follows from our assumption that $\ell \geq 5$.

\textbf{Codimension zero complex bundles} $\vect{\bb{C},\ell}^0(\bb{CP}^\ell)$ is zero, because it is in the stable range.

\textbf{Codimension zero real bundles} (\cref{fig:E2AHSS}, $i=0$) $\vect{\bb{R}, 2\ell}^0(\bb{CP}^\ell)$. We consider the Atiyah--Hirzebruch spectral sequence of \cref{fig:E2AHSS} with $i=0$, where we wish to compute the groups $E_{\infty}^{p,q}$ for $p+q = 0$ (the associated graded of $[\bb{CP}^\ell,\bb{RP}_{2 \ell}^\infty] = \vect{\bb{R}, 2 \ell}^0(\mathbb{CP}^\ell)$). The $E_2$-page is concentrated in even columns in the half-infinite vertical band bounded by the inequalities $0 \leq p \leq 2 \ell$ and $q \leq - 2 \ell$. The top of this band takes the form
\[\begin{tikzcd}[ampersand replacement=\&,cramped,sep=tiny]
	\&\& {} \&\&\&\&\&\&\&\& \\
	\& {E_2} \&\& {2\ell-6} \& {2\ell -5} \& {2\ell -4} \& {2\ell -3} \& {2\ell -2} \& {2\ell -1} \& {2\ell } \\
	{} \&\&\&\&\&\&\&\&\&\& {} \\
	\& {\geq -(2\ell-1)} \&\& 0 \& 0 \& 0 \& 0 \& 0 \& 0 \& 0 \\
	\& {-2\ell} \&\& {\bb{Z}} \& 0 \& {\bb{Z}} \& 0 \& {\bb{Z}} \& 0 \& {\bb{Z}} \\
	\&\& {} \& {}
	\arrow[no head, from=1-3, to=6-3]
	\arrow[no head, from=3-1, to=3-11]
\end{tikzcd}\]
so the only non-zero group on the main diagonal $p+q=0$ is $\pi_{2\ell}^s(\bb{RP}^\infty_{2\ell}) \cong \bb{Z}$ (\cref{thm: sphere diagrams 2} c.f.~\cref{fig:even sphere diagram 2}). For degree reasons, this group can support no non-trivial differentials, so it must survive to $E^\infty$. It follows that $\vect{\bb{R}, 2 \ell}^0(\mathbb{CP}^\ell) \cong \bb{Z}$, as desired.

\textbf{Realification} $\vect{\bb{C}, \ell}^0(\bb{CP}^\ell) \to \vect{\bb{R}, 2\ell}^0(\bb{CP}^\ell)$. The domain is zero, so the map is necessarily zero.

\textbf{Codimension one real bundles} $\vect{\bb{R}, 2\ell-1}^0(\bb{CP}^\ell)$ (\cref{fig:E2AHSS}, $i=1$). When $\ell$ is even, the $E_2$-page has only zeroes on the diagonal, so $\vect{\bb{R}, 2\ell-1}^0(\bb{CP}^\ell) = 0$. When $\ell$ is odd, the $E_2$-page begins
\[\begin{tikzcd}[ampersand replacement=\&,cramped,sep=tiny]
	\&\& {} \&\&\&\&\&\&\&\& \\
	\& {E_2} \&\& {2\ell-6} \& {2\ell -5} \& {2\ell -4} \& {2\ell -3} \& {2\ell -2} \& {2\ell -1} \& {2\ell } \\
	{} \&\&\&\&\&\&\&\&\&\& {} \\
	\& {\geq -(2\ell-2)} \&\& 0 \& 0 \& 0 \& 0 \& 0 \& 0 \& 0 \\
	\& {-(2\ell-1)} \&\& {\bb{Z}/2} \& 0 \& {\bb{Z}/2} \& 0 \& {\bb{Z}/2} \& 0 \& {\bb{Z}/2} \\
	\& {-2\ell} \&\& {\bb{Z}/2} \& 0 \& {\bb{Z}/2} \& 0 \& {\bb{Z}/2} \& 0 \& {\bb{Z}/2} \\
	\&\& {} \& {}
	\arrow[no head, from=1-3, to=7-3]
	\arrow[no head, from=3-1, to=3-11]
	\arrow["{d_2}"{description}, from=5-8, to=6-10]
\end{tikzcd}\]
There is only one potential $d_2$-differential $d_2: \bb{Z}/2 \to \bb{Z}/2$ affecting the main diagonal. This differential is induced by the attaching map for the two-cell complex $\bb{CP}^\ell_{\ell-1}$. Since $\ell-1$ is even, this differential must be zero (c.f.~\cref{rmk: differentials}). The single $\bb{Z}/2$ on the diagonal survives, since there is no room for higher differentials, giving the result claimed in \cref{fig:CPell diagram}: $$\vect{\bb{R}, 2\ell-1}^0(\bb{CP}^\ell) \cong \begin{cases}
    \bb{Z}/2 & \ell \textrm{ odd,} \\
    0 & \ell \textrm{ even.}
\end{cases}$$

\textbf{Stabilisation} $\vect{\bb{R}, 2\ell-1}^0(\bb{CP}^\ell) \to \vect{\bb{R}, 2\ell}^0(\bb{CP}^\ell)$ is then automatically zero, since the codomain is $\bb{Z}$.

\textbf{Codimension two real bundles} $\vect{\bb{R}, 2\ell-2}^0(\bb{CP}^\ell)$ (\cref{fig:E2AHSS}, $i=2$). For $\ell$ odd the $E_2$-page begins
\[\begin{tikzcd}[ampersand replacement=\&,cramped,sep=tiny]
	\&\& {} \&\&\&\&\&\&\&\& \\
	\& {E_2} \&\& {2\ell-6} \& {2\ell -5} \& {2\ell -4} \& {2\ell -3} \& {2\ell -2} \& {2\ell -1} \& {2\ell } \\
	{} \&\&\&\&\&\&\&\&\&\& {} \\
	\& {\geq -(2\ell-3)} \&\& 0 \& 0 \& 0 \& 0 \& 0 \& 0 \& 0 \\
	\& {-(2\ell-2)} \&\& {\bb{Z}} \& 0 \& {\bb{Z}} \& 0 \& {\bb{Z}} \& 0 \& {\bb{Z}} \\
	\& {-(2\ell-1)} \&\& {(\bb{Z}/2)^2} \& 0 \& {(\bb{Z}/2)^2} \& 0 \& {(\bb{Z}/2)^2} \& 0 \& {(\bb{Z}/2)^2} \\
	\& {-2\ell} \&\& {(\bb{Z}/2)^2} \& 0 \& {(\bb{Z}/2)^2} \& 0 \& {(\bb{Z}/2)^2} \& 0 \& {(\bb{Z}/2)^2} \\
	\&\& {}
	\arrow[no head, from=1-3, to=8-3]
	\arrow[no head, from=3-1, to=3-11]
	\arrow["{d_2}"{description}, from=5-8, to=6-10]
	\arrow["{d_2}"{description}, from=6-8, to=7-10]
\end{tikzcd}\]
where there are two possible differentials affecting the main diagonal. Again, since $\ell$ is odd, these differentials are zero (c.f.~\cref{rmk: differentials}). There is no room for longer differentials, so the Atiyah--Hirzebruch spectral sequence collapses at the $E_2$-page. Here we get lucky with the direction the filtration is going: the $E_\infty$-page tells us that $\vect{\bb{R}, 2\ell-2}^0(\bb{CP}^\ell) \cong [\bb{CP}^{\ell},\bb{RP}^{\infty}_{2 \ell-2}]$ sits in a short exact sequence
\begin{equation}\label{eqn: real codim 2 splitting}
    0 \longrightarrow (\bb{Z}/2)^2 \longrightarrow [\bb{CP}^{\ell},\bb{RP}^{\infty}_{2 \ell-2}] \longrightarrow \bb{Z} \longrightarrow 0 
\end{equation}
which splits since $\bb{Z}$ is free. Choosing a splitting gives an isomorphism $\vect{\bb{R}, 2\ell-2}^0(\bb{CP}^\ell) \cong \bb{Z} \oplus (\bb{Z}/2)^2$. We will later fix a particular choice of splitting.

For $\ell$ even, the $E_2$-page takes the form
\[\begin{tikzcd}[ampersand replacement=\&,cramped,sep=tiny]
	\&\& {} \&\&\&\&\&\&\&\& \\
	\& {E_2} \&\& {2\ell-6} \& {2\ell -5} \& {2\ell -4} \& {2\ell -3} \& {2\ell -2} \& {2\ell -1} \& {2\ell } \\
	{} \&\&\&\&\&\&\&\&\&\& {} \\
	\& {\geq -(2\ell-3)} \&\& 0 \& 0 \& 0 \& 0 \& 0 \& 0 \& 0 \\
	\& {-(2\ell-2)} \&\& {\bb{Z}} \& 0 \& {\bb{Z}} \& 0 \& {\bb{Z}} \& 0 \& {\bb{Z}} \\
	\& {-(2\ell-1)} \&\& {\bb{Z}/4} \& 0 \& {\bb{Z}/4} \& 0 \& {\bb{Z}/4} \& 0 \& {\bb{Z}/4} \\
	\& {-2\ell} \&\& 0 \& 0 \& 0 \& 0 \& 0 \& 0 \& 0 \\
	\&\& {}
	\arrow[no head, from=1-3, to=8-3]
	\arrow[no head, from=3-1, to=3-11]
	\arrow["{d_2}"{description}, from=5-8, to=6-10]
\end{tikzcd}\]
where there is only a single possible differential affecting the main diagonal $p+q = 0$. Since $\ell$ is even, this differential is given by precomposition $\eta^*$ with $\eta$ (c.f.~\cref{rmk: differentials}). Consulting the Adams charts for $\bb{RP}^\infty_{2 \ell-2}$, (\cref{fig:RPASS2}-left and \cref{fig:RPASS4}-left) we see that there are $h_1$ multiplications connecting the first two columns, so $\eta^*$ induces the multiplication-by-$2$ map $\pi_{2 \ell - 2}(\bb{RP}^\infty_{2 \ell-2}) \cong \bb{Z} \to \bb{Z}/4 \cong \pi_{2 \ell - 1}(\bb{RP}^\infty_{2 \ell-2})$. The only non-zero group on the main diagonal of the $E_3$-page is the kernel of this map, which is $\bb{Z}$. There is no room for differentials on later pages, so the spectral sequence collapses at $E_3$, and the result follows.

\textbf{Stabilisation} $\vect{\bb{R}, 2\ell-2}^0(\bb{CP}^\ell) \to \vect{\bb{R}, 2\ell-1}^0(\bb{CP}^\ell)$. For $\ell$ even, the map is necessarily zero since the codomain is zero. For $\ell$ odd, we consider the map induced on Atiyah--Hirzebruch spectral sequences by truncation $\bb{RP}^\infty_{2 \ell - 2} \to \bb{RP}^\infty_{2 \ell - 1}$. We have seen the $E_2$-pages above, they look as follows.

\begin{adjustbox}{max width = \textwidth}
\begin{tikzcd}[ampersand replacement=\&,cramped,sep=tiny]
	\&\& {} \&\&\&\&\&\&\&\&\&\& {} \&\&\&\&\&\& \\
	\& {E_2} \&\& {2 \ell-4} \& {2\ell-3} \& {2\ell -2} \& {2\ell -1} \& {2\ell } \&\&\&\& {E_2} \&\& {2\ell-4} \& {2\ell-3} \& {2\ell -2} \& {2\ell -1} \& {2\ell } \\
	{} \&\&\&\&\&\&\&\& {} \&\& {} \&\&\&\&\&\&\&\& {} \\
	\& {\geq-(2\ell-3)} \&\& 0 \& 0 \& 0 \& 0 \& 0 \&\&\&\& {\geq -(2\ell-3)} \&\& 0 \& 0 \& 0 \& 0 \& 0 \\
    \& {-(2\ell-2)} \&\& \bb{Z} \& 0 \& \bb{Z} \& 0 \& \bb{Z} \&\&\&\& {-(2\ell-2)} \&\& 0 \& 0 \& 0 \& 0 \& 0 \\
	\& {-(2\ell-1)} \&\& {(\bb{Z}/2)^2} \& 0 \& {(\bb{Z}/2)^2} \& 0 \& {(\bb{Z}/2)^2} \&\&\&\& {-(2\ell-1)} \&\& {\bb{Z}/2} \& 0 \& {\bb{Z}/2} \& 0 \& {\bb{Z}/2} \\
	\& {-2\ell} \&\& {(\bb{Z}/2)^2} \& 0 \& {(\bb{Z}/2)^2} \& 0 \& {(\bb{Z}/2)^2} \&\&\&\& {-2\ell} \&\& {\bb{Z}/2} \& 0 \& {\bb{Z}/2} \& 0 \& {\bb{Z}/2} \\
	\&\& {} \&\&\&\&\&\&\&\&\&\& {}
	\arrow[no head, from=1-3, to=8-3]
	\arrow[no head, from=1-13, to=8-13]
	\arrow[no head, from=3-1, to=3-9]
	\arrow[no head, from=3-11, to=3-19]
	\arrow[curve={height=18pt}, dashed, from=7-8, to=7-18]
\end{tikzcd}
\end{adjustbox}
The arrow indicates the only potentially non-zero component of the map of $E_2$-pages on the main diagonal. This map is the map on $\pi_{2 \ell}$ induced by the truncation map, which we calculated in terms of our bases to be projection to the first factor (\cref{thm: sphere diagrams 2}, more precisely \cref{fig:even sphere diagram 2} with $n = \ell$). This tells us the map on $E_\infty$-pages. We may now fix the splitting of \cref{eqn: real codim 2 splitting} so that the map is zero on the $\bb{Z}$-summand.

\textbf{Codimension one complex bundles} $\vect{\bb{C},\ell-1}^0(\bb{CP}^\ell)$ were computed by Hu~\cite[Theorem 1.1]{Hu}.

\textbf{Realification} $\vect{\bb{C}, \ell-1}^0(\bb{CP}^\ell) \to \vect{\bb{R}, 2\ell-2}^0(\bb{CP}^\ell)$. To identify the map we need only consider the case where $\ell$ is odd, since for $\ell$ even, the domain is zero. For $\ell$ odd, Hu gives an $E_2$-page in the complex case \cite[p. 7795]{Hu}, and we computed the real one above. The map then takes the form:
    
\begin{adjustbox}{max width = \textwidth}
\begin{tikzcd}[ampersand replacement=\&,sep=tiny]
	\&\& {} \&\&\&\&\&\&\&\& {} \\
	\& \C \&\& {2\ell -2} \& {2\ell -1} \& {2\ell } \&\&\&\& \R \&\& {2\ell -2} \& {2\ell -1} \& {2\ell } \\
	{} \&\&\&\&\&\& {} \&\& {} \&\&\&\&\&\& {} \\
	\& {-(2\ell-2)} \&\& 0 \& 0 \& 0 \&\&\&\& {-(2\ell-2)} \&\& {\bb{Z}} \& 0 \& {\bb{Z}} \\
	\& {-(2\ell-1)} \&\& {\bb{Z}} \& 0 \& {\bb{Z}} \&\&\&\& {-(2\ell-1)} \&\& {(\bb{Z}/2)^2} \& 0 \& {(\bb{Z}/2)^2} \\
	\& {-2\ell} \&\& {\bb{Z}/2} \& 0 \& {\bb{Z}/2} \&\&\&\& {-2\ell} \&\& {(\bb{Z}/2)^2} \& 0 \& {(\bb{Z}/2)^2} \\
	\&\& {} \&\&\&\&\&\&\&\& {}
	\arrow[no head, from=1-3, to=7-3]
	\arrow[no head, from=1-11, to=7-11]
	\arrow[no head, from=3-1, to=3-7]
	\arrow[no head, from=3-9, to=3-15]
	\arrow[curve={height=-18pt}, dashed, from=4-4, to=4-12]
	\arrow[curve={height=18pt}, dashed, from=6-6, to=6-14]
\end{tikzcd}
\end{adjustbox}
where again dotted arrows indicate potentially non-zero components of the map affecting the main diagonal. The first of these is zero, because its domain is zero, and we computed the second (\cref{thm: sphere diagrams 2}, i.e.~\cref{fig:even sphere diagram 2}, first non-trivial realification) in terms of our basis. It is the inclusion of the first factor
\[
\pi_{2\ell}^s(\Sigma\bb{CP}^\infty_{\ell-1}) \cong \bb{Z}/2 \hookrightarrow (\bb{Z}/2)^2 \cong \pi_{2\ell-2}^s(\bb{RP}^\infty_{2\ell-2}).
\]
This tells us the map on the associated graded. Reassembling the target, we find that it follows for filtration reasons, or alternatively because a map from a torsion group to a free group must be zero, that whatever our choice of splitting in \cref{eqn: real codim 2 splitting} that the component of the map landing in the $\bb{Z}$-summand is zero.

\textbf{Codimension three real bundles} $\vect{\bb{R}, 2\ell-3}^0(\bb{CP}^\ell)$ (\cref{fig:E2AHSS}, $i=3$). For $\ell$ odd, the $E_2$-page looks like

\begin{adjustbox}{max width = \textwidth}
\begin{tikzcd}[ampersand replacement=\&,cramped,sep=tiny]
	\&\& {} \&\&\&\&\&\&\&\& \\
	\& {E_2} \&\& {2\ell-6} \& {2\ell -5} \& {2\ell -4} \& {2\ell -3} \& {2\ell -2} \& {2\ell -1} \& {2\ell } \\
	{} \&\&\&\&\&\&\&\&\&\& {} \\
	\& {\geq -(2\ell-4)} \&\& 0 \& 0 \& 0 \& 0 \& 0 \& 0 \& 0 \\
	\& {-(2\ell-3)} \&\& {\bb{Z}/2} \& 0 \& {\bb{Z}/2} \& 0 \& {\bb{Z}/2} \& 0 \& {\bb{Z}/2} \\
	\& {-(2\ell-2)} \&\& 0 \& 0 \& 0 \& 0 \& 0 \& 0 \& 0 \\
	\& {-(2\ell-1)} \&\& {\bb{Z}/2} \& 0 \& {\bb{Z}/2} \& 0 \& {\bb{Z}/2} \& 0 \& {\bb{Z}/2} \\
	\& {-2\ell} \&\& \begin{array}{c} \begin{cases} \bb{Z}/2 & \text{if $\ell \equiv 1(4)$} \\ (\bb{Z}/2)^2 & \text{if $\ell \equiv 3(4)$} \end{cases} \end{array} \& 0 \& \begin{array}{c} \begin{cases} \bb{Z}/2 & \text{if $\ell \equiv 1(4)$} \\ (\bb{Z}/2)^2 & \text{if $\ell \equiv 3(4)$} \end{cases} \end{array} \& 0 \& \begin{array}{c} \begin{cases} \bb{Z}/2 & \text{if $\ell \equiv 1(4)$} \\ (\bb{Z}/2)^2 & \text{if $\ell \equiv 3(4)$} \end{cases} \end{array} \& 0 \& \begin{array}{c} \begin{cases} \bb{Z}/2 & \text{if $\ell \equiv 1(4)$} \\ (\bb{Z}/2)^2 & \text{if $\ell \equiv 3(4)$} \end{cases} \end{array} \\
	\&\& {} \& {}
	\arrow[no head, from=1-3, to=9-3]
	\arrow[no head, from=3-1, to=3-11]
	\arrow["{d_2}"{description}, from=7-8, to=8-10]
\end{tikzcd}
\end{adjustbox}

where the single possible non-zero $d_2$ differential affecting the main diagonal is necessarily zero by~\cref{rmk: differentials}. As always, there can be be no $d_3$'s, but there is a possible $d_4$, shown below, where we write $\ast$ for irrelevant entries.
\[\begin{tikzcd}[ampersand replacement=\&,cramped,sep=tiny]
	\&\& {} \&\&\&\&\&\&\&\& \\
	\& {E_4} \&\& {2\ell-6} \& {2\ell -5} \& {2\ell -4} \& {2\ell -3} \& {2\ell -2} \& {2\ell -1} \& {2\ell } \\
	{} \&\&\&\&\&\&\&\&\&\& {} \\
	\& {\geq -(2\ell-4)} \&\& 0 \& 0 \& 0 \& 0 \& 0 \& 0 \& 0 \\
	\& {-(2\ell-3)} \&\& \ast \& 0 \& {\bb{Z}/2} \& 0 \& \ast \& 0 \& \ast \\
	\& {-(2\ell-2)} \&\& 0 \& 0 \& 0 \& 0 \& 0 \& 0 \& 0 \\
	\& {-(2\ell-1)} \&\& \ast \& 0 \& \ast \& 0 \& \ast \& 0 \& \ast \\
	\& {-2\ell} \&\& \ast \& 0 \& \ast \& 0 \& \ast \& 0 \& \begin{array}{c} \begin{cases} \bb{Z}/2 & \text{if $\ell \equiv 1(4)$} \\ (\bb{Z}/2)^2 & \text{if $\ell \equiv 3(4)$} \end{cases} \end{array} \\
	\&\& {} \& {}
	\arrow[no head, from=1-3, to=9-3]
	\arrow[no head, from=3-1, to=3-11]
	\arrow["{d_4}"{description}, from=5-6, to=8-10]
\end{tikzcd}\]
By~\cref{rmk: differentials}, this differential is governed by the cell structure of $\bb{CP}^\ell_{\ell-2}$, so is given by $\lambda \nu$, where $\lambda$ is as in \cref{lem:Mosher} with $n = \ell - 2$. Recalling that $\ell$ is odd, we read off that if $\ell \equiv 3 (8)$ then $\lambda = 2$, if $\ell \equiv 1,5 (8)$ then $\lambda = 1$, and if $\ell \equiv 7 (8)$ then $\lambda = 0$. Since the target of the proposed differential is annihilated by $\cdot 2$, the differential is zero unless $\ell \equiv 1,5 (8)$. We now need to examine the effect of $\nu^*$ on $\pi_*(\bb{RP}_{2 \ell - 3}^\infty)$, and $\ell \equiv 1,5 (8)$ implies $2 \ell - 3 \equiv 7(8)$, so we must look at \cref{fig:RPASS4}-right. Here the $h_2$ multiplication is zero, but it is still possible that we have a hidden $\nu$-multiplication connecting the dots in columns $8k+7$ and $8k+10$. One can resolve this problem as follows. The truncation map $\bb{RP}^\infty_{8k+6} \to \bb{RP}^\infty_{8k+7}$ is a surjection on $\pi_{8k+7}$ (either by \cref{fig:odd sphere diagram 2}, or by reading directly off the charts in \cref{fig:RPASS4}). For $\bb{RP}^\infty_{8k+6}$, the map $\nu^* : \pi_{8k+7}(\bb{RP}^\infty_{8k+6}) \to \pi_{8k+10}(\bb{RP}^\infty_{8k+6})$ is trivial, because the codomain is zero. It follows that $\nu^* : \pi_{8k+7}(\bb{RP}^\infty_{8k+7}) \to \pi_{8k+10}(\bb{RP}^\infty_{8k+7})$ , hence our $d_4$-differential, must also be trivial. The result follows.

For $\ell$ even the $E_2$-page has the form
\[\begin{tikzcd}[ampersand replacement=\&,cramped,sep=tiny]
	\&\& {} \&\&\&\&\&\&\&\& \\
	\& {E_2} \&\& {2\ell-6} \& {2\ell -5} \& {2\ell -4} \& {2\ell -3} \& {2\ell -2} \& {2\ell -1} \& {2\ell } \\
	{} \&\&\&\&\&\&\&\&\&\& {} \\
	\& {\geq-(2\ell-4)} \&\& 0 \& 0 \& 0 \& 0 \& 0 \& 0 \& 0 \\
	\& {-(2\ell-3)} \&\& {\bb{Z}/2} \& 0 \& {\bb{Z}/2} \& 0 \& {\bb{Z}/2} \& 0 \& {\bb{Z}/2} \\
	\& {-(2\ell-2)} \&\& {\bb{Z}/2} \& 0 \& {\bb{Z}/2} \& 0 \& {\bb{Z}/2} \& 0 \& {\bb{Z}/2} \\
	\& {-(2\ell-1)} \&\& {\bb{Z}/8} \& 0 \& {\bb{Z}/8} \& 0 \& {\bb{Z}/8} \& 0 \& {\bb{Z}/8} \\
	\& {-2\ell} \&\& {\bb{Z}/2} \& 0 \& {\bb{Z}/2} \& 0 \& {\bb{Z}/2} \& 0 \& {\bb{Z}/2} \\
	\&\& {} \& {}
	\arrow[no head, from=1-3, to=9-3]
	\arrow[no head, from=3-1, to=3-11]
	\arrow["{d_2}"{description}, from=6-8, to=7-10]
	\arrow["{d_2}"{description}, from=7-8, to=8-10]
\end{tikzcd}\]
where we have marked the only potentially non-zero differentials we need to worry about, the others being zero by~\cref{rmk: differentials}.  The differential $\bb{Z}/2 \to \bb{Z}/8$ is given by $\eta$, and is always injective. This can be seen from the Adams chart \cref{fig:RPASS3}-right if $\ell \equiv 0 (4)$ and the Adams chart \cref{fig:RPASS1}-right if $\ell \equiv 2 (4)$. These charts also tell us that the other differential, $\bb{Z}/8 \to \bb{Z}/2$, is surjective if $\ell \equiv 0 (4)$, and is zero if $\ell \equiv 2 (4)$. It follows that for $\ell \equiv 0 (4)$, the diagonal of the $E_3$-page is zero, so the target is too.

For $\ell \equiv 2 \pmod{4}$ there is a potential $d_4$ differential 
\[\begin{tikzcd}[ampersand replacement=\&,cramped,sep=tiny]
	\&\& {} \&\&\&\&\&\&\&\& \\
	\& {E_4} \&\& {2\ell-6} \& {2\ell -5} \& {2\ell -4} \& {2\ell -3} \& {2\ell -2} \& {2\ell -1} \& {2\ell } \\
	{} \&\&\&\&\&\&\&\&\&\& {} \\
	\& {\geq -(2\ell-4)} \&\& 0 \& 0 \& 0 \& 0 \& 0 \& 0 \& 0 \\
	\& {-(2\ell-3)} \&\& \ast \& 0 \& {\bb{Z}/2} \& 0 \& \ast \& 0 \& \ast \\
	\& {-(2\ell-2)} \&\& \ast \& 0 \& \ast \& 0 \& 0 \& 0 \& \ast \\
	\& {-(2\ell-1)} \&\& \ast \& 0 \& \ast \& 0 \& \ast \& 0 \& \ast \\
	\& {-2\ell} \&\& \ast \& 0 \& \ast \& 0 \& \ast \& 0 \& {\bb{Z}/2} \\
	\&\& {} \& {}
	\arrow[no head, from=1-3, to=9-3]
	\arrow[no head, from=3-1, to=3-11]
	\arrow["{d_4}"{description}, from=5-6, to=8-10]
\end{tikzcd}\]
which by~\cref{rmk: differentials} is given by $ \lambda \nu$ for $\lambda$ equal to 0 or 2, hence is zero, being a self-map of $\bb{Z}/2$. It follows that the $\bb{Z}/2$ on the diagonal survives the spectral sequence, hence the desired computation.

\textbf{Stabilisation} $\vect{\bb{R}, 2\ell-3}^0(\bb{CP}^\ell) \to \vect{\bb{R}, 2\ell-2}^0(\bb{CP}^\ell)$. If $\ell$ is even, the map is zero because the domain is torsion and the codomain is $\bb{Z}$. For $\ell$ odd, our spectral sequence calculations reduce the question to determining the map $\pi_{2 \ell}(\bb{RP}^\infty_{2 \ell-3}) \to \pi_{2 \ell}(\bb{RP}^\infty_{2 \ell-2})$, which was done in \cref{fig:even sphere diagram 2}.

\textbf{Codimension four real bundles} $\vect{\bb{R}, 2\ell-4}^0(\bb{CP}^\ell)$ (\cref{fig:E2AHSS}, $i=4$). We divide into four cases.

\textbf{Codimension $4$ real bundles for $\ell \equiv 0 \pmod{4}$.} The $E_2$-page (\cref{fig:E2AHSS}) is given by
 \[\noindent\resizebox{\linewidth}{!}{%
\begin{tikzcd}[ampersand replacement=\&,cramped,sep=tiny]
	\&\& {} \&\&\&\&\&\&\&\& \\
	\& {E_2} \&\& {2\ell-6} \& {2\ell -5} \& {2\ell -4} \& {2\ell -3} \& {2\ell -2} \& {2\ell -1} \& {2\ell } \\
	{} \&\&\&\&\&\&\&\&\&\& {} \\
	\& {\geq -(2\ell-5)} \&\& 0 \& 0 \& 0 \& 0 \& 0 \& 0 \& 0 \\
	\& {-(2\ell-4)} \&\& {\bb{Z}} \& 0 \& {\bb{Z}} \& 0 \& {\bb{Z}} \& 0 \& {\bb{Z}} \\
	\& {-(2\ell-3)} \&\& {(\bb{Z}/2)^2} \& 0 \& {(\bb{Z}/2)^2} \& 0 \& {(\bb{Z}/2)^2} \& 0 \& {(\bb{Z}/2)^2} \\
	\& {-(2\ell-2)} \&\& {(\bb{Z}/2)^2} \& 0 \& {(\bb{Z}/2)^2} \& 0 \& {(\bb{Z}/2)^2} \& 0 \& {(\bb{Z}/2)^2} \\
	\& {-(2\ell-1)} \&\& {\bb{Z}/16 \oplus \bb{Z}/4} \& 0 \& {\bb{Z}/16 \oplus \bb{Z}/4} \& 0 \& {\bb{Z}/16 \oplus \bb{Z}/4} \& 0 \& {\bb{Z}/16 \oplus \bb{Z}/4} \\
	\& {-2\ell} \&\& {\bb{Z}/2} \& 0 \& {\bb{Z}/2} \& 0 \& {\bb{Z}/2} \& 0 \& {\bb{Z}/2} \\
	\&\& {} \& {}
	\arrow[no head, from=1-3, to=10-3]
	\arrow[no head, from=3-1, to=3-11]
	\arrow["{d_2}"{description}, from=5-4, to=6-6]
	\arrow["{d_2}"{description}, from=7-8, to=8-10]
	\arrow["{d_2}"{description}, from=8-8, to=9-10]
\end{tikzcd}%
}\]
Potentially non-zero $d_2$-differentials affecting the diagonal $p+q=0$ are shown. Those from the $(2\ell-4)$-th to the to $(2\ell-2)$-nd column are zero, while those connecting the other two columns are given by $\eta^*$  (using \cref{rmk: differentials}). We have $2 \ell - 4 \equiv 4 (8)$, so we consult \cref{fig:RPASS3}-left, where we see that the top left $d_2$ is projection onto the second factor, the bottom-right $d_2$ is surjective, and the remaining $d_2$ (immediately above it) is injective. As usual, there can be no $d_3$ differentials, so the $E_4$-page is given by
\[\begin{tikzcd}[ampersand replacement=\&,cramped,sep=tiny]
	\&\& {} \&\&\&\&\&\&\&\& \\
	\& {E_4} \&\& {2\ell-6} \& {2\ell -5} \& {2\ell -4} \& {2\ell -3} \& {2\ell -2} \& {2\ell -1} \& {2\ell } \\
	{} \&\&\&\&\&\&\&\&\&\& {} \\
	\& {\geq -(2\ell-5)} \&\& 0 \& 0 \& 0 \& 0 \& 0 \& 0 \& 0 \\
	\& {-(2\ell-4)} \&\& \ast \& 0 \& {\bb{Z}} \& 0 \& \ast \& 0 \& \ast \\
	\& {-(2\ell-3)} \&\& \ast \& 0 \& \ast \& 0 \& \ast \& 0 \& \ast \\
	\& {-(2\ell-2)} \&\& \ast \& 0 \& \ast \& 0 \& 0 \& 0 \& \ast \\
	\& {-(2\ell-1)} \&\& \ast \& 0 \& \ast \& 0 \& \ast \& 0 \& {\bb{Z}/8 \oplus \bb{Z}/2} \\
	\& {-2\ell} \&\& \ast \& 0 \& \ast \& 0 \& \ast \& 0 \& 0 \\
	\&\& {}
	\arrow[no head, from=1-3, to=10-3]
	\arrow[no head, from=3-1, to=3-11]
	\arrow["{d_4}"{description}, from=5-6, to=8-10]
\end{tikzcd}\]
with only one potential non-zero differential affecting the diagonal, namely $d_4: \bb{Z} \to \bb{Z}/8 \oplus \bb{Z}/2$, which is given by $\nu^*$ since $\ell \equiv 0 (4)$ (\cref{rmk: differentials} c.f.~\cref{lem:Mosher}). Consulting \cref{fig:RPASS3}-left again, the $h_2$ multiplications coming out of the first $h_0$-tower tell us that this $d_4$ sends the generator of the domain, $1 \in \bb{Z}$, to a class of order 8. The kernel is $\bb{Z} \cong 8 \bb{Z} \subset \bb{Z}$, and the result follows.

\textbf{Codimension $4$ real bundles for $\ell \equiv 1 \pmod{4}$.}  The $E_2$-page (\cref{fig:E2AHSS}) is
\[\begin{tikzcd}[ampersand replacement=\&,cramped,sep=tiny]
	\&\& {} \\
	\& {E_2} \&\& {2\ell-6} \& {2\ell -5} \& {2\ell -4} \& {2\ell -3} \& {2\ell -2} \& {2\ell -1} \& {2\ell } \\
	{} \&\&\&\&\&\&\&\&\&\& {} \\
	\& {\geq -(2\ell-5)} \&\& 0 \& 0 \& 0 \& 0 \& 0 \& 0 \& 0 \\
	\& {-(2\ell-4)} \&\& {\bb{Z}} \& 0 \& {\bb{Z}} \& 0 \& {\bb{Z}} \& 0 \& {\bb{Z}} \\
	\& {-(2\ell-3)} \&\& {\bb{Z}/4} \& 0 \& {\bb{Z}/4} \& 0 \& {\bb{Z}/4} \& 0 \& {\bb{Z}/4} \\
	\& {-(2\ell-2)} \&\& 0 \& 0 \& 0 \& 0 \& 0 \& 0 \& 0 \\
	\& {-(2\ell-1)} \&\& {\bb{Z}/4} \& 0 \& {\bb{Z}/4} \& 0 \& {\bb{Z}/4} \& 0 \& {\bb{Z}/4} \\
	\& {-2\ell} \&\& 0 \& 0 \& 0 \& 0 \& 0 \& 0 \& 0 \\
	\&\& {} \& {}
	\arrow[no head, from=1-3, to=10-3]
	\arrow[no head, from=3-1, to=3-11]
	\arrow["\eta"{description}, from=5-6, to=6-8]
\end{tikzcd}\]
The only possible relevant $d_2$-differential $\bb{Z}\to \bb{Z}/4$ is $\eta^*$ by~\cref{rmk: differentials}. We consult \cref{fig:RPASS4}-left, since $2 \ell - 4 \equiv 6 (8)$, and find that $\eta^*$ acts by $\bb{Z} \xrightarrow{\cdot 2} \bb{Z}/4$. The $E_4$-page is
\[\begin{tikzcd}[ampersand replacement=\&,cramped,sep=tiny]
	\&\& {} \\
	\& {E_4} \&\& {2\ell-6} \& {2\ell -5} \& {2\ell -4} \& {2\ell -3} \& {2\ell -2} \& {2\ell -1} \& {2\ell } \\
	{} \&\&\&\&\&\&\&\&\&\& {} \\
	\& {\geq -(2\ell-5)} \&\& 0 \& 0 \& 0 \& 0 \& 0 \& 0 \& 0 \\
	\& {-(2\ell-4)} \&\& \ast \& 0 \& {2\bb{Z}} \& 0 \& \ast \& 0 \& \ast \\
	\& {-(2\ell-3)} \&\& \ast \& 0 \& \ast \& 0 \& {\bb{Z}/2} \& 0 \& \ast \\
	\& {-(2\ell-2)} \&\& 0 \& 0 \& 0 \& 0 \& 0 \& 0 \& 0 \\
	\& {-(2\ell-1)} \&\& \ast \& 0 \& \ast \& 0 \& \ast \& 0 \& {\bb{Z}/4} \\
	\& {-2\ell} \&\& 0 \& 0 \& 0 \& 0 \& 0 \& 0 \& 0 \\
	\&\& {} \& {}
	\arrow[no head, from=1-3, to=10-3]
	\arrow[no head, from=3-1, to=3-11]
	\arrow["\nu"{description}, from=5-6, to=8-10]
\end{tikzcd}\]
The only relevant $d_4$-differential is given by $\nu^*$ (\cref{rmk: differentials} c.f.~\cref{lem:Mosher}). Consulting the $h_2$ multiplications in \cref{fig:RPASS4}-left, we see that this map acts on homotopy groups by
$$\pi_{2 \ell - 4}(\bb{RP}^\infty_{2 \ell - 4}) \cong \bb{Z} \longtwoheadrightarrow \bb{Z}/4 \cong \pi_{2 \ell - 1}(\bb{RP}^\infty_{2 \ell - 4})$$
but in passing to $E_4$-pages, we passed to the subgroup $2 \bb{Z} \subset \bb{Z}$ in the domain, hence the $d_4$-differential has image given by the multiples of 2 inside $\bb{Z}/4$. The spectral sequence therefore collapses at the $E_5$-page, with remaining group the kernel, $4 \bb{Z} \cong \bb{Z}$.

\textbf{Codimension $4$ real bundles for $\ell \equiv 2 \pmod{4}$.} The $E_2$-page (\cref{fig:E2AHSS}) is
\[\begin{tikzcd}[ampersand replacement=\&,cramped,sep=tiny]
	\&\& {} \&\&\&\&\&\&\&\& \\
	\& {E_2} \&\& {2\ell-6} \& {2\ell -5} \& {2\ell -4} \& {2\ell -3} \& {2\ell -2} \& {2\ell -1} \& {2\ell } \\
	{} \&\&\&\&\&\&\&\&\&\& {} \\
	\& {\geq -(2\ell-5)} \&\& 0 \& 0 \& 0 \& 0 \& 0 \& 0 \& 0 \\
	\& {-(2\ell-4)} \&\& {\bb{Z}} \& 0 \& {\bb{Z}} \& 0 \& {\bb{Z}} \& 0 \& {\bb{Z}} \\
	\& {-(2\ell-3)} \&\& {(\bb{Z}/2)^2} \& 0 \& {(\bb{Z}/2)^2} \& 0 \& {(\bb{Z}/2)^2} \& 0 \& {(\bb{Z}/2)^2} \\
	\& {-(2\ell-2)} \&\& {(\bb{Z}/2)^2} \& 0 \& {(\bb{Z}/2)^2} \& 0 \& {(\bb{Z}/2)^2} \& 0 \& {(\bb{Z}/2)^2} \\
	\& {-(2\ell-1)} \&\& {(\bb{Z}/8)^2} \& 0 \& {(\bb{Z}/8)^2} \& 0 \& {(\bb{Z}/8)^2} \& 0 \& {(\bb{Z}/8)^2} \\
	\& {-2\ell} \&\& {\bb{Z}/2} \& 0 \& {\bb{Z}/2} \& 0 \& {\bb{Z}/2} \& 0 \& {\bb{Z}/2} \\
	\&\& {} \& {}
	\arrow[no head, from=1-3, to=10-3]
	\arrow[no head, from=3-1, to=3-11]
	\arrow["{d_2}"{description}, from=5-4, to=6-6]
	\arrow["{d_2}"{description}, from=7-8, to=8-10]
	\arrow["{d_2}"{description}, from=8-8, to=9-10]
\end{tikzcd}\]
where the relevant potentially non-zero differentials are deduced, as normal, by~\cref{rmk: differentials}. They are all given by $\eta^*$. We have $2 \ell - 4 \equiv 0 (8)$, so we consult \cref{fig:RPASS1}-left, where we see that (top-to-bottom), the three differentials are respectively: projection onto the second factor, injective, and zero. The $E_4$-page is
\[\begin{tikzcd}[ampersand replacement=\&,cramped,sep=tiny]
	\&\& {} \&\&\&\&\&\&\&\& \\
	\& {E_4} \&\& {2\ell-6} \& {2\ell -5} \& {2\ell -4} \& {2\ell -3} \& {2\ell -2} \& {2\ell -1} \& {2\ell } \\
	{} \&\&\&\&\&\&\&\&\&\& {} \\
	\& {\geq -(2\ell-5)} \&\& 0 \& 0 \& 0 \& 0 \& 0 \& 0 \& 0 \\
	\& {-(2\ell-4)} \&\& \ast \& 0 \& {\bb{Z}} \& 0 \& \ast \& 0 \& \ast \\
	\& {-(2\ell-3)} \&\& \ast \& 0 \& {\bb{Z}/2} \& 0 \& \ast \& 0 \& \ast \\
	\& {-(2\ell-2)} \&\& \ast \& 0 \& \ast \& 0 \& 0 \& 0 \& \ast \\
	\& {-(2\ell-1)} \&\& \ast \& 0 \& \ast \& 0 \& \ast \& 0 \& {(\bb{Z}/4)^2} \\
	\& {-2\ell} \&\& \ast \& 0 \& \ast \& 0 \& \ast \& 0 \& {\bb{Z}/2} \\
	\&\& {} \& {}
	\arrow[no head, from=1-3, to=10-3]
	\arrow[no head, from=3-1, to=3-11]
	\arrow["{d_4}"{description}, from=5-6, to=8-10]
	\arrow["{d_4}"{description}, from=6-6, to=9-10]
\end{tikzcd}\]
    where the relevant $d_4$ differentials (\cref{rmk: differentials}, c.f.~\cref{lem:Mosher}) are zero if $\ell \equiv 2 (8)$, and are given by $2 \nu^*$ if $\ell \equiv 6 (8)$. Since it is always 2-divisible, the lower differential is always zero. When $\ell \equiv 6 (8)$, the upper one might be non-zero: consulting the Adams chart (\cref{fig:RPASS1}-left) shows that it sends the generator $1 \in \bb{Z}$ to a class of order 4 (noting that this $(\bb{Z}/4)^2$ arose in the passage from $E_2$ as the quotient of $\pi_{2 \ell - 1}(\bb{RP}^\infty_{2 \ell - 4}) \cong (\bb{Z}/8)^2$ by the subgroup of elements annihilated by $\cdot 2$). It follows that the kernel is $4\bb{Z} \cong \bb{Z}$. As for the codimension 2 case (\cref{eqn: real codim 2 splitting}), we may now choose a splitting and conclude that
    $$\vect{\bb{R}, 2\ell-4}^0(\bb{CP}^\ell) \cong \bb{Z} \oplus \bb{Z}/2$$
    (though in light of the above the interpretation of this group is different when $\ell \equiv 2 (8)$ and when $\ell \equiv 6 (8)$).

\textbf{Codimension $4$ real bundles for $\ell \equiv 3 \pmod{4}$.} The $E_2$-page is
\[\begin{tikzcd}[ampersand replacement=\&,cramped,sep=tiny]
	\&\& {} \&\&\&\&\&\&\&\& \\
	\& {E_2} \&\& {2\ell-6} \& {2\ell -5} \& {2\ell -4} \& {2\ell -3} \& {2\ell -2} \& {2\ell -1} \& {2\ell } \\
	{} \&\&\&\&\&\&\&\&\&\& {} \\
	\& {\geq -(2\ell-5)} \&\& 0 \& 0 \& 0 \& 0 \& 0 \& 0 \& 0 \\
	\& {-(2\ell-4)} \&\& {\bb{Z}} \& 0 \& {\bb{Z}} \& 0 \& {\bb{Z}} \& 0 \& {\bb{Z}} \\
	\& {-(2\ell-3)} \&\& {\bb{Z}/4} \& 0 \& {\bb{Z}/4} \& 0 \& {\bb{Z}/4} \& 0 \& {\bb{Z}/4} \\
	\& {-(2\ell-2)} \&\& 0 \& 0 \& 0 \& 0 \& 0 \& 0 \& 0 \\
	\& {-(2\ell-1)} \&\& {\bb{Z}/4} \& 0 \& {\bb{Z}/4} \& 0 \& {\bb{Z}/4} \& 0 \& {\bb{Z}/4} \\
	\& {-2\ell} \&\& {\bb{Z}/2} \& 0 \& {\bb{Z}/2} \& 0 \& {\bb{Z}/2} \& 0 \& {\bb{Z}/2} \\
	\&\& {} \& {}
	\arrow[no head, from=1-3, to=10-3]
	\arrow[no head, from=3-1, to=3-11]
	\arrow["{0}"{description}, from=5-4, to=6-6]
	\arrow["\eta^*"{description}, from=5-6, to=6-8]
	\arrow["{0}"{description}, from=8-8, to=9-10]
\end{tikzcd}\]
 in which the relevant $d_2$ differentials are as shown by~\cref{rmk: differentials}. Consulting \cref{fig:RPASS2}-left, since $2\ell-4 \equiv 2 (8)$, the middle differential has image of order 2. The $E_4$-page is then
\[\begin{tikzcd}[ampersand replacement=\&,cramped,sep=tiny]
	\&\& {} \&\&\&\&\&\&\&\& \\
	\& {E_4} \&\& {2\ell-6} \& {2\ell -5} \& {2\ell -4} \& {2\ell -3} \& {2\ell -2} \& {2\ell -1} \& {2\ell } \\
	{} \&\&\&\&\&\&\&\&\&\& {} \\
	\& {\geq -(2\ell-5)} \&\& 0 \& 0 \& 0 \& 0 \& 0 \& 0 \& 0 \\
	\& {-(2\ell-4)} \&\& \ast \& 0 \& {2\bb{Z}} \& 0 \& \ast \& 0 \& \ast \\
	\& {-(2\ell-3)} \&\& \ast \& 0 \& {\bb{Z}/4} \& 0 \& \ast \& 0 \& \ast \\
	\& {-(2\ell-2)} \&\& 0 \& 0 \& 0 \& 0 \& 0 \& 0 \& 0 \\
	\& {-(2\ell-1)} \&\& \ast \& 0 \& \ast \& 0 \& \ast \& 0 \& {\bb{Z}/4} \\
	\& {-2\ell} \&\& \ast \& 0 \& \ast \& 0 \& \ast \& 0 \& {\bb{Z}/2} \\
	\&\& {} \& {}
	\arrow[no head, from=1-3, to=10-3]
	\arrow[no head, from=3-1, to=3-11]
	\arrow["{d_4}"{description}, from=5-6, to=8-10]
	\arrow["{d_4}"{description}, from=6-6, to=9-10]
\end{tikzcd}\]
in which the two relevant $d_4$-differentials are given by $2\nu^*$ if $\ell \equiv 3 (8)$ and are zero if $\ell \equiv 7 (8)$. Arguing as in the previous case, we deduce that the differentials are always zero, and obtain the result.

\textbf{Stabilisation} $\vect{\bb{R}, 2\ell-4}^0(\bb{CP}^\ell) \to \vect{\bb{R}, 2\ell-3}^0(\bb{CP}^\ell)$. For $\ell \equiv 0 \pmod{4}$, the map is necessarily zero on account of the codomain being zero. 

For $\ell \equiv 1 \pmod{4}$, we have a map $\bb{Z} \to \bb{Z}/2$. There are only two such maps, the zero map, and the other one. We will argue that it is the other one.

Our map is induced by pushforward along the (stable) truncation map $$\bb{Z} \cong \vect{\bb{R}, 2\ell-4}^0(\bb{CP}^\ell) \cong [\bb{CP}^\ell,\bb{RP}^\infty_{2 \ell - 4}] \xrightarrow{\ s_*\ } [\bb{CP}^\ell,\bb{RP}^\infty_{2 \ell - 3}] \cong \vect{\bb{R}, 2\ell-3}^0(\bb{CP}^\ell) \cong \bb{Z}/2.$$
Since the codomain of both mapping sets are $(2 \ell - 5)$-connected, and $\bb{CP}^{\ell - 3}$ is $(2 \ell - 6)$-dimensional, we can use the cofibre sequence $\bb{CP}^{\ell - 3} \to \bb{CP}^\ell \to \bb{CP}^\ell_{\ell -2}$ to see that our map is equivalent to 
$$\bb{Z} \cong [\bb{CP}^\ell_{\ell - 2},\bb{RP}^\infty_{2 \ell - 4}] \xrightarrow{\ s_*\ } [\bb{CP}^\ell_{\ell -2},\bb{RP}^\infty_{2 \ell - 3}] \cong \bb{Z}/2,$$
which we will now study using the cell structure of $\bb{CP}^\ell_{\ell - 2}$. Since $\ell \equiv 1 (4)$, an Adams chart for $\bb{CP}^\ell_{\ell - 2}$ is given (up to a shift) by \cref{fig:SCPASS2}-right. This Adams chart shows that the inclusion of the bottom cell in $\bb{CP}^\ell_{\ell - 2}$ is null after composition with $\eta$ or with $\nu$. We know that $\bb{CP}^\ell_{\ell - 2}$ is composed of three cells, and the only way to put three cells together to make $\eta$ and $\nu$ null on the bottom cell is if we actually have
$$\bb{CP}^\ell_{\ell - 2} \simeq \cofibre(S^{2 \ell - 3} \oplus S^{2 \ell - 1} \xrightarrow{(\eta, \nu)} S^{2 \ell - 4}).$$

Since we only want to show that $s_*: [\bb{CP}^\ell_{\ell - 2},\bb{RP}^\infty_{2 \ell - 4}] \to [\bb{CP}^\ell_{\ell -2},\bb{RP}^\infty_{2 \ell - 3}]$ is non-zero, it suffices to exhibit a map $f: \bb{CP}^\ell_{\ell - 2} \to \bb{RP}^\infty_{2 \ell - 4}$ such that $sf$ is not null. Since $\fibre(\bb{RP}^\infty_{2 \ell - 4} \xrightarrow{s} \bb{RP}^\infty_{2 \ell - 3}) \simeq S^{2 \ell - 4}$, we may equivalently show that $f$ does not factor over the inclusion of the bottom cell in $\bb{RP}^\infty_{2 \ell - 4}$.

To construct such an $f$, consider the diagram
\[\begin{tikzcd}
    S^{2 \ell - 3} \oplus S^{2 \ell -1}  & S^{2 \ell - 4} & \bb{CP}^\ell_{\ell - 2} \\
    & S^{2 \ell - 4} & \bb{RP}^{\infty}_{2 \ell - 4}, 
    \arrow["{(\eta,\nu)}",from=1-1, to=1-2]
    \arrow["\cdot 4",from=1-2, to=2-2]
    \arrow["f",dashed,from=1-3, to=2-3]
    \arrow["i",hook,from=1-2, to=1-3]
    \arrow["j",hook,from=2-2, to=2-3]
\end{tikzcd}\]
where the top row is the cofibre sequence defining $\bb{CP}^\ell_{\ell - 2}$ that we obtained earlier, we write $i$ and $j$ for the respective bottom cell inclusions and $\cdot 4$ for the degree 4 map. Since $\ell \equiv 1 (4)$, we have $2 \ell - 4 \equiv 6 (8)$, so (e.g.~by \cref{fig:RPASS4}-left), the homotopy groups $\pi_{2\ell - 1}(\bb{RP}^{\infty}_{2 \ell - 4})$ and $\pi_{2\ell - 3}(\bb{RP}^{\infty}_{2 \ell - 4})$ are both isomorphic to $\bb{Z}/4$. It follows that the composite $j \circ (\cdot 4) \circ (\eta,\nu)$ is null (because it is in particular a 4-multiple with domain $S^{2 \ell - 3} \oplus S^{2 \ell -1}$). Since the top row is a cofibration, we get an extension $f$ making the diagram commute. This $f$ cannot factor over the inclusion $j$ of the bottom cell, because if it did then by a diagram chase we would have $4 \nu \simeq *$, and we know that $\nu$ has order 8. This completes the argument for $\ell \equiv 1 \pmod{4}$.

For $\ell \equiv 2 \pmod{4}$, examining the spectral sequences shows that the map must be a isomorphism after restriction to torsion subgroups. It follows that without loss of generality we may change the splitting we chose to identify the domain as $\bb{Z} \oplus \bb{Z}/2$ so that the map is zero on the first summand.

For $\ell \equiv 3 \pmod{4}$ we must consider a map $\bb{Z} \oplus \bb{Z}/2 \to (\bb{Z}/2)^2$, which, as in the case $\ell \equiv 1 (4)$, we may identify with the map
$$[\bb{CP}^\ell_{\ell-2},\bb{RP}^{\infty}_{2 \ell - 4}] \xrightarrow{s_*} [\bb{CP}^\ell_{\ell-2},\bb{RP}^{\infty}_{2 \ell - 3}].$$

Write $C_{\eta}$ for the cofibre of $\eta$, and $j$ for the inclusion of its bottom cell. Note first that since $\ell \equiv 3(4)$, we know that we may describe the cell structure of $\bb{CP}^\ell_{\ell-2}$ as
$$\bb{CP}^\ell_{\ell-2} \simeq \cofibre(S^{2 \ell - 1} \xrightarrow{\lambda j \nu} \Sigma^{2 \ell - 4} C_{\eta}) \simeq \cofibre(S^{2 \ell - 1} \oplus S^{2 \ell - 3} \xrightarrow{(\lambda \nu, \eta)} S^{2 \ell - 4}),$$ where by \cref{lem:Mosher}, $\lambda$ is determined mod 4 (since $\nu$ becomes 4-torsion in $C_{\eta}$ c.f.~\cref{fig:mod-2-Moore and hook}-right), and is respectively $2$ or $0$ when $\ell$ is respectively $3$ or $7$ modulo 8.

Considering the long exact sequences obtained by applying the functors $[(-),\bb{RP}^\infty_{2\ell-4}]$ and $[(-),\bb{RP}^\infty_{2\ell-3}]$ to the cofibre sequence $$S^{2 \ell - 3} \xrightarrow{\eta} S^{2 \ell - 4} \to \Sigma^{2 \ell - 4}C_{\eta}$$ (and reading the effect of $\eta^*$ off the two charts in \cref{fig:RPASS2}) tells us that $$[C_{\eta}, \bb{RP}^\infty_{2\ell-3}] = 0, \textrm{ and } [C_{\eta}, \bb{RP}^\infty_{2\ell-4}] = \bb{Z}\{\widetilde{2i}\},$$
where $\widetilde{2i}$ is determined by the property that its restriction to the bottom cell is twice the generator $i$ of $\pi_{2 \ell - 4}(\bb{RP}^\infty_{2\ell-4})$.

When $\lambda = 0$ ($\ell \equiv 7(8)$), the top cell of $\bb{CP}^\ell_{\ell-2}$ splits off, so we may write $$[\bb{CP}^\ell_{\ell-2},(-)] \cong [\Sigma^{2 \ell - 4} C_\eta,(-)] \oplus \pi_{2\ell}(-),$$
and deduce from the above analysis that $s_*$ is zero on the infinite order summand, and by \cref{thm: sphere diagrams 2} c.f.~\cref{fig:even sphere diagram 2} is injective on the torsion summand, as required.

When $\lambda = 2$ ($\ell \equiv 3(8)$), things get messy, so we make a simplification by reducing from $\bb{CP}^\ell_{\ell-2}$ (which has three cells) to $\Sigma^{2 \ell - 4} C_{2 \nu}$ (which has two cells). To do this, note that the above cofibre description gives us a pushout
\[\begin{tikzcd}
    S^{2 \ell - 4} & \Sigma^{2 \ell - 4} C_{2 \nu} \\
    \Sigma^{2 \ell - 4} C_{\eta} & \bb{CP}^{\ell}_{\ell -2},
    \arrow[from=1-1, to=1-2]
    \arrow[from=1-1, to=2-1]
    \arrow[from=1-2, to=2-2]
    \arrow[from=2-1, to=2-2]
\end{tikzcd}\]
and hence a cofibre sequence
$$S^{2 \ell - 4} \to \Sigma^{2 \ell - 4} C_{\eta} \oplus \Sigma^{2 \ell - 4} C_{2 \nu}  \to \bb{CP}^{\ell}_{\ell - 2} \to S^{2 \ell - 3} \to \Sigma^{2 \ell - 3} C_{\eta} \oplus \Sigma^{2 \ell - 3} C_{2 \nu}$$

Applying either $[(-),\bb{RP}_{2 \ell - 4}^{\infty}]$ or $[(-),\bb{RP}_{2 \ell - 3}^{\infty}]$ to this pushout gives (contravariantly) a Mayer--Vietoris long exact sequence, and $s:\bb{RP}_{2 \ell - 4}^{\infty} \to \bb{RP}_{2 \ell - 3}^{\infty}$ gives a map between these long exact sequences, which looks as follows (where summands preserve the order of the summands in the cofibre sequence above).
\[\begin{tikzcd}
    \bb{Z}\{i\} & \bb{Z} \oplus (\bb{Z} \oplus \bb{Z}/2) & {[\bb{CP}^\ell_{\ell - 2},\bb{RP}^\infty_{2 \ell - 4}]} & \pi_{2 \ell - 3}(\bb{RP}^\infty_{2 \ell - 4}) & \cdots \\
    0 & 0 \oplus (\bb{Z}/2)^2 & {[\bb{CP}^\ell_{\ell - 2},\bb{RP}^\infty_{2 \ell - 3}]} & \pi_{2 \ell - 3}(\bb{RP}^\infty_{2 \ell - 3}) & \cdots
    \arrow["\alpha",swap,from=1-2, to=1-1]
    \arrow[from=1-3, to=1-2]
    \arrow[from=1-4, to=1-3]
    \arrow["\beta",swap,from=1-5, to=1-4]
    \arrow[from=2-2, to=2-1]
    \arrow[from=2-3, to=2-2]
    \arrow[from=2-4, to=2-3]
    \arrow["\gamma",swap,from=2-5, to=2-4]
    \arrow["s_*",from=1-1, to=2-1]
    \arrow["s_*",from=1-2, to=2-2]
    \arrow["s_*",from=1-3, to=2-3]
    \arrow["s_*",from=1-4, to=2-4]
\end{tikzcd}\]
We leave it to the reader to verify that the groups are as claimed, using the long exact sequences obtained from the cofibrations building the 2-cell complexes $\Sigma^{2 \ell - 4} C_{\eta}$ and $\Sigma^{2 \ell - 4} C_{2 \nu}$, and this calculation will also show that $\alpha$ sends the generator of each $\bb{Z}$-summand to twice a generator of the target. Note in particular that the projection $p : \Sigma^{2 \ell - 4}C_{2 \nu} \to S^{2 \ell}$ to the top cell is an isomorphism on $\pi_{2 \ell}$. Next, $\beta$ and $\gamma$ are both surjections, because for $X = \bb{RP}^{\infty}_{2 \ell - 4},\bb{RP}^{\infty}_{2 \ell - 3}$, any map $S^{2 \ell - 3} \xrightarrow{f} X$ has $f \eta = 0$, and hence extends over $\Sigma^{2 \ell - 3} C_{\eta}$ (since $\pi_{2 \ell - 2}(X) = 0$). It follows that for $X = \bb{RP}^{\infty}_{2 \ell - 4},\bb{RP}^{\infty}_{2 \ell - 3}$, the maps
\[
[\Sigma^{2 \ell - 4}C_{2 \nu},X] \longleftarrow [\bb{CP}^{\ell}_{\ell-2},X]
\]
induced by the inclusion $\Sigma^{2 \ell - 4}C_{2 \nu} \hookrightarrow \bb{CP}^{\ell}_{\ell-2}$ are isomorphisms, and we may equivalently evaluate the map 
$$[\Sigma^{2 \ell - 4}C_{2 \nu},\bb{RP}^\infty_{2 \ell - 4}] \xrightarrow{s_*} [\Sigma^{2 \ell - 4}C_{2 \nu},\bb{RP}^\infty_{2 \ell - 3}].$$
This completes our restriction from a 3-cell complex to a 2-cell complex. We will now argue that this map $s_*$ is surjective.

By the cofibration $\bb{RP}^\infty_{2 \ell - 4} \xrightarrow{s} \bb{RP}^{\infty}_{2 \ell - 3} \xrightarrow{\partial} S^{2 \ell - 3}$, we may equivalently argue that the connecting map $$\partial_* : [\Sigma^{2 \ell - 4}C_{2 \nu},\bb{RP}^\infty_{2 \ell - 3}] \to [\Sigma^{2 \ell - 4}C_{2 \nu},S^{2 \ell - 3}] $$ is zero. Writing $p : \Sigma^{2 \ell - 4}C_{2 \nu} \to S^{2 \ell}$ for the projection to the top cell (we noted above that this map is a $\pi_{2 \ell}$-isomorphism) we obtain a diagram
\[\begin{tikzcd}
    & {[\Sigma^{2 \ell - 4}C_{2 \nu},\bb{RP}^\infty_{2 \ell - 3}]} & {[\Sigma^{2 \ell - 4}C_{2 \nu},S^{2 \ell - 3}]} \\
     & \pi_{2 \ell}(\bb{RP}^\infty_{2 \ell - 3}) & \pi_{2 \ell}(S^{2 \ell - 3}) \\ 
     & (\bb{Z}/2)^2 & \bb{Z}/8\{\nu\}.
    \arrow["\partial_*",from=2-2, to=2-3]
    \arrow["\partial_*",from=1-2, to=1-3]
    \arrow["p^*",from=2-2, to=1-2]
    \arrow["\cong",swap,from=2-2, to=1-2]
    \arrow["p^*",from=2-3, to=1-3]
    \arrow[equal,from=2-3, to=3-3]
    \arrow[equal,from=2-2, to=3-2]
\end{tikzcd}\]
The bottom map $\partial_*$ can only hit elements which are annihilated by multiplication by 2. The map $p$ fits into a cofibration $\Sigma^{2 \ell - 4}C_{2 \nu} \xrightarrow{p} S^{2 \ell} \xrightarrow{2 \nu} S^{2 \ell - 3}$, and so the right hand column of the diagram may be extended to an exact sequence witnessing that $p^*(2 \nu) = 0$, so in particular the right hand $p^*$ is zero on elements annihilated by 2. It follows that the top $\partial_*$ is zero, hence that $s_*$ is surjective, and hence, by changing basis if necessary, that the map is as claimed.

\textbf{Codimension two complex bundles} $\vect{\bb{C},\ell-2}^0(\bb{CP}^\ell)$. The groups have been computed in~\cite[Theorem 1.2]{Hu}, but we only extract the $2$-local information.

\textbf{Realification} $\vect{\bb{C}, \ell-2}^0(\bb{CP}^\ell) \to \vect{\bb{R}, 2\ell-4}^0(\bb{CP}^\ell)$. For $\ell \equiv 0 ,1 \pmod{4}$, the map is zero because the domain is zero. For $\ell \equiv 2 ,3 \pmod{4}$, the domain is torsion-free, so the image of this map is contained in the torsion subgroup. This is important: although we potentially changed basis in the target of this map when we computed the stabilisation map $\vect{\bb{R}, 2\ell-4}^0(\bb{CP}^\ell) \to \vect{\bb{R}, 2\ell-3}^0(\bb{CP}^\ell)$, we only changed our choice of the infinite order generator. Such a change does not affect the description in terms of the basis of a map whose image is contained in the torsion subgroup, so we do not need to worry about propagating that change to this step, and it suffices to calculate the map on torsion subgroups. This map is a map from a cyclic group of even order to $\bb{Z}/2$. Such a map is either surjective or zero. To see which, note that it follows from our spectral sequence calculations that this map on torsion subgroups is identified with the map on $\pi_{2 \ell}$, and we have seen (\cref{fig:even sphere diagram 2}, leftmost column) that this latter map is always surjective.

\textbf{Stabilisation} $\vect{\bb{C}, \ell-2}^0(\bb{CP}^\ell) \to \vect{\bb{C}, \ell - 1}^0(\bb{CP}^\ell)$ is the only map in the leftmost square of \cref{fig:CPell diagram} that we have not yet computed, and it follows by a diagram chase that this map must be zero.
\end{proof}
\printbibliography
\end{document}